\documentclass[10pt,reqno]{amsart}
\usepackage{amscd}
\usepackage{amsmath}
\usepackage{amsmath,amsthm,amssymb,amsfonts}
\usepackage{mathrsfs}
\usepackage{graphicx}
\usepackage{hyperref}
\usepackage{fancyhdr}
\usepackage{stmaryrd}
\usepackage[colorinlistoftodos,prependcaption,textsize=tiny]{todonotes}
\usepackage{color}
\usepackage{amsmath}
\usepackage{amsfonts}
\usepackage{amssymb}
\usepackage{geometry}
\usepackage{mathtools}
\usepackage{chngcntr}
\usepackage{amssymb}
\usepackage{pdfsync}
\usepackage{epstopdf}
\usepackage{setspace}

\newcommand{\be}{\begin{equation}}
\newcommand{\ee}{\end{equation}}

\newtheorem{thm}{Theorem}[section]

\newtheorem{lem}[thm]{Lemma}
\newtheorem{prop}[thm]{Proposition}
\theoremstyle{definition}
\newtheorem{defn}[thm]{Definition}
\newtheorem{rem}[thm]{Remark}
\numberwithin{equation}{section}

\newcommand{\eps}{\varepsilon}

\newcommand{\R}{\mathbb{R}}

\newcommand{\abs}[1]{\left\vert#1\right\vert}

\allowdisplaybreaks
\begin{document}
\title[Spherical interface to Cahn-Hilliard flow]{Solutions with single radial interface of the generalized Cahn-Hilliard flow}
\author{Chao Liu}
\address{School of Mathematics and Information Science, Guangzhou University,
Guangzhou 510006, People's Republic of China}
\email{chaoliuwh@whu.edu.cn}

\author{Jun Yang}
\address{School of Mathematics and Information Science, Guangzhou University,
Guangzhou 510006, People's Republic of China}
\email{jyang2019@gzhu.edu.cn}


\begin{abstract}
We consider the generalized parabolic Cahn-Hilliard equation
$$
u_t=-\Delta\left[\Delta u -W'(u)\right]+W''(u)\left[\Delta u -W'(u)\right]
\qquad \forall\,   (t,  x)\in \widetilde{{\mathbb R}}\times{\mathbb R}^n,
$$
where $n=2$ or $n\geq 4$,  $W(\cdot)$ is the typical double-well potential function and $\widetilde{\mathbb R}$ is given by
$$
\widetilde{\mathbb R}=\left\{
                    \begin{array}{rl}
                      (0,  \infty),  &\quad \mbox{if } n=2,
\\
                      (-\infty,  0),  & \quad\mbox{if } n\geq 4.
                    \end{array}
                  \right.
$$
We construct a radial solution $u(t, x)$ possessing an interface.
At main order this solution consists of a traveling copy of the steady
state $\omega(|x|)$,  which satisfies $\omega''(y)-W'(\omega(y))=0$.
Its interface is resemble at main order copy of the sphere of the following form
$$
|x|=\sqrt[4]{-2(n-3)(n-1)^2t}, \qquad \forall\,   (t,  x)\in \widetilde{{\mathbb R}}\times{\mathbb R}^n,
$$
which is a solution to the Willmore flow in Differential Geometry.
When $n=1$ or $3$,  the result consists trivial solutions.
\end{abstract}


\thanks{The work is supported by NSFC(No. 12171109) and Basic Research Project of Guangzhou(No. 202201020094).}



\date{}

\dedicatory{}

\commby{}

\maketitle
\tableofcontents

\section{Introduction}

We are interested in the construction of  solutions of the following generalized parabolic  Cahn-Hilliard equation
\begin{equation}\label{eq1.1}
  u_t=-\Delta\left[\Delta u -W'(u)\right]+W''(u)\left[\Delta u -W'(u)\right],  \qquad  \forall\,  (t,  x)\in  \widetilde{\mathbb R}\times \R^n,
\end{equation}
where $n=2$ or $n\geq4$,
$\widetilde{\mathbb R}$ is given by
$$
\widetilde{\mathbb R}=\left\{
                    \begin{array}{rl}
                      (0,  \infty),  &\quad \mbox{if } n=2,
\\[2mm]
                      (-\infty,  0),  & \quad \mbox{if } n\geq 4,
                    \end{array}
                  \right.
$$
and the functions $W'(s)$ and $W''(s)$ denote the derivatives of first and second orders of $W$ respectively.
The potential $W(s)$ is a smooth function,  which satisfies
the following assumptions
\begin{equation}\label{eq1.4}
  \left\{
\begin{array}{l}
 W(s)>W(-1)=W(1)\qquad \text{in}\ (-1, 1), \\[1mm]W(s)=W(-s),\qquad \text{for all}\ s\in\R, \\[1mm]
 W'(-1)=W'(1)=0, \\[1mm] W''(-1)=W''(1)>0.
  \end{array}
\right.
\end{equation}
The potential $W(u)$ has two non-degenerate local minimum points $u=+1$ and $u=-1$,  which are stable equilibria of \eqref{eq1.1}.
In particular,  the function $W(s)=\frac{1}{4}(1-s^2)^2$ obviously satisfies the above conditions in (\ref{eq1.4}).
Up to a scaling,  the function $W(s)=\cos(s)$  also satisfies \eqref{eq1.4}.

\subsection{Backgrounds}
Generally speaking,  the Cahn-Hilliard equation means the following equation
\begin{equation}\label{eq1.2}
  \left\{
\begin{array}{l}
u_t=-\Delta\left[\Delta u-W'(u)\right], \qquad \forall\,  (t, x)\in (-\infty, +\infty)\times\Omega,
\\[2mm]
u(0, x)=u_0(x), \hspace{2cm} \forall\,  x\in\Omega,
  \end{array}
\right.
\end{equation}
where $\Omega$ denotes a smooth bounded domain in $\R^n$ or the whole space $\R^n$($n\geq1$),  which describes phase separation processes of binary alloys in \cite{ch}.
Various kinds of problems of this class equation have been extensively studied in recent thirty years.
By applying a priori estimate and continuity argument,  the existence and asymptotic behaviors  of  global smooth solutions of Cauchy problem \eqref{eq1.2} have been proved in \cite{lwz} when the initial value $u_0$ is close to stable equilibria $\bar{u}$ ($W(\bar{u})=0$) in the $L^\infty\cap L^1(\R^n)$ space.
Using Fourier transform  and estimates on the kernel of a linear parabolic operator,   a uniform  $L^\infty$ bound estimate for solutions of perturbed Cauchy problem \eqref{eq1.2} with additional nonlinear terms,  was established by Caffarelli and Muler in \cite{cm}.
In a bounded domain,  the global well-posedness and long-time behavior of solutions to problem \eqref{eq1.2} with several types of dynamic boundary conditions were studied in \cite{ cgw, cgm,  gms, lw,  rz}.

Based on the works of De Giorgi in \cite{d, de},   numerous authors \cite{bm, rs, t} studied the diffuse approximation of the Willmore functional:
 \begin{equation}\label{wf}
   \mathcal{W}(S, \Omega)=\frac{1}{2}\int_{\partial S\cap\Omega}\abs{H_{\partial S}(x)}^2 {\mathrm d}\mu^{n-1},
 \end{equation}
 where $\Omega$ is a given open set in $\R^n$,  the set $S\subset\R^n$ with smooth boundary $\partial S\subset \Omega$,  and $H_{\partial S}(x)$ is the mean curvature of surface $\partial S$ at point $x\in \partial S$. The approximating functional is defined by
 \begin{equation}
   \mathbb{W}_\varepsilon(u)=\left\{
\begin{array}{l}
\frac{1}{2\varepsilon}\int_{\Omega}\Big(\varepsilon\Delta u -\frac{W'(u)}{\varepsilon}\Big)^2{\mathrm d}x,\qquad \text{if}\ u\in L^1(\Omega)\cap W^{2, 2}(\Omega),
\\[2mm]
+\infty,\hspace{3.9cm} \text{otherwise in}\ L^1(\Omega),
  \end{array}
\right.
\label{CHEnergy}
 \end{equation}
 where the function $W$ satisfies \eqref{eq1.4}. The essential and more challenging work is to  rigourously prove that the approximating functional
 $\mathbb{W}_\varepsilon(u)$  $\Gamma$-converges to the Willmore functional $ \mathcal{W}(S, \Omega)$ as $\varepsilon$ goes to $0$. Bellettini
  and Paolini in \cite{bp} proved the $\Gamma$-lim sup inequality for smooth Willmore hypersurfaces. However,  the $\Gamma$-lim inf inequality is
 more hard to prove. Up to now,  it has been proved in $\R^n$ with $n=2, 3$ in \cite{rs} or $n=2$ in \cite{nt}. This problem is still open in $\R^n$ with $n\geq4$.
The relation of the critical points of \eqref{wf} and \eqref{CHEnergy} was exemplified by M. Rizzi \cite{r} and also A. Malchiodi, R. Mandel, M. Rizzi \cite{MalchiodiMandelRizzi}.
On the other hand,  for the parabolic Cahn-Hilliard equation \eqref{eq1.2},  the gamma convergence results  have obtained by Le in \cite{l} under suitable conditions. And see \cite{bb} for the case of \eqref{eq1.2} on the one-dimensional torus.

The $L^2$ gradient flow of the approximating energy $\mathbb{W}_\varepsilon (u)$  is equivalent to the evolution equation:
\begin{equation}\label{eq1}
  \partial_tu_\varepsilon=-\Delta\Big[\Delta u_\varepsilon -\frac{1}{\varepsilon^2}W'(u_\varepsilon)\Big]
+\frac{1}{\varepsilon^2}W''(u_\varepsilon)\Big[\Delta u_\varepsilon -\frac{1}{\varepsilon^2}W'(u_\varepsilon)\Big]
\qquad  \text{in}\ (-\infty, +\infty)\times \Omega,
\end{equation}
which was introduced in \cite{dlw}  to describe the deformation of a vesicle membrane under
the elastic bending energy,  with prescribed bulk volume and surface area. The well-posedness of the phase field model \eqref{eq1} with fixed $\varepsilon$ has been proved in \cite{cl} providing a volume constraint for the average of $u$,  or in \cite{cl1} with both volume and area constraints. By applying formal method of matched asymptotic expansions,  Loreti and March in \cite{lm} (or Wang in \cite{w}) showed that if $\Gamma(t)\subset \R^n$ with $n=2$ or $3$,  is a family of compact closed smooth interfaces and evolves by Willmore flow,  it can be approximated by nodal set of the solution $u_\varepsilon$ to the phase field (\ref{eq1}) when $\varepsilon$ goes to $0$. The Willmore flow equation is given by
 \begin{equation}\label{willmore}
   V(t)=\Delta_{\Sigma(t)}H-\frac{1}{2}H^3+H\|A\|^2,
 \end{equation}
 which is the $L^2$ gradient flow for \eqref{wf} with $\partial S(t)=\Sigma(t)$,  where $V(t)$ denotes the outer normal velocity at $x \in \Sigma(t)$,  $\Delta_{\Sigma(t)}$ is the Laplace-Beltrami operator on surface $\Sigma(t)$,  $H$ and $A$ are the mean curvature and the second fundamental form of $\Sigma(t)$ respectively,  and $\|A\|^2$ is the sum of squared coefficients of $A$. Fei and Liu \cite{fl} proved that for given a solution $\Gamma_0(t)$ of \eqref{willmore} in $\R^n(n=2, 3)$,  there exists a solution $u_\varepsilon$ of equation \eqref{eq1} with Neumann boundary condition such that its level set convergence to $\Gamma_0(t)$ as the parameter $\varepsilon$ goes to zero.
 Moreover,  a variety of problems for the Willmore flow \eqref{willmore} has been  investigated by Kuwert and Sch\"{a}tzle in several papers (for example see \cite{ks, ks1, ks2}). However,  the study of connections between equation \eqref{eq1} and the Willmore flow \eqref{willmore} is a  challenging work when $n\geq4$. The present paper is one of the first attempts in this direction.

\subsection{Main results}
In this paper,  we want to find solutions of (\ref{eq1.1}) whose values lie at all times in $[-1, 1]$,
and approach either $+1$ or $-1$ in the most of the space $\R^n$.
This type of solution
corresponds to a continuous realization of a material,  in which the two states
($u=- 1$ and $u=+1$) coexist.  The main difficult point of the study of this type solution
of (\ref{eq1.1}),  is to derive qualitative information on the interface region(the walls
separating the two phases).
It is easy to find that $u(t, x)$ is a solution of \eqref{eq1.1} if and only if $u_\eps(t, x):=u(\eps^{-4}t, \eps^{-1} x)$ satisfies equation \eqref{eq1}. Basing on the results in \cite{lm} or \cite{w} with $n=2$ and $3$,   we know that the nodal set of $u_\varepsilon(t, x)$ approximates to the solution of Willmore flow \eqref{willmore}. Let us
consider the sphere $\Gamma_n(t)$ evolving by the Willmore flow in \eqref{willmore}.
Then we have that the radius $\gamma_n(t)$ of $\Gamma_n(t)$ satisfies the equation
\begin{equation}\label{willmoreflowradial}
  \gamma_n'(t)=-\frac{1}{2}\Big(\frac{n-1}{\gamma_n(t)}\Big)^3+\frac{\big(n-1\big)^2}{\big(\gamma_n(t)\big)^3},
\end{equation}
which has a solution
\begin{equation}\label{sphere eq}
\gamma_n(t):=\sqrt[4]{-2(n-3)(n-1)^2t},
\end{equation}
where $t\leq0$ when $n\geq 3$ and $t\geq0$ when $n=2$. Due to the self-similarity,  there holds that the sphere $\abs{x}=\gamma_n(t)$  is also the transition layer (nodal set) for $u(t, x)$,  which is a solution of \eqref{eq1.1}.

Our aim is to construct  solutions to equation \eqref{eq1.1} with one transition layer closes to the sphere $\abs{x}=\gamma_n(t)$. In which,  $t\leq0$ when $n\geq 4$,  these solutions are called ancient solutions. And $t\geq0$ for $n=2$,  they are long time solutions.
We shall mention that the ancient radially symmetric solutions for the parabolic Allen-Cahn equation
\begin{equation}\label{Allen}
  u_t=\Delta u+u-u^3 \qquad \text{in} \ (-\infty, 0]\times\R^n,
\end{equation}
have been obtained by del Pino and Gkikas in \cite{dG1} for $n=1$ and \cite{dG2} with $n\geq2$.

Let us introduce a layer function by considering the following problem of semilinear elliptic equation
  \begin{equation}\label{eqq}
    \omega''(y)-W'\big(\omega(y)\big)=0, \quad \omega'(y)>0,  \quad y\in\R, \quad
 \omega(0)=0\quad \text{and} \quad \lim\limits_{y\rightarrow\pm\infty}\omega(y)=\pm 1,
  \end{equation}
 which has  a unique smooth solution $\omega(y)$ obtained in \cite{ac},  where $W$ satisfies the conditions in (\ref{eq1.4}).
In particular if we choose $W'(s)=s^3-s$,  then (\ref{eqq}) is the elliptic Allen-Cahn equation,  and its solution is written as
$$
\omega(y)=\tanh(\frac{y}{\sqrt{2}}).
$$
 For general functions $W$ satisfying \eqref{eq1.4},  the solution $\omega$ of
 (\ref{eqq}) has no explicit expression.  In \cite{ac},  the inverse function of $\omega$ is given by
 \begin{equation}\label{eqq2}
 \lambda(s):=\int^s_{0}\frac{1}{\sqrt{2(W(\tau)-W(1))}}\mathrm{d}\tau, \qquad s\in(-1, 1).
\end{equation}
Thanks to (\ref{eq1.4}),  the function $\lambda(s)$ is well-defined. Taking $\omega$ as a basic layer
(the name layer is motivated by the fact that $\omega$ approaches the limits $1$ and $-1$ at $\pm \infty$),  we will construct a
solution of (\ref{eq1.1}) with a `transition layer' which is symmetric about the sphere $\abs{x}=\gamma_n(t)$ in \eqref{sphere eq}.

More precisely,   we want to find a solution of equation \eqref{eq1.1},  which has the following asymptotical behavior
\begin{equation}\label{eq1.12}
u(t, x)\approx \omega\big(\abs{x}-\rho(t)\big),  
\end{equation}
where $\omega$ is given by (\ref{eqq}).
The function $\rho(t)$ satisfies
\begin{equation}\label{eq1.13}
 \rho(t)=\gamma_n(t)+h(t), \qquad \text{with}\ h(t)=O\left(\frac{1}{\log\abs{t}}\right),
\end{equation}
where $h(t)$ is an $C^1$ function with respect of $t$ and
the function $\gamma_n(t)$ is defined by \eqref{sphere eq}.
In fact,  $\rho(t)$ can be chosen by solving the following ODE
\begin{equation*}
 \rho'(t)+\frac{(n-3)(n-1)^2}{2\rho^3(t)}=Q\big(\rho(t), \rho'(t)\big), \qquad \text{with}\ Q\big(\rho(t), \rho'(t)\big)=O\left(\frac{1}{\abs{t}\big[\log\abs{t}\big]^{2(p-1)}}\right),
\end{equation*}
for all $t\leq0$ when $n\geq4$ or $t\geq0$ when $n=2$,  where $p\in(n, n+1]$,  see Section \ref{sec:tc}.

 Our main results can be stated as follows.
\begin{thm}\label{thm1}
When $n=2$,  there exists a radial solution $u(t, \abs{x})$ of equation $(\ref{eq1.1})$ with $t\geq0$,  which has  the form $u(t, x)\approx \omega\big(\abs{x}-\rho(t)\big)+\phi(t, \abs{x})$,  with
\begin{equation}\label{1.4}
\rho(t)=\sqrt[4]{2t}+h(t).
\end{equation}
Moreover,  $h(t)$ is a $C^1$ function with the decay of order $O(\frac{1}{\log\abs{t}})$ as $t\rightarrow+\infty$ and $\lim\limits_{t\rightarrow +\infty}\phi(t, \abs{x})=0$ uniformly in $x\in \R^2$.
\qed
\end{thm}

\begin{thm}\label{thm4}
When $n\geq4$,  there exists an ancient solution $u(t, x)$ of equation with $t\leq0$,
which has the form $u(t, x)\approx \omega\big(\abs{x}-\rho(t)\big)+\phi(t, \abs{x})$ with
\begin{equation}\label{1.7}
\rho(t)=\sqrt[4]{-2(n-3)(n-1)^2t}+h(t).
\end{equation}
Moreover,  $h(t)$ is a $C^1$ function with the decay of order $O(\frac{1}{\log\abs{t}})$ as $t\rightarrow-\infty$  and $\lim\limits_{t\rightarrow -\infty}\phi(t, \abs{x})=0$ uniformly in $x\in \R^n$.
\qed
\end{thm}

\begin{rem}
In this paper,  we mainly prove that the results of Theorem \ref{thm4} hold.
For Theorem \ref{thm1},  its proof is similar and we just now notice that $t>0$ in this case $n=2$.
When $n=1$ and $3$,  it is easily to check that problem \eqref{eq1.1} has a radical solution of the form
 \begin{equation*}
   u(t, \abs{x})=\omega(\abs{x}-c),
 \end{equation*}
 with any constant $c\in\R$,  where $\omega(y)$ is the solution of \eqref{eqq}.
\qed

\end{rem}

\begin{rem}
We prove Theorem \ref{thm4} by using Lyapunov-Schmidt reduction method.
This method has been widely applied to solving the existence problem  of solutions to various kinds of equations \cite{cdm, CDD, dd1, dd2, dG2}.
Comparing to the previous works,  we are facing
three difficulties during the process of dealing with problem \eqref{eq1.1} as the following:

\noindent {\textbf{(1)}}.
Firstly,  the main term $\omega\big(\abs{x}-\rho(t)\big)$ in \eqref{eq1.13}
is not a good approximate solution since it does not satisfy the boundary conditions in \eqref{eq2.2} and does not provide enough decay of the error.
To modify it,  we introduce a cut-off function around the origin and a correction function,
see \eqref{dz} in Section \ref{sec:a}.
This main idea comes from \cite{r}.

\noindent {\textbf{(2)}}.
Secondly,  equation \eqref{eq1.1} is a parabolic equation of fourth order,  which does not satisfy the Maximum Principle since the heat kernel of the biharmonic parabolic operator is sign-changing,  see \cite{FG}.
To overcome it,  we employ the blow-up technique together with the representation of parabolic kernel in Section \ref{sec:lp}.

\noindent {\textbf{(3)}}.
Lastly, the correlative heat kernel of the linearization operator of the equation \eqref{eq1.1} is more intricate,  which leads to hard task to get suitable a priori estimates of
linearized problem of equation \eqref{eq1.1}.
To solve it,  we modify the estimate of the error term $E(t, r)$ in \eqref{Error} with an polynomial decay in Section \ref{sec:a},  and perform more delicate calculations in Section \ref{sec:lp}.
\qed
\end{rem}

The paper is organized as follows.

$\clubsuit$ In the first part of Section \ref{sec:a},  we deduce some estimates of decay for $\omega$ in (\ref{eqq}) and its derivatives.
After that,  an approximate solution,  say $z(t,\abs{x})$ with a parameter $\rho(t)$ in \eqref{dz} and \eqref{drho1},   will be defined.
By the perturbation of $z(t,\abs{x})+\phi(t,\abs{x})$,  the setting-up of a projected form of \eqref{eq1.1} will be derived,
see (\ref{eq2.10})-(\ref{eq2.11}) with the Lagrange multiplier $c(t)$.
With the introduction of a suitable norm,  the estimates of the error will be provided in the last part of
Section \ref{sec:a}.

$\clubsuit$ Section \ref{sec:lp} is devoting to the collection of some results of linear parabolic equations with a biharmonic operator and then obtain the solvability of a linear projected problem in (\ref{eq3.1}).

$\clubsuit$
In Section \ref{sec:nl},  we solve the nonlinear problem (\ref{eq2.10})-(\ref{eq2.11}) by an argument of  the fixed-point theorem.

$\clubsuit$
In Section \ref{sec:tc}, in order to obtain a radial solution to \eqref{eq1.1}
we choose a suitable parameter $h(t)$ (in other words,  adjusting the parameter $\rho$ given by \eqref{drho1})
such that  $c(t)$ equals to zero,  in problem (\ref{eq2.10})-\eqref{eq2.11}.

\section{The setting-up: ansatz, the nonlinear projected problem for perturbation term} \label{sec:a}

\subsection{Some estimates of the basic layer}
Before proving the main theorems,  we first derive the decay estimates of the basic layer $\omega$,  which is the solution to equation (\ref{eqq}).
\begin{lem}\label{lem1}Let $\omega$ be the solution of problem \eqref{eqq},  then we have
\begin{equation}\label{eq2.0}
\begin{aligned}
 \left\{
\begin{array}{l}
 \lim\limits_{y\rightarrow+\infty}\frac{\omega(y)-1+\beta e^{-\sqrt{W''(1)}y}}{\beta^2e^{-2\sqrt{W''(1)}y}}=\frac{W^{(3)}(1)}{6W''(1)};
\\[3mm]
\lim\limits_{y\rightarrow-\infty}\frac{\omega(y)+1-\beta e^{\sqrt{W''(-1)}y}}{\beta^2e^{2\sqrt{W''(-1)}y}}=-\frac{W^{(3)}(-1)}{6W''(-1)};
\\[3mm]
\lim\limits_{y\rightarrow+\infty}\frac{\omega'(y)}{e^{-\sqrt{W''(1)}y}}=\beta\sqrt{W''(1)}
\quad \text{and}\quad
\lim\limits_{y\rightarrow-\infty}\frac{\omega'(y)}{e^{\sqrt{W''(-1)}y}}=\beta\sqrt{W''(-1)},
  \end{array}
\right.
\end{aligned}
\end{equation}
where $W''(s)$ and $W^{(3)}(s)$ denote the derivatives of second and third orders of the function $W(s)$ respectively,
and
\begin{equation}\label{beta}
\beta:=\exp\Bigg\{ \sqrt{W''(1)}\int^1_0
\Bigg[\frac{1}{\sqrt{2\big(W(s)-W(1)\big)}}-\frac{1}{\sqrt{W''(1)}\,(1-s)}\Bigg]\mathrm{d}s\Bigg\}.
\end{equation}
\end{lem}

\begin{proof}
Since the function $W$ satisfies the conditions in $(\ref{eq1.4})$,
especially  $W''(1)=W'(-1)>0$,
the above limits in $(\ref{eq2.0})$ are well-defined and $\beta$ is finite.
Recall that the inverse function of $\omega$ is the function $\lambda(s)$ given by (\ref{eqq2}). Using L'Hospital's rule,
Taylor's formula,  the function $W\in C^{3, 1}_{\text{loc}}(\R)$ satisfies (\ref{eq1.4}) and (\ref{eqq2}),  we find that
  $$\lim_{s\rightarrow 1^-}\frac{1}{1-s}
  \Bigg\{ \lambda(s)+\frac{\ln(1-s)}{\sqrt{W''(1)}}-\int^1_0\frac{1}
  {\sqrt{2(W(s)-W(1))}}-\frac{1}{\sqrt{W''(1)}(1-s)}\mathrm{d}s\Bigg\}
  =-\frac{W^{(3)}(1)}{6(W''(1))^{3/2}}
  $$
  and
  $$ \lim_{s\rightarrow (-1)^+}\frac{1}{1+s}
  \Bigg\{ \lambda(s)- \frac{\ln(1+s)}{\sqrt{W''(-1)}}-\int^0_{-1}\frac{1}
  {\sqrt{2(W(s)-W(1))}}-\frac{1}{\sqrt{W''(1)}(1+s)}\mathrm{d}s\Bigg\}
  =\frac{W^{(3)}(-1)}{6(W''(-1))^{3/2}}.
  $$
According to the fact that if $x=\lambda(s)$,  then $s=\omega(y)$,  using Taylor's formula again and the evenness of $W$ in $(-1, 1)$,  we easily deduce that
the first two limits in (\ref{eq2.1}) hold.
So we have
$$ \lim\limits_{y\rightarrow+\infty}\frac{1-\omega(y)}{e^{-\sqrt{W''(1)}y}}=\beta\ \ \text{and}\ \ \lim\limits_{y\rightarrow-\infty}\frac{1+\omega(y)}{e^{\sqrt{-W''(1)}y}}=\beta.$$
Using L'Hospital's rule again,  the condition (\ref{eq1.4}) and the equation (\ref{eqq}),  we have
\begin{equation*}
\begin{aligned}
\lim\limits_{y\rightarrow+\infty}\frac{\omega'(y)}{e^{-\sqrt{W''(1)}y}}
=
\lim\limits_{y\rightarrow+\infty}\Bigg\{\frac{\omega(y)-1}{-\sqrt{W''(1)}e^{-\sqrt{W''(1)}y}}
\times\frac{W'(\omega)-W'(1)}{\omega(y)-1}\Bigg\}
=\beta\sqrt{W''(1)}
\end{aligned}
\end{equation*}
and
$$\lim\limits_{y\rightarrow-\infty}\frac{\omega'(y)}{e^{\sqrt{W''(-1)}y}}
=
\lim\limits_{y\rightarrow-\infty}\Bigg\{\frac{\omega(y)+1}{\sqrt{W''(-1)}e^{\sqrt{W''(-1)}y}}
\times\frac{W'(\omega)-W'(-1)}{\omega(y)+1}\Bigg\}
=\beta\sqrt{W''(-1)}.$$
\end{proof}

Next we consider the kernel of a fourth  order linear operator. The main result is stated as the following.

\begin{lem} Let $W(s)$ be a smooth function satisfying the conditions in \eqref{eq1.4}.
Then we have that any solution of the homogeneous problem
\begin{equation*}
 \big[\,\partial_{xx}-W'(\omega(x))\,\big]^2\varphi=0,  \quad \abs{\varphi}\leq 1,  \quad \abs{\varphi_{xx}}\leq 1 \qquad \text{in}\ \R,
\end{equation*}
has the form
\begin{equation*}
  \varphi(x)=c\omega'(x),
\end{equation*}
with some constant $c\in\R$.
\end{lem}

For the proof of this lemma,  see Lemma \ref{lem2} which is a more general result.
\qed

\subsection{The setting-up of the problem}\label{section2.2}
We will prove that Theorem \ref{thm4} holds.
Thus,  we always assume that $n\geq 4$ and $t<0$ in the rest of the present paper.

Let $\widetilde{u}(t, \abs{x})$ be a solution of \eqref{eq1.1},  by a translation $u(t, \abs{x})=\widetilde{u}(t-T, \abs{x})$ with some abuse of notation,   then we have that $u(t, r)$ satisfies  the following problem
\begin{equation}\label{eq2.2}
  \left\{
\begin{array}{l}
 u_t=-u_{rrrr}-\frac{2(n-1)}{r}u_{rrr}+\Big(2W''(u)-\frac{(n-1)(n-3)}{r^2}\Big)u_{rr}
 +\Big(\frac{2(n-1)W''(u)}{r}+\frac{(n-1)(n-3)}{r^3}\Big)u_{r}
\\[3mm]
\hspace{0.8cm}+W'''(u)u^2_r-W'(u)W''(u),
 \qquad \ \ \forall\, (t, r)\in (-\infty, -T]\times(0, +\infty),
\\[3mm]
  u_{rrr}(t, 0)=u_r(t, 0)=0,\hspace{2.35cm}\text{for all}\ t\in (-\infty, -T],
  \end{array}
\right.
\end{equation}
where $r=\abs{x}$ and $T$ is a large positive number whose value can be adjusted at different steps.
For convenience,  we denote the right hand side of the first equation
in (\ref{eq2.2}) by $F(u)$,  that is
\begin{equation}\label{defF}
  \begin{aligned}
F(u):=&-u_{rrrr}-\frac{2(n-1)}{r}u_{rrr}+\Big(2W''(u)-\frac{(n-1)(n-3)}{r^2}\Big)u_{rr}
\\[2mm]
&+\Big(\frac{2(n-1)W''(u)}{r}+\frac{(n-1)(n-3)}{r^3}\Big)u_{r}+W'''(u)u^2_r-W'(u)W''(u).
\end{aligned}
\end{equation}

Our purpose is to find a solution of \eqref{eq2.2} with the property
\begin{equation*}
  u(t, r)\approx \omega\big(r-\rho(t)\big),  
\end{equation*}
where $\omega(y)$ is the solution of problem \eqref{eqq}.
Firstly,  we notice that $\omega\big(r-\rho(t)\big)$ does not satisfies the boundary conditions in \eqref{eq2.2}.
 A smooth cut-off function $\chi(r)$ can be defined in the form
\begin{equation}\label{dcut-off}
  \chi(r)=0, \qquad \text{for}\ r\leq\frac{\delta_0}{2}\qquad \text{and}\qquad \chi(r)=1, \qquad \text{for}\ r\geq\delta_0,
\end{equation}
for some  small fixed positive number $\delta_0$.
We define the first approximate solution of (\ref{eq2.2}) as the following
\begin{equation}\label{eq2.3}
  \widehat{\omega}(t, r)=\omega\big(r-\rho(t)\big)\chi(r)+\chi(r)-1.
\end{equation}
Here we assume that the function $\rho(t)$ has the from
\begin{equation}
   \rho(t)= \gamma_n(t)+h(t),
\label{drho1}
\end{equation}
where the function $h(t)=O((\log|t|)^{-1})$ as $t\rightarrow-\infty$
and the function $\gamma_n(t)$ is defined in \eqref{willmoreflowradial}-\eqref{sphere eq}, i.e.
\begin{equation}\label{dgamma-n}
 \gamma_n(t):=\sqrt[4]{-2(n-1)^2(n-3)t}, \qquad t<0,
\end{equation}
which is a radial solution of Willmore flow equation \eqref{willmore} with $n\geq4$.
More precisely,  we assume that $h(t)$ satisfies the following constraint
\begin{equation}\label{assumeh}
  \sup_{t\leq-1}\abs{h(t)}+\sup_{t\leq-1}\left\{\frac{\abs{t}}{\log\abs{t}}\abs{h'(t)}\right\}\leq 1.
\end{equation}

Secondly,  by using \eqref{eqq},  we have that
\begin{equation}
\begin{aligned}\label{eq2.4}
  -\partial_t\widehat{\omega}(t, r)+F(\widehat{\omega}(t, r))&=\rho'(t)\omega'\big(r-\rho(t)\big)-\frac{(n-3)(n-1)}{r^2}\omega''\big(r-\rho(t)\big)
  \\[2mm]
&\quad+\frac{(n-1)(n-3)}{r^3}\omega'\big(r-\rho(t)\big), \qquad \text{for}\ r>\delta_0,
\end{aligned}
\end{equation}
where the operator $F(u)$ is defined by \eqref{defF}.
By \eqref{dgamma-n}, we find that the second term in the right hand side of equality \eqref{eq2.4},
that is
\begin{equation*}
  -\frac{(n-1)(n-3)}{r^2}\omega''\big(r-\rho(t)\big),
\end{equation*}
has a slow decay of order $O\left(\abs{t}^{-\frac{1}{2}}\right)$ as $t$ goes to negative infinity.
However,  it is not enough to solve equation
\eqref{eq2.4} since this term is much bigger than other terms in \eqref{eq2.4}.
To cancel it and improve the approximate solution,  inspired by \cite{r},  we define a correction function
\begin{equation}\label{w1}
  \widetilde{\omega}(y)
:=-\omega'(y)\int_0^y \Bigg[
(\omega'(\hat{y}))^{-2}\int_{-\infty}^{\hat{y}}\frac{s(\omega'(s))^2}{2}\mathrm{d}s
\Bigg]{\mathrm d}\hat{y},
\end{equation}
Then we have that
\begin{equation}\label{eq2.1}\begin{aligned}
  L^*(\widetilde{\omega}):=\big[-\partial_{yy}+W''(\omega(y))\big]\widetilde{\omega}(y)=\frac{1}{2}y\omega'(y),
\qquad
(L^*)^2\big[\widetilde{\omega}(y)\big]=-\omega''(y), \qquad \forall\,  y\in\R,
\end{aligned}\end{equation}
and $\widetilde{\omega}(y)$ is an odd function with exponential decay such that
\begin{equation}\label{dwtd}
  \int_{\R}\omega'(y)\widetilde{\omega}(y){\mathrm d}y=0\quad
  \text{and}\quad
  \abs{\widetilde{\omega}(y)}\leq Ce^{-\frac{3\alpha}{4}\abs{y}},\qquad \text{for}\ y\in\R.
\end{equation}
At last,  we define an approximate solution of problem \eqref{eq2.2} as the following
\begin{equation}\label{dz}
  z(t, r):=\omega\big(r-\rho(t)\big)\chi(r)+\chi(r)-1+\frac{(n-1)(n-3)}{r^2}  \widetilde{\omega}\big(r-\rho(t)\big)\chi(r),
\end{equation}
where  the cut-off function $\chi(r)$ and the function $\rho(t)$ are given
by \eqref{dcut-off} and \eqref{drho1}.

We will look for a solution of equation (\ref{eq2.2}) of the form
\begin{equation}\label{eq2.30000}
  u(t, r)=z(t, r)+\phi(t, r),
\end{equation}
where $\phi$ is a small perturbation term.
This can be done by using the Lyapunov-Schmidt reduction method in two steps.

\noindent {\bf (1).}
The first step (see Sections \ref{sec:lp}-\ref{sec:nl}) is solving the following projected version of problem (\ref{eq2.2}) in terms of $\phi(t, r)$:
\begin{equation}\label{eq2.10}
\begin{aligned}
\phi_t=L[\phi]+E(t, r)+N(\phi)-c(t)\partial_r\widehat{\omega}(t, r)\qquad \text{in}\ (-\infty, -T]\times (0, \infty),
\end{aligned}
\end{equation}
and
\begin{equation}\label{eq2.11}
 \int_{0}^\infty \phi(t, r)\omega'\big(r-\rho(t)\big) r^{n-1}{\mathrm d}r=0, \qquad  \text{for all}\ t<-T,
\end{equation}
where the function $\widehat{\omega}(t, r)$ is defined by \eqref{eq2.3},  the error term $E(t, r)$ and nonlinear term $N(\phi)$ are defined respectively by
\begin{equation}\label{Error}
E(t, r):=F\big(z(t, r)\big)-\frac{\partial z(t, r)}{\partial t}
\end{equation}
and
\begin{equation}\label{nonlinearterm}
N(\phi):=F\big(z(t, r)+\phi(t, r)\big)-F\big(z(t, r)\big)-F'\big(z(t, r)\big)[\phi].
\end{equation}
In the above, $F(u)$ is defined by \eqref{defF} and the linear operator $L[\phi]:=F'\big(z(t, r)\big)[\phi]$ is defined as follows
\begin{equation}\label{eq2.13}
\begin{aligned}
F'\big(z(t, r)\big)[\phi]:=&-\phi_{rrrr}-\frac{2(n-1)}{r}\phi_{rrr}
+\Bigg[2W''\big(z(t, r)\big) -\frac{(n-1)(n-3)}{r^2}\Bigg]\phi_{rr}
\\[2mm]
&-\big(W''\big(z(t, r)\big)\big)^2\phi
+\Bigg[\frac{2(n-1)W''\big(z(t, r)\big)}{r}-\frac{(3-n)(n-1)}{r^3}\Bigg]\phi_{r}
\\[2mm]
&+2W'''\big(z(t, r)\big)\phi_rz_r-W'''\big(z(t, r)\big)W'\big(z(t, r)\big)\phi+W^{(4)}\big(z(t, r)\big)\abs{z_r}^2\phi
\\[2mm]
&+2W'''\big(z(t, r)\big)z_{rr}\phi+2\frac{n-1}{r}W'''\big(z(t, r)\big)z_r\phi.
\end{aligned}
\end{equation}

\noindent {\bf (2).}
The second step is to choose the function $c(t)$ in such a way that $\phi$ satisfies the orthogonality condition (\ref{eq2.11}),  namely the following equality holds:
\begin{align}
  &c(t)\int_{0}^\infty\partial_r\widehat{\omega}(t, r)\omega'\big(r-\rho(t)\big)r^{n-1}{\mathrm d}r
\nonumber\\[2mm]
&=\int_{0}^\infty \Big[\omega'''\big(r-\rho(t)\big)+\frac{n-1}{r}\omega''\big(r-\rho(t)\big)
  -W''\big(z(t, r)\big)\omega'\big(r-\rho(t)\big)\Big]
\nonumber\\[2mm]
&\qquad\qquad\times\left(-\phi_{rr}-\frac{n-1}{r}\phi_r+W''\big(z(t, r)\big)\phi\right)
r^{n-1}{\mathrm d}r
\nonumber\\[2mm]
&\quad +\int_{0}^\infty\left[\partial_{rr}z(t, r)
  +\frac{n-1}{r}\partial_{r}z(t, r)-W'\big(z(t, r)\big)\right]\phi \omega'\big(r-\rho(t)\big)r^{n-1}{\mathrm d}r
\nonumber\\[2mm]
&\quad +\int_{0}^\infty \phi(t, r)\partial_t\big[\omega'\big(r-\rho(t)\big)\big]r^{n-1}{\mathrm d}r
\ +\
\int_{0}^\infty\big(E(t, r)+N(\phi)\big)\omega'\big(r-\rho(t)\big)r^{n-1}{\mathrm d}r,
\label{eq2.14}
\end{align}
for all $t<-T$. Later on, in Section \ref{sec:tc}, we will choose $h(t)$ such that $c(t)=0$.
This means that the function $u$ in \eqref{eq2.30000} will exactly solve \eqref{eq2.2}.

\subsection {Estimates of the error terms}
 We will establish some estimates for the error term $E(t, r)$ in \eqref{Error}.
By Taylor's formula,  the definitions in \eqref{defF} and \eqref{dz},  we have that
\begin{align*}
E(t, r)=&F\big(z(t, r)\big)-\frac{\partial z(t, r)}{\partial t}
  \\[2mm]
=&F\big(\widehat{\omega}(t, r)+\widetilde{z}(t, r)\big)
-\frac{\partial\big[ \widehat{\omega}(t, r)+\widetilde{z}(t, r)\big]}{\partial t}
   \\[2mm]
=&F\big(\widehat{\omega}(t, r)\big)
+F'\big(\widehat{\omega}(t, r)\big)\big[\widetilde{z}(t, r)\big]
+F''\big(\widehat{\omega}(t, r)+\theta\widetilde{z}(t, r)\big)\big[\widetilde{z}(t, r), \widetilde{z}(t, r)\big]
   \\[2mm]
&-\frac{\partial \widetilde{z}(t, r)}{\partial t}-\frac{\partial \widehat{\omega}(t, r)}{\partial t}
   \\[2mm]
    :=&E_1(t, r)+E_2(t, r),
\end{align*}
where $\theta\in(0, 1)$.
In the above,  the operators $F(u)$ and $F'(u)[v]$ are given by \eqref{defF} and \eqref{eq2.13} respectively,  and
 the function $\widehat{\omega}(t, r)$ is given by \eqref{eq2.3} and $\widetilde{z}(t, r)$ is defined by
 \begin{equation}\label{dzt}\begin{aligned}
   \widetilde{z}(t, r):=\frac{(n-1)(n-3)}{r^2}  \widetilde{\omega}\big(r-\rho(t)\big)\chi(r),
 \end{aligned}\end{equation}
with $\rho(t)$ and $\chi(r)$ given by \eqref{drho1} and \eqref{dcut-off} respectively.
The operator $F''(u)[v_1, v_2]$ is defined as the following
\begin{equation}\label{dFs}\begin{aligned}
 F''(u)[v_1, v_2]:=&\Delta\big[W'''(u)v_1v_2\big]
\ -\ \left\{W'''(u)W''(u) \ +\ W^{(4)}(u)\big[-\Delta u+W'(u)\big]\right\}v_1v_2
 \\[2mm]
 &\ +\ W'''(u)\Big\{\big[\Delta v_1-W''(u)v_1\big]v_2 \ +\ \big[\Delta v_2-W''(u)v_2\big]v_1\Big\}.
\end{aligned}
\end{equation}
 The terms $E_1(t, r)$ and $E_2(t, r)$  have the following explicit forms
 \begin{equation}\label{de1}
   E_1(t, r):=F\big(\widehat{\omega}(t, r)\big)-\frac{\partial \widehat{\omega}(t, r)}{\partial t}+\frac{(n-1)(n-3)}{r^2}\partial_{rr}\widehat{\omega}(t, r)+\frac{(n-1)(n-3)^2}{2r^3}\partial_r\widehat{\omega}(t, r),
 \end{equation}
 and
\begin{align}
  E_2(t, r):=&-\left[\frac{(n-1)(n-3)}{r^2}\partial_{rr}\widehat{\omega}(t, r)+\frac{(n-1)(n-3)^2}{2r^3}\partial_r\widehat{\omega}(t, r)\right]
  +F'\big(\widehat{\omega}(t, r)\big)\big[\widetilde{z}(t, r)\big]
\nonumber \\[2mm]
& +F''\big(\widehat{\omega}(t, r)+\theta\widetilde{z}(t, r)\big)\big[\widetilde{z}(t, r), \widetilde{z}(t, r)\big]
   -\frac{\partial \widetilde{z}(t, r)}{\partial t}.
\label{de2}\end{align}
The main result is given by the following lemma.
\begin{lem}\label{lem10}
Let $\alpha:=\sqrt{W''(1)}$,  $p\in(n, n+1]$ and $T>1$,  we set
\begin{equation}\label{dpsi}
  \Phi(t, r):=\left\{
\begin{array}{l}
\frac{\log\abs{t}}{\abs{t}^{\frac{1}{2}}}\frac{1}{\big(1+\abs{r-\gamma_n(t)-\frac{1}{4\alpha}\log\abs{t}}\big)^{p}}, \qquad \text{if} \ r\in\left[\delta_0, +\infty\right);
\\[4mm]
 \frac{\log\abs{t}}{\abs{t}^{\frac{1}{2}}}\overline{\chi}_{\left\{\frac{\delta_0}{2}\leq r<\delta_0\right\}},
 \hspace{2.33cm}\text{if}\ r\in\left(0, \delta_0\right);
  \end{array}
\right.
\end{equation}
where the function $\gamma_n(t)$ is defined by \eqref{dgamma-n}
and $\delta_0$ is a small positive number given in \eqref{dcut-off}.
Here $\overline{\chi}_{A}$ is the characteristic function of the set $A$.
Then there exists constant $C>0$ which depends only on $\delta_0$,  $\alpha$,  and $n$,  such that
$$|E(t, r)|\leq C\frac{\Phi(t, r)}{\log\abs{t}}, $$
for all $(t, r)\in (-\infty, -T]\times(0, +\infty)$.
\end{lem}
\begin{proof}
We first estimate the term $E_1(t, r)$ in \eqref{de1}. By the definitions of $F(u)$ and $\widehat{\omega}(t, r)$ in \eqref{defF} and \eqref{eq2.3},  lemma \ref{lem1},
we have that
\begin{equation*}\begin{aligned}
  E_1(t, r)=&\left[\rho'(t)+\frac{(n-1)^2(n-3)}{2r^3}\right]\omega'\big(r-\rho(t)\big)\overline{\chi}_{\left\{ r\geq\delta_0\right\}}+O\left(\frac{1}{[\gamma_n(t)]^2}\right)\overline{\chi}_{\left\{ \frac{\delta_0}{2}<r<\delta_0\right\}},
\end{aligned}\end{equation*}
where $\overline{\chi}_{A}$ is the characteristic function of the set $A$.
Furthermore,  using the assumption of $\rho(t)$ in \eqref{drho1}-\eqref{assumeh},  the definition of $\gamma_n(t)$ in \eqref{dgamma-n} and Lemma \ref{lem1},   we have that
\begin{align}
\nonumber \abs{E_1(t, r)}
\leq&C\Bigg[\frac{1}{\abs{t}^{\frac{3}{4}}}\overline{\chi}_{\left\{ r\geq\delta_0\right\}}+\frac{1}{r^3}\overline{\chi}_{\left\{ r\geq\frac{\gamma_{n}(t)}{2}\right\}}+\frac{1}{r^3}\overline{\chi}_{\left\{ \frac{\gamma_{n}(t)}{2}\geq r\geq\delta_0\right\}}\Bigg]e^{-\alpha\abs{r-\gamma_n(t)}}
 \,+\,
\frac{C}{\abs{t}^{\frac{1}{2}}}\overline{\chi}_{\left\{ \frac{\delta_0}{2}<r<\delta_0\right\}}
\nonumber \\[3mm]
\leq &C\frac{e^{-\frac{3\alpha}{4}\abs{r-\gamma_n(t)-\frac{1}{3\alpha}\log\abs{t}}}}{\abs{t}^{\frac{1}{2}}}\overline{\chi}_{\left\{ r\geq\delta_0\right\}}
+\frac{C}{r^3}e^{-\frac{3\alpha}{4}\abs{r-\gamma_n(t)}}\overline{\chi}_{\left\{ \frac{\gamma_{n}(t)}{2}\geq r\geq\delta_0\right\}}
+\frac{C}{\abs{t}^{\frac{1}{2}}}\overline{\chi}_{\left\{ \frac{\delta_0}{2}<r<\delta_0\right\}}
\nonumber \\[3mm]
\leq& C\frac{\Phi(t, r)}{\log\abs{t}},\qquad \text{ for all}\ (t, r)\in (-\infty, -T]\times(0, +\infty),
\label{eE1}
\end{align}
where $C$ is a positive constant only depending on $\delta_0$,  $\alpha$ and $n$. Here we used  the fact that there exits a positive constant $C>0$ only depending on $\delta_0$,  $\alpha$ and $n$ such that
\begin{equation}\label{dafact}
  \frac{1}{r^3}e^{-\frac{3\alpha}{4}\abs{r-\gamma_n(t)}}\overline{\chi}_{\left\{ \frac{\gamma_{n}(t)}{2}\geq r\geq\delta_0\right\}}\leq
  Ce^{-\frac{\alpha}{4}\gamma_n(t)}\quad \text{and} \quad e^{-\frac{3\alpha}{4}\abs{x}}\leq \frac{C}{\big(1+\abs{x}\big)^{p}},
\end{equation}
for all $ x\in\R$ and $p\in(n, n+1]$, where $C>0$ does not depend on $x$ and $t$.

Next we consider the term $E_2(t, r)$ in \eqref{de2}.  Using the definitions of linear operators $F'(u)[v]$ and $F''(u)[v, v]$
in \eqref{eq2.13} and \eqref{dFs},  Lemma \ref{lem1} and estimate in \eqref{dwtd},  we derive that
\begin{equation*}
\begin{aligned}
E_2(t, r)=&\Bigg\{F'\big(\omega\big(r-\rho(t)\big)\big)\big[\widetilde{z}(t, r)\big]-\frac{(n-1)(n-3)}{r^2}\omega''\big(r-\rho(t)\big)
+\frac{(n-1)(n-3)^2}{2r^3}\omega'\big(r-\rho(t)\big)
\\[2mm]
& +F''\big(\omega\big(r-\rho(t)\big)+\theta\widetilde{z}(t, r)\big)\big[\widetilde{z}(t, r), \widetilde{z}(t, r)\big]
+\rho'(t)\omega'\big(r-\rho(t)\big)\frac{(n-1)(n-3)}{r^2}\Bigg\}\overline{\chi}_{\left\{r\geq\delta_0\right\}}\\&+O\left(\frac{1}{[\gamma_n(t)]^2}\right)\overline{\chi}_{\left\{ \frac{\delta_0}{2}<r<\delta_0\right\}}
\\[2mm]
=&\Bigg\{\left[-\Big(\partial_{rr}-W''\big(\omega\big(r-\rho(t)\big)\big)\Big)^2\widetilde{\omega}\big(r-\rho(t)\big)-\omega''\big(r-\rho(t)\big)\right]\frac{(n-1)(n-3)}{r^2}
\\[2mm]
& +O\left(\frac{1}{r^3}+\frac{\rho'(t)}{r^2}\right)e^{-\frac{3\alpha}{4}\abs{r-
\rho(t)}}
\Bigg\}\overline{\chi}_{\left\{r\geq\delta_0\right\}}
\ +\
O\left(\frac{1}{[\gamma_n(t)]^2}\right)\overline{\chi}_{\left\{ \frac{\delta_0}{2}<r<\delta_0\right\}}
\\[2mm]
=&O\left(\frac{1}{r^3}+\frac{\rho'(t)}{r^2}\right)e^{-\frac{3\alpha}{4}\abs{r-
\rho(t)}}\overline{\chi}_{\left\{r\geq\delta_0\right\}}
\ +\
O\left(\frac{1}{[\gamma_n(t)]^2}\right)\overline{\chi}_{\left\{ \frac{\delta_0}{2}<r<\delta_0\right\}},
\end{aligned}
\end{equation*}
where we used the equalities in \eqref{eq2.1} and the fact that the function $\widetilde{\omega}(t, r)$ in \eqref{w1} and its derivatives are all exponentially decaying.

By the same argument in \eqref{eE1} and the equality in \eqref{dafact},  we can get that
\begin{equation*}
  |E_2(t, r)|\leq C\frac{\Phi(t, r)}{\log\abs{t}},\qquad \text{ for all}\ (t, r)\in (-\infty, -T]\times(0, +\infty),
\end{equation*}
where $C$ is a positive constant only depending on $\alpha$ and $n$.

Eventually, combining the above estimates of the terms $E_1(t,r)$ and $E_2(t,r)$,  we can obtain the desired results.
\end{proof}
\begin{rem}
According to the above proof of Lemma \ref{lem10},  it is easy to find that the error term $E(t, r)$ has an exponentially decaying in space variable
\begin{equation}\label{deee}
  \abs{E(t, r)}\leq \frac{C}{\abs{t}^{\frac{1}{2}}e^{\frac{\alpha}{2}\abs{r-\gamma_n(t)-\frac{1}{4\alpha}\log\abs{t}}}},
\end{equation}
for all $r\geq\delta_0$,  where $C>0$ does only depend on $n$ and $\alpha$.
Our goal is to solve nonlinear problem \eqref{eq2.10}-\eqref{eq2.11}.
Hence,  according to \eqref{deee},  we may consider the following linear parabolic problem of fourth order
\begin{equation}\label{deeee}
  -\partial_t\varphi+F'(z(t, r))[\varphi]=f(t, r), \qquad  (t, r)\in(-\infty, -T)\times(0, +\infty),
\end{equation}
with the function $f(t, r)$ satisfying
\begin{equation*}
\sup_{(t, r)\in(-\infty, -T)\times(0, +\infty)}
  \frac{\abs{f(t, r)}}{\abs{t}^{\frac{1}{2}}e^{\frac{\alpha}{2}\abs{r-\gamma_n(t)-\frac{1}{4\alpha}\log\abs{t}}}}<+\infty,
\end{equation*}
where the linear operator $F'(z(t, r))$ is defined in \eqref{eq2.13}. However,  the heat kernel of linear parabolic operator
 $-\partial_t+F'(z(t, r))$ would not has the exponential decay in \eqref{deee}.
As far as we known, it can satisfy the polynomial decay. Hence, we provide an estimate of the error term $E(t,r)$ with the polynomial decay
 in the above lemma.
\qed
\end{rem}

\section{The Linear Problem}\label{sec:lp}

In this section,  firstly,  we will obtain the solvability of a class of semilinear biharmonic parabolic equations by applying
some properties of biharmonic heat kernel and fixed-point arguments,  which is given in Proposition \ref{prop8}.
Secondly,  we will prove that the linear projected problem \eqref{eq3.1} is solvable by using Proposition \ref{prop8} and a priori estimate in Lemma \ref{lem5}. The main result is given by Proposition \ref{prop2}.

\subsection{A few results of linear parabolic equations with a biharmonic operator} We first collect some known results for homogeneous biharmonic parabolic equation,   from \cite{ FG, GG1, G}. Its solution can be represented by the convolution of a biharmonic heat kernel and initial function.
\begin{prop}\label{prop6}We consider the following Cauchy problem for the biharmonic heat equation:
 \begin{equation}\label{heateq}
\left\{
\begin{array}{l}
u_t+(-\Delta)^2u=0\qquad  \text{in}\ \R^{n+1}_+:=(0, +\infty)\times\R^n,
\\[2mm]
u(0, x)=u_0(x)\hspace{0.3cm} \qquad   \text{in}\ \R^n,
 \end{array}
\right.
\end{equation}
where $n\geq 1$ and $u_0\in C^1(\R^n)\cap L^\infty(\R^n)$.
Then \eqref{heateq} admits a unique global in time solution explicitly given by
$$
u(t, x)=\int_{\R^n}u_0(y)p_n(t, x-y){\mathrm d}y,
\qquad
p_n(t, x):=\bar{\alpha}_nt^{-n/4}f_n\Big(\frac{\abs{x}}{t^{1/4}}\Big),
\qquad
\forall\,(t, x)\in \R^{n+1}_+.
$$
Here $p_n(t, x)$ is called the biharmonic heat kernel and $\bar{\alpha}_n$ denotes a suitable positive normalization  number which depends on $n$ and satisfies
$$
\bar{\alpha}_nt^{-n/4}\int_{\R^n}f_n\Big(\frac{\abs{y}}{t^{1/4}}\Big){\mathrm d}y=1,   \qquad \text{for all}\ t>0.
$$
Moreover,  the function $f_n$ is given by
$$
f_n(s):=s^{1-n}\int_{0}^{+\infty}e^{-\varrho^4}(s\varrho)^{\frac{n}{2}}J_{\frac{n-2}{2}}(s\varrho) {\mathrm d}\varrho,  \qquad \text{for}\ s>0,
$$
where $J_\nu$ denotes the $\nu$-th Bessel function of the first kind.
There exist two constants $K_n$ and $\mu_n$ depending on $n$ such that
\begin{equation}\label{esf}
\abs{f_n(s)}\leq K_n\exp(-\mu_n s^{4/3}),  \qquad \text{for all}\ s\geq 0.
\end{equation}
For the derivative of $f_n$,  the following formula holds
\begin{equation}\label{fdf}
f'_n(s)=-sf_{n+2}(s).
\end{equation}
\qed
\end{prop}

Next we will use the above results to  study the following inhomogeneous problem:
\begin{equation}\label{noneq}
\left\{
\begin{array}{l}
u_t+(-\Delta)^2u=f(t, x)\qquad\text{in}\ (t_0, t_1)\times\R^{n},
\\[2mm]
u(t_0, x)=u_0(x)\hspace{1.7cm} \text{in}\ \R^n,
 \end{array}
\right.
\end{equation}
where $f(t, x)\in  L^\infty((t_0, t_1)\times\R^{n})$ and $u_0\in C^1(\R^n)\cap L^\infty(\R^n)$.
\begin{defn}\label{demild}
Assume that $t_1\in (t_0, +\infty)$. We say that $u\in L^\infty((t_0, t_1)\times\R^n)$ is a mild solution of \eqref{noneq} if
\begin{equation}\label{demil}
u(t, x)=\Gamma_{t-t_0}[u_0](x)+\int^t_{t_0}\Gamma_{t-\tau}[f(\tau, \cdot)](x){\mathrm d}\tau,\qquad\text{for}\ (t, x)\in (t_0, t_1)\times\R^n.
\end{equation}
Here $\Gamma_{t}$ is a linear operator defined by
$$
\Gamma_{t}[u](x):=\Big(p_n(t, \cdot)\ast u\Big)(x)=\int_{\R^n}p_n(t, x-y)u(y){\mathrm d}y,
$$
where the biharmonic heat kernel $p_n(t, x)$ is given by Proposition \ref{prop6}.
In addition,  when $t_1=+\infty$,  a mild solution also can be defined by
\eqref{demil} provided $u(t,  x)\in L^\infty_{loc}(t_0, +\infty)\times L^\infty(\R^n)$.
\end{defn}

Notice that it is not hard to verify that $u$ is a mild solution of \eqref{noneq} if and only if $u$ solves it in the pointwise sense.
We give some regularity estimates for mild solutions.
\begin{prop}\label{prop7}Let $u$ be a mild solution of problem \eqref{noneq}. Assume that $t_1<+\infty$. Then the following estimates are valid:
 \begin{itemize}
   \item If $u_0\in L^{\infty}(\R^n)\cap C^1(\R^n)$,  then for every $t_*\in(t_0, t_1)$,  $t_*>0$,  $\vartheta\in (0, 4)$ and $\theta\in (0, 1)$,  it holds that
   \begin{equation}\label{es1}
     \begin{aligned}
   \sup_{t\in(t_*, t_1)}\|u(t, \cdot)\|_{C^{\vartheta}(\R^n)}
   +\sup_{x\in\R^n}\|u(\cdot, x)\|_{C^{\theta}(t_*, t_1)}
   \leq
   C_1\Big(\|f\|_{L^{\infty}((t_*, t_1)\times\R^n)}+\|u_0\|_{L^\infty(\R^n)}\Big),
     \end{aligned}
   \end{equation}
   for some positive constant $C_1$ depending on $n, \theta, t_1-t_0,  t_1-t_*, \vartheta$ and $t_*$.
   \item If $u_0\in C^\vartheta(\R^n)$ for some $\vartheta\in (0, 4)$ and $t_0>0$,  then it holds
   \begin{equation}\label{es2}
     \begin{aligned}
     \sup_{t\in(t_0, t_1)}\|u(t, \cdot)\|_{C^{\vartheta}(\R^n)}+\sup_{x\in\R^n}\|u(\cdot, x)\|_{C^{\frac{\vartheta}{4}}(t_0, t_1)}
     \leq C_{2}\Big(\|f\|_{L^{\infty}((t_0, t_1)\times\R^n)}+\|u_0\|_{C^\vartheta(\R^n)}\Big),
     \end{aligned}
   \end{equation}
   for some positive constant $C_2$ depending on $t_1-t_0,  n$ and $\vartheta$.
 \end{itemize}
\end{prop}
\begin{proof}
According to the formula \eqref{fdf} and some direct computations,  we can derive some estimates for the derivatives of biharmonic heat kernel $p_n(t, x)$ as follows:
\begin{itemize}\label{1}
  \item $\abs{\partial_tp_n(t, x)}\leq C_n\Big(t^{-1}\abs{p_n(t, x)}+t^{-1}\abs{x}^2\abs{p_{n+2}(t, x)}\Big)$;
\medskip
  \item $\abs{\nabla_x p_n(t, x)}\leq C_n\abs{x}\abs{p_{n+2}(t, x)}$;
\medskip
  \item $\abs{\nabla^2_x p_n(t, x)}\leq C_n\Big(\abs{p_{n+2}(t, x)}+\abs{x}^2\abs{p_{n+4}(t, x)}\Big)$;
\medskip
  \item $\abs{\nabla^3_x p_n(t, x)}\leq C_n\Big(\abs{x}\abs{p_{n+4}(t, x)}+\abs{x}^3\abs{p_{n+6}(t, x)}\Big)$;
\medskip
  \item $\abs{\nabla^4_x p_n(t, x)}\leq C_n\Big(\abs{p_{n+4}(t, x)}+\abs{x}^2\abs{p_{n+6}(t, x)}+\abs{x}^4\abs{p_{n+8}(t, x)}\Big)$,
\end{itemize}
where $C_n>0$ only depends on $n$.

We decompose the mild solution in (\ref{demil}) as follows
\begin{equation*}
  u(t, x)=U_1(t, x)+U_2(t, x),
\end{equation*}
where
$$
U_1(t, x)=\Gamma_{t-t_0}[u_0](x),
\quad
U_2(t, x)=\int^t_{t_0}\Gamma_{t-\tau}[f(\tau, \cdot)](x){\mathrm d}\tau.
$$

The analysis will begin with the first term $U_1(t, x)$.
For any $t>t_0$ and $x\in\R^n$,  using the above estimates and \eqref{esf},  we derive that
\begin{align*}
\abs{\partial_tU_1(t, x)}=\abs{\int_{\R^n}\partial_tp_n(t-t_0, x-y)u_0(y){\mathrm d}y}
\leq \frac{C_n}{t-t_0}\|u_0\|_{L^\infty(\R^n)}.
\end{align*}
Moreover,  using (\ref{esf}),  we have that
$$
\abs{U_1(t, x)}\leq\int_{\R^n}\abs{p_n(t-t_0, x-y)u_0(y)}{\mathrm d}y\leq C_n\|u_0\|_{L^\infty(\R^n)}.
$$

For any $\vartheta\in(0, 4)$ and $t,  t^1\in(t_0, t_1)$,  using the fact $p_n(t, x)=t^{-n/4}p_n(\frac{x}{t^{1/4}}, 1)$,
\begin{equation}\label{esq}\begin{aligned}
\frac{\abs{U_1(t, x)-U_1(t^1, x)}}{\abs{t-t^1}^{\frac{\vartheta}{4}}}&\leq\abs{\int_{\R^n}
\frac{u_0\big((t-t_0)^{1/4}y\big) \,-\, u_0\big((t^1-t_0)^{1/4}y\big)}
{\abs{t-t^1}^{\frac{\vartheta}{4}}}p_n(1, y){\mathrm d}y}
\\[2mm]
&\leq\|u_0\|_{C^\vartheta(\R^n)}\int_{\R^n}\abs{p_n(1, y)}\abs{y}^\vartheta {\mathrm d}y\leq C_n\|u_0\|_{C^\vartheta(\R^n)}.
\end{aligned}\end{equation}
Here we have used the inequality $a^{\vartheta/4}-b^{\vartheta/4}\leq (a-b)^{\vartheta/4}$ for any $a\geq b\geq0$ and $\vartheta\in(0, 4)$ in the second inequality,  $C_n>0$ only depends on $n$.

Moreover,  for the derivatives of $U_1$ with respect to $x$,  by (\ref{esf}),  we have that
\begin{align*}
\abs{\nabla_xU_1(t, x)}&=\abs{\int_{\R^n}\nabla_x p_n(t-t_0, x-y)u_0(y){\mathrm d}y}
\\[2mm]
&\leq C_n\|u_0\|_{L^\infty(\R^n)}\int_{\R^n}\abs{y}\abs{p_{n+2}(y, t-t_0)}{\mathrm d}y
\\[2mm]
&\leq \frac{C_n}{(t-t_0)^{1/4}}\|u_0\|_{L^\infty(\R^n)},
\end{align*}
where $C_n$ does only depend on $n$.
Similarly,  we can get that
\begin{align*}
\abs{\nabla^2_xU_1(t, x)}\leq&\frac{C_n}{(t-t_0)^{1/2}}\|u_0\|_{L^\infty(\R^n)},
\qquad
\abs{\nabla^3_xU_1(t, x)}\leq \frac{C_n}{(t-t_0)^{3/4}}\|u_0\|_{L^\infty(\R^n)},
\end{align*}
and $$\abs{\nabla^4_xU_1(t, x)}\leq \frac{C_n}{(t-t_0)}\|u_0\|_{L^\infty(\R^n)}.$$
Thus for any $x, y\in\R^n$ and $\theta\in(0, 1)$,  we have
$$[U_1]_\theta(t):=\frac{\abs{U_1(t, x)-U_1(t, y)}}{\abs{x-y}^\theta}\leq\left\{
\begin{array}{l}
2\|U_1(t, \cdot)\|_{L^\infty(\R^n)},\  \quad \ \quad \text{if}\ \abs{x-y}\geq1,
\\[2mm]
\|\nabla _x U_1(t, \cdot)\|_{L^\infty(\R^n)},\qquad \text{if}\ \abs{x-y}\leq1.
 \end{array}
\right.$$
Hence,  for any $t_*>t_0$,  we have
$$\sup_{t\in(t_*, t_1)}[U_1]_\theta(t)\leq \frac{C_n}{(t_*-t_0)^{1/4}}\|u_0\|_{L^\infty(\R^n)}.$$
By similar arguments,  we can obtain the following estimates
$$\sup_{t\in(t_*, t_1)}[\nabla U_1]_\theta(t)\leq \frac{C_n}{(t_*-t_0)^{1/2}}\|u_0\|_{L^\infty(\R^n)},
\qquad
\sup_{t\in(t_*, t_1)}[\nabla^2 U_1]_\theta(t)\leq \frac{C_n}{(t_*-t_0)^{3/4}}\|u_0\|_{L^\infty(\R^n)},
$$
and
$$
\sup_{t\in(t_*, t_1)}[\nabla^3 U_1]_\theta(t)\leq \frac{C_n}{(t_*-t_0)}\|u_0\|_{L^\infty(\R^n)}.
$$
Combining the above estimates,   we have that for any $\vartheta\in(0, 4)$,  $t_*>t_0$ and $\theta\in(0, 1)$,
\begin{equation}\label{esU1}
\sup_{t\in(t_*, t_1)}\|U_1(t, \cdot)\|_{C^{\vartheta}(\R^n)}+\sup_{x\in\R^n}\|U_1(\cdot, x)\|_{C^{\theta}(t_*, t_1)}
\leq
C_{n}\big[1+\varrho(t_*-t_0)\big]\|u_0\|_{L^\infty(\R^n)}
\end{equation}
where $\varrho(t_*-t_0)=\frac{1}{(t_*-t_0)}+\frac{1}{(t_*-t_0)^{1/4}}+\frac{1}{(t_*-t_0)^\theta}$.

Using the formula
$$
U_1(t, x)=\int_{\R^n}p_n(t-t_0, y)u_0(x-y){\mathrm d}y
$$
and \eqref{esq},  for $\vartheta\in (0, 4)$,  we have
\begin{equation}\label{esU2}
  \sup_{t\in(t_0, t_1)}\|U_1(t, \cdot)\|_{C^{\vartheta}(\R^n)}+\sup_{x\in\R^n}\|U_1(\cdot, x)\|_{C^{\frac{\vartheta}{4}}(t_0, t_1)}
     \leq C_{n}\|u_0\|_{C^\vartheta(\R^n)},
\end{equation}
where $C_{n}$ is a positive constant which depends on $n, \vartheta$.

Next we will estimate $U_2(t, x)$. First we find that
\begin{align*}
\abs{U_2(t, x)}\leq& C_n\int_{t_0}^t\int_{\R^n}\abs{p_n(t-\tau, x-y)f(\tau,  y)}{\mathrm d}y{\mathrm d}\tau \\[2mm]
\leq& C_n(t_1-t_0)\|f\|_{L^{\infty}}((t_1, t)\times\R^n).
\end{align*}
 For any $t, t^1\in(t_0, t_1)$,  by the previous estimate of $\partial_tp_n(t, x)$,   we have
$$
\abs{p_n(t, x)-p_n(t^1, x)}\leq C_n\abs{\log t- \log t^1}\Big(\abs{p_n\big((1-\theta_1)t+\theta_1 t^1, x\big)}+\abs{x}^2\abs{p_{n+2}\big((1-\theta_1)t+\theta_1 t^1, x\big)}\Big),
$$
for some $\theta_1\in(0, 1)$.
Using the above inequality and \eqref{esf},  for $\theta\in (0, 1)$,  we have
\begin{align*}
\frac{\abs{U_2(t, x)-U_2(t^1, x)}}{\abs{t-t^1}^\theta}\leq&
  \frac{\abs{\int_{t_0}^t\int_{\R^n}\Big[p_n(t-\tau, x-y)-p_n(t^1-\tau, x-y)\Big]f(y, \tau){\mathrm d}y{\mathrm d}\tau}}{\abs{t-t^1}^\theta}
  \\[2mm]
&+\frac{\abs{\int^t_{t^1}\int_{\R^n}p_n(t^1-\tau, x-y)f(\tau, y){\mathrm d}y{\mathrm d}\tau}}{\abs{t-t^1}^\theta}
  \\[2mm]
  \leq &C_n\Bigg(1+\log\frac{t_1-t_0}{\abs{t-t^1}}\Bigg)\abs{t-t^1}^{1-\theta}\|f\|_{L^{\infty}((t_1, t)\times\R^n)},
\end{align*}
where $C_n$ is a positive constant which only depends on $n$.

Second,  we study the derivatives of $U_2(t, x)$ with respect to $x$. By some direct computations and (\ref{esf}),  we have
\begin{align*}
\abs{\nabla_xU_2(t, x)}&=\abs{\int_{t_0}^{t}\int_{\R^n}\nabla_x p_n(t-\tau, x-y)f(\tau, y){\mathrm d}y{\mathrm d}\tau}
\\[2mm]
&\leq C_n
\|f\|_{L^{\infty}((t_0, t)\times\R^n)}\int_{t_0}^t\int_{\R^n}\abs{y}\abs{p_{n+2}(t-\tau, y)}{\mathrm d}y{\mathrm d}\tau
\\[2mm]
&\leq C_n \|f\|_{L^{\infty}((t_0, t)\times\R^n)}\int_{t_0}^t(t-\tau)^{-1/4}{\mathrm d}\tau
\\[2mm]
&\leq C_n (t-t_0)^{3/4} \|f\|_{L^{\infty}((t_0, t)\times\R^n)}.
\end{align*}
By the same arguments as above,  we can derive that
$$
\abs{\nabla^2_xU_2(t, x)}\leq C_n (t-t_0)^{1/2} \|f\|_{L^{\infty}((t_0, t)\times\R^n)},
\qquad
\abs{\nabla^3_xU_2(t, x)}\leq C_n (t-t_0)^{1/4} \|f\|_{L^{\infty}((t_0, t)\times\R^n)},
$$
and
for any $\theta\in(0, 1)$,
\begin{align*}
&\frac{\abs{\nabla^3_xU_2(t, x_1)-\nabla^3_xU_2(t, x_2)}}{\abs{x_1-x_2}^\theta}
\\[2mm]
& =\alpha_n\frac{\abs{\int_{t_0}^t\int_{\R^n}(t-\tau)^{-\frac{n+3}{4}}\Big[\nabla^3_xf_n\Big(\frac{\abs{x_1-y}}{(t-\tau)^{1/4}}\Big)-
 \nabla^3_xf_n\Big(\frac{\abs{x_2-y}}{(t-\tau)^{1/4}}\Big)\Big]f(y, \tau){\mathrm d}y{\mathrm d}\tau}}{\abs{x_1-x_2}^\theta}
 \\[2mm]
&\leq C_n \frac{\abs{\int_{t_0}^t\int_{\R^n}(t-\tau)^{-\frac{3}{4}}\Big[\nabla^3_xf_n\Big(\abs{y-\frac{x_1}{(t-\tau)^{1/4}}}\Big)-
 \nabla^3_xf_n\Big(\abs{y-\frac{x_2}{(t-\tau)^{1/4}}}\Big)\Big]f\Big((t-\tau)^{1/4}y, \tau\Big){\mathrm d}y{\mathrm d}\tau}}{\abs{x_1-x_2}^\theta}
\\[2mm]
 &\leq C_n(t-t_0)^{(1-\theta)/4}\|f\|_{L^{\infty}((t_0, t)\times\R^n)}.
 \end{align*}

Combing the above estimates of $U_2$,  for any $\vartheta\in(0, 4)$ and $\theta\in(0, 1)$,  we have that
\begin{equation}\label{esU21}
\sup_{t\in(t_0, t_1)}\|U_2(t, \cdot)\|_{C^{\vartheta}(\R^n)}+\sup_{x\in\R^n}\|U_2(\cdot, x)\|_{C^{\theta}(t_0, t_1)}
\leq
C_{n}\Big[\upsilon(t_1-t_0)+(t_1-t_0)^{1-\theta}\Big]\|f\|_{L^{\infty}((t_0, t)\times\R^n)}.
\end{equation}
On the other hand,  for $\vartheta\in(0, 4)$,  there holds
\begin{equation}\label{esU22}
  \sup_{t\in(t_0, t_1)}\|U_2(t, \cdot)\|_{C^{\vartheta}(\R^n)}+\sup_{x\in\R^n}\|U_2(\cdot, x)\|_{C^{\frac{\vartheta}{4}}(t_0, t_1)}
    \leq C_{n}\upsilon(t_1-t_0)\|f\|_{L^{\infty}((t_0, t)\times\R^n)},
\end{equation}
where $\upsilon(t-t_0)=(t-t_0)+(t-t_0)^{1-\frac{\vartheta}{4}}$ and $C_{n}$ is a positive constant which depends on $n$.

 The Proposition \ref{prop7} follows from \eqref{esU1},  \eqref{esU2},  \eqref{esU21} and \eqref{esU22}.
\end{proof}

We will apply the above proposition to study the solvability of a class of semilinear parabolic equations.
Let $u_0\in C^1(\R^n)\cap L^\infty(\R^n)$,  and we consider the initial value problem
\begin{equation}\label{heatnonlineareq}
\left\{
\begin{array}{l}
u_t+(-\Delta)^2u=G[\Delta u, \nabla u, u, t, x]\qquad \text{in}\ (t_0, t_1)\times\R^n,
\\[2mm]
 u(t_0, x)=u_0(x)\hspace{3.35cm} \text{in}\ \R^n,
 \end{array}
\right.
\end{equation}
where $G[p, \vec{q}, s, t, x]:\R\times\R^n\times\R\times[t_0, +\infty)\times\R^n\rightarrow \R$  is a measurable function satisfying
\begin{enumerate}
\item
For every $M>0$ and $T>t_0$,  there exists $C_{T, M}>0$ such that
$\Big{|}G[p, \vec{q}, s, t, x]\Big{|}\leq C_{T, M}$ for all $x\in\R^n$,  $t\in[t_0, T]$ and $p, \abs{\vec{q}}, s \in[-M, M]$.
\item
There is a constant $\sigma>0$ in such a way that
\begin{equation}\label{delpf}
  \Big{|}G[p_1, \vec{q}_1, s_1, t, x]-G[p_2, \vec{q}_2, s_2, t, x]\Big{|}\leq \sigma\Big(\abs{p_1-p_2}+\abs{\vec{q}_1-\vec{q}_2}+\abs{s_1-s_2}\Big),
\end{equation}
for all $ x\in\R^n, t\geq t_0$,  $p_1, p_2 \in\R$,  $\vec{q}_1, \vec{q}_2\in\R^n$ and $s_1, s_2\in\R$.
\end{enumerate}
 In particular,  we can take
\begin{equation}\label{defG}
G[\Delta u, \nabla u, u,  t,  x]=a(t, x)\Delta u+\sum_{i=1}^nb^i(t, x)\nabla_{x_i} u+c(t, x)u+g(t, x),
\end{equation}
where those functions $a(t, x)$,  $b^1(t, x), \cdots,  b^n(t,  x)$,  $c(t, x)$ are all belonging to $L^\infty ([t_0, +\infty), C^1(\R^n))$ and $g\in L^\infty([t_0, +\infty)\times\R^n)$.

Assume that $u(t, x)$ satisfies that
$$
\Delta u, \ \abs{\nabla u}, \ u\in L^{\infty}((t_0, t_1)\times\R^n),
$$
then by \eqref{demil},  $u$ is a mild solution of
\eqref{heatnonlineareq} if and only if
\begin{equation}\label{demils}
u(t, x)=\Gamma_{t-t_0}[u_0](x)+\int^t_{t_0}\Gamma_{t-\tau}\Big[G[\Delta u,  \nabla u,  u, \tau,  \cdot](\tau, \cdot)\Big](x){\mathrm d}\tau\qquad\text{for}\ (t, x)\in (t_0, t_1)\times\R^n.
\end{equation}
Notice that $G[\Delta u,  \nabla u,  u,  t,  x]\in L^{\infty}((t_0, t_1)\times\R^n)$ by the assumptions on $u$ and $G$.
Thus (\ref{demils}) is well defined by Proposition \ref{prop7}.

Now we devote to the study of the solvability of problem \eqref{heatnonlineareq}.
For convenience,  we define a map $\mathcal{N}_{G, u_0}$
from $L^\infty[(t_0, t_1); C^{2}(\R^{n})]$ to itself as follows
$$\mathcal{N}_{G, u_0}[u](t, x):=\Gamma_{t-t_0}[u_0](x)+\int^t_{t_0}\Gamma_{t-\tau}\Big[G[\Delta u,  \nabla u,  u](\cdot, \tau)\Big](x){\mathrm d}\tau\qquad\text{for}\ (x, t)\in \R^n\times(t_0, t_1), $$
for $u_0\in C^1(\R^n)\cap L^\infty(\R^n)$ and $u\in L^\infty[(t_0, t_1); C^{2}(\R^{n})]$. The main result is the following:

\begin{prop}\label{prop8}
Given any $u_0\in C^1(\R^n)\cap L^\infty(\R^n)$,  there exists a unique mild solution $u$ to problem \eqref{heatnonlineareq} with $t_1=+\infty$. Moreover,  if $u_0\in C^2(\R^n)$,  then the solution $u$ has the property: for $s>t_0$,  there exists $C$ independent of $t$ such that
\begin{equation}\label{ess}
\sup_{t\in(t_0, s)}\|u(t, \cdot)\|_{C^{2}(\R^n)}
 \leq C\Big(\Big[(s-t_0)+(s-t_0)^{1/2}\Big]\|G(0, 0, 0, x, t)\|_{L^{\infty}((t_0, s)\times\R^n)}+\|u_0\|_{C^2(\R^n)}\Big).
\end{equation}
\end{prop}

\begin{proof}
Notice that this proposition is equivalent to that problem \eqref{heatnonlineareq} has a unique solution $u$ for any $t_1\in (t_0, +\infty)$.
According to the definition in \eqref{demil},  $u$ is a mild solution of \eqref{heatnonlineareq} if and only if $u$ is a fixed point of the map $\mathcal{N}_{G, u_0}$. Thus we will study the existence and uniqueness of fixed point for the map $\mathcal{N}_{G, u_0}$.

Using Proposition \ref{prop7} and the assumptions on $G$,  we have that $\mathcal{N}_{G, u_0}$ is an operator from the space $L^\infty[(t_0, t_1); C^{2}(\R^{n})]$ to itself.
We claim that the map $\mathcal{N}_{G, u_0}$ is a contraction when $T_1:=t_1-t_0$ is small enough.

Indeed,  by Proposition \ref{prop6},  the assumptions on $G$ in (\ref{delpf}) and some direct computations,   we have that for any $u, w\in L^\infty[(t_0, t_1); C^{2}(\R^{n})]$,
\begin{align*}
&\big|\mathcal{N}_{G, u_0}[u](t, x)-\mathcal{N}_{G, u_0}[w](t, x)\big|
\\[2mm]
&\leq \int^t_{t_0}\int_{\R^n}\abs{p_n(t-\tau, x-y)}\cdot
\Big{|}G[\Delta u,  \nabla u,  u](\tau, y)-G[\Delta w,  \nabla w,  w](\tau, y)\Big{|}{\mathrm d}y{\mathrm d}\tau
\\[2mm]
&\leq \sigma T_1C_n\Big(\|\Delta u-\Delta w\|_{L^{\infty}((t_0, t_1)\times\R^n) }+\|\nabla u-\nabla w\|_{L^{\infty}((t_0, t_1)\times\R^n )}+\|u-w\|_{L^{\infty}((t_0, t_1)\times \R^n)}\Big),
\end{align*}
for a.e. $x\in\R^n$,  $t\in (t_0, t_1)$,  where $C_n$ is a positive number  which only depends on $n$.
Hence,  we can choose $T_1\leq\frac{1}{2\sigma C_n}$ such that
\begin{align*}
&\big\|\mathcal{N}_{G, u_0}[u](t, x)-\mathcal{N}_{G, u_0}[w](t, x)\big\|_{L^{\infty}((t_0, t_1)\times\R^n )}
\\[2mm]
& \leq \frac{1}{2}\Big(\|\Delta u-\Delta w\|_{L^{\infty}((t_0, t_1)\times\R^n )}+\|\nabla u-\nabla w\|_{L^{\infty}((t_0, t_1)\times\R^n )}+\|u-w\|_{L^{\infty}((t_0, t_1)\times\R^n )}\Big).
\end{align*}
Thus,  the map $\mathcal{N}_{G, u_0}$ is a contraction.

Next we consider the following iterative sequence $$u_m(t, x)=(\mathcal{N}_{G, u_0})^m[0](t, x), $$ by Banach Fixed Theorem,  then there is $u\in L^\infty[(t_0, t_1); C^{2}(\R^{n})]$ such that $u_m(x, t)$ converges to $ u(x, t)$ in the space $L^\infty[(t_0, t_1); C^{2}(\R^{n})]$ and $\mathcal{N}_{G, u_0}
(u)=u$,  which is the unique  solution of  problem \eqref{heatnonlineareq} with $t_1=t_0+\frac{1}{2\sigma C_n}$.

Let $u^1(t, x)$ denote the mild solution of \eqref{heatnonlineareq},  which can be obtained by the above arguments.
We consider the following initial problem:
\begin{equation*}
\left\{
\begin{array}{l}
u_t+(-\Delta)^2u=G[\Delta u, \nabla u, u, x,  t]  \hspace{1cm}\text{in}\ \big(t_0+\frac{1}{2\sigma C_n},  t_0+\frac{1}{\sigma C_n}\big)\times\R^n,
\\[2mm]
u\big(t_0+\frac{1}{2\sigma C_n},  x\big)=u^1\big(t_0+\frac{1}{2\sigma C_n},  x\big)\qquad \text{in}\ \R^n.
 \end{array}
\right.
\end{equation*}
Repeating the previous arguments,  we can get that the above equation has a unique solution $u^2(t, x)$.
Thus we can extend the time interval of existence to $\big(t_0, t_0+\frac{l}{2\sigma C_n}\big)$ with any positive integer $l\geq 1$ step by step.
In this process,  we used the fact that the initial  function is belonging to $C^1(\R^n)\cap L^\infty(\R^n)$ on each step.
Indeed,  the solutions are well-defined bounded continuous functions of $t$ and $x$,  and are continuously differentiable in $x$,
due to the estimate \eqref{es2} with $\vartheta=1$ in Proposition \ref{prop7}.
The estimate \eqref{ess} comes from Proposition \ref{prop7}.
\end{proof}

\subsection{A priori estimate of a linear problem involving the operator $L$}
Recall the functions $\rho$, $\gamma_n$ and $h$ in \eqref{drho1}-\eqref{assumeh}.
By defining
$$
\|\psi\|_{C_{\Phi}\left((t_1, t_2)\times(0, \infty)\right)}:=\left\| \frac{\psi}{\Phi}\right\|
_{L^\infty((t_1, t_2)\times(0, \infty))},
$$
we choose a set consisting of continuous functions as the following
\begin{align}
& C_{\Phi}((t_1, t_2)\times(0, \infty))
\nonumber\\[2mm]
&:=\left\{
\psi \ :\
\psi(t, r)=0, \ \forall\,  (t, r)\in(t_1, t_2)\times\left(0, \frac{\delta_0}{2}\right), \ \text{and}\ \|\psi\|_{C_{\Phi}((t_1, t_2)\times(0, \infty))}<+\infty
\right\},
\label{eq3.3}
\end{align}
where $t_1<t_2<0$,   $\delta_0>0$ is a small fixed number,  the functions $\Phi(t, r)$ and is given by (\ref{dpsi}).

Let us consider the following Cauchy problem:
\begin{equation}\label{eq3.5}
\left\{
\begin{array}{l}
\psi_t=L[\psi]+g(t, r)\qquad  \ \text{in}\ (s, -T]\times(0, \infty),
\\[2mm]
\psi(s, r)=0,\hspace{1.85cm}  \forall\, r\in (0, +\infty),
\\[2mm]
\partial^j_r\psi(t, 0)=0, \hspace{1.55cm} \forall\, t\in (s, -T]\ \text{and}\ j=1, 3,
 \end{array}
\right.
\end{equation}
where $g\in C_{\Phi}\left((-\infty, -T]\times(0, \infty)\right)$,  $T>0$ and $s+1<-T$.
Notice that problem \eqref{eq3.5} has a unique solution $\psi^s(t, r)$.
Indeed,  by the definition of the operator $L[\psi]$ in \eqref{eq2.13},  we can rewrite $L[\psi]$ as follows
\begin{align*}
L[\psi]=&-(\Delta)^2\psi+2W''\big(z(t, r)\big)\Delta \psi+2W'''\big(z(t, r)\big)\Delta z(t, r)\psi
+W^{(4)}\big(z(t, r)\big)\abs{\nabla z(t, r)}^2\psi
\\[2mm]
&+2W'''\big(z(t, r)\big)\nabla z(t, r)\cdot\nabla \psi
-\Big[\big|W''\big(z(t, r)\big)\big|^2-W'\big(z(t, r)\big)W'''\big(z(t, r)\big)\Big]\psi,
\end{align*}
where we recall that approximate solution
$$
z(t, r)=\left(\omega\big(r-\rho(t)\big)+\frac{(n-1)(n-3)}{r^2}  \widetilde{\omega}\big(r-\rho(t)\big)\right)\chi\big(r\big)+\chi\big(r\big)-1,
$$
 and the functions $\chi(r)$ and $\rho(t)$ are given by \eqref{dcut-off} and \eqref{drho1} respectively. According to Lemma \ref{lem1} and the assumptions on $W(s)$ in \eqref{eq1.4},
 those coefficients in front of $\Delta \psi$,  $\nabla\psi$ and $\psi$ are all smooth and bounded for $x\in\R^n$ and $t<-2$.
Thus we know that problem \eqref{eq3.5} is uniquely solvable by applying Proposition \ref{prop8}, and denote this solution by $\psi^s(t, r)$.
Moreover,  we have that $\psi^s(t, r)=0$ when $r\in(0,\frac{\delta_0}{2})$. Indeed, we consider
\begin{equation*}
\psi^{e}(t,r):=\left\{
\begin{array}{l}
\psi^s(t,r),\qquad \text{if}\ r\in(0,\frac{\delta_0}{2}]; \\[2mm]
0,\hspace{1.65cm} \text{if}\ r\in(\frac{\delta_0}{2},+\infty),
 \end{array}
\right.
\end{equation*}
which satisfies that
\begin{equation*}
\left\{
\begin{array}{l}
\psi_t=-(\Delta)^2\psi+4\Delta\psi-4\psi\qquad  \ \text{in}\ (s, -T]\times(0,+\infty),
\\[2mm]
\psi(s, r)=0,\hspace{3.1cm}  \forall\, r\in (0, +\infty),
\\[2mm]
\partial^j_r\psi(t, 0)=0, \hspace{2.83cm} \forall\, t\in (s, -T]\ \text{and}\ j=1, 3,
 \end{array}
\right.
\end{equation*}
where $n\geq4$ and we have used the fact that $z(t,r)=-1$ when $r\in(0,\frac{\delta_0}{2}]$.
According to the proposition \ref{prop8}, the  above equation has a unique mild solution $\psi(t,r)=0$.
Hence $\psi^e(t,r)=0$, that is $\psi^s(t, r)=0$ when $r\in(0,\frac{\delta_0}{2}]$.

We next establish a priori estimate for the solutions of problem (\ref{eq3.5}).
\begin{lem}\label{lem5}
Let $g\in C_{\Phi}((s, -T)\times(0, \infty))$ and $\psi^s$ be a solution of problem
$(\ref{eq3.5})$ which satisfies the orthogonality condition
\begin{equation}\label{eq3.6}
  \int_0^\infty \psi^s(t, r)\omega'\big(r-\rho(t)\big)r^{n-1}{\mathrm d}r=0,  \qquad s<t<-T.
\end{equation}
Then there exists a uniform constant $T_0>T$ such that for any $t\in (s, -T_0]$,  the following estimate is valid
\begin{equation}\label{eq3.7}
  \sum_{l=0}^3\|\partial_r^l\psi^s\|_{C_{\Phi}((s, t)\times(0, \infty))}\leq C\|g\|_{C_{\Phi}((s, t)\times(0, \infty))},
\end{equation}
where $C>0$ is a uniform positive constant independent of $t$ and $s$.
\end{lem}

The proofs of this lemma consist of tedious analysis,  which will be given in the sequel.

\subsubsection{Proof of Lemma \ref{lem5}} Here are the details.
\begin{proof}
We prove the above lemma by contradiction. Assume that there exist two sequences
$\{s_i\}$ and $\{t_i\}$ such that $s_i+1<t_i< 0$ and $s_i\rightarrow -\infty$,   $t_i\rightarrow -\infty$
when the sub-index $i$ goes to infinity.
We assume that: for each given $i$, there exists $g^{i}\in C_{\Phi}((s_i, t_i)\times(0, \infty))$,  such that
\begin{equation}\label{eq3.8}
  \begin{aligned}
 \sum_{l=0}^3\big\|\partial_r^l\psi^i\big\|_{C_{\Phi}((s_i, t_i)\times(0, \infty))}=1,
\end{aligned}
\end{equation}
and
\begin{equation}\label{eq8}
\|g^{i}\|_{C_{\Phi}((s_i, t_i)\times(0, \infty))}\rightarrow 0,
\qquad \text{as}\ i\rightarrow+\infty,
\end{equation}
where $\psi^i$ is the solution of problem (\ref{eq3.5})-(\ref{eq3.6}) with $g=g^i$,  $s=s_i$ and $-T=t_i$.

\textbf{Assertion}:
\emph{
For any $R>0$,  we have that
\begin{equation}\label{eq3.9}
\lim_{i\rightarrow\infty}\sum_{l=0}^3\left\|\frac{\partial^l_r\psi^i}{\Phi}\right\|_{L^\infty\big(A^{(s_i, t_i)}_{R}\big)}=0,
\end{equation}
where $\Phi(t, r)$ is given by $(\ref{eq2.13})$ and $A^{(s_i, t_i)}_{R}$ is defined by
$$
A^{(s_i, t_i)}_{R}:=\Big\{(t, r)\in (s_i, t_i)\times (0, \infty): \ \abs{r-\gamma_n(t)}<R+1\Big\}
$$
and the function $\gamma_n(t)$ is defined by \eqref{dgamma-n}.
}

The proof of (\ref{eq3.9}) will be provided in Section \ref{proofAssertion}.
We first accept the validity of (\ref{eq3.9}).
According to \eqref{eq3.8},  we have that there exists $l\in\{0, 1, 2, 3\}$ such that
 $$\big\|\partial_r^l\psi^i\big\|_{C_{\Phi}((s_i, t_i)\times(0, \infty))}\geq\frac{1}{4}.$$
Without loss of generality,  we assume that $\|\psi^i\|_{C_{\Phi}((s_i, t_i)\times(0, \infty))}\geq\frac{1}{4}$. For the other situations,  the following
arguments can be adapted in a simple way.

Recall the definition of $C_{\Phi}((s_i, t_i)\times(0, \infty))$ in \eqref{eq3.3},  then we derive that
\begin{equation*}
  \|\psi^i\|_{C_{\Phi}((s_i, t_i)\times(0, \infty))}=\sup_{(t, r)\in(s_i, t_i)\times(0, \infty)}\frac{\abs{\psi^i(t, r)}}{\Phi(t, r)}\geq\frac{1}{4}.
\end{equation*}
Thus,  there exists $(\bar{t}_i, \bar{r}_i)\in (s_i, t_i)\times(0, \infty)$ such that
\begin{equation*}
 \frac{\abs{\psi^i(\bar{t}_i, \bar{r}_i)}}{\Phi(\bar{t}_i, \bar{r}_i)}\geq\frac{1}{5}.
\end{equation*}
Furthermore,  by (\ref{eq3.9}),  we have
\begin{equation}\label{j1}
  \lim_{i\rightarrow +\infty}\abs{\bar{r}_i-\gamma_{n}(\bar{t}_i)}=+\infty.
\end{equation}

Let us define
\begin{equation*}
\phi^i(\nu, y):=\frac{\psi^i\big(\nu+\bar{t}_i, y+\bar{y}_i+\gamma_n(\nu+\bar{t}_i)\big)}
{\Phi\big(\bar{t}_i, \bar{y}_i+\gamma_{n}(\bar{t}_i)\big)}, \qquad\text{where}\ \bar{y}_i=\bar{r}_i-\gamma_{n}(\bar{t}_i),
\end{equation*}
where the functions $\gamma_n(t)$ and $\Phi(t,r)$ are given by \eqref{dgamma-n} and \eqref{dpsi} respectively. Thus $\phi^i(\nu, y)$ satisfies the following problem
\begin{align}
\partial_\nu\phi^i
=&
-\phi_{yyyy}^{i}
-\frac{2(n-1)}{y+\bar{y}_i+\gamma_{n}(\nu+\bar{t}_i)}\phi_{yyy}^i
+\Bigg[2W''\big(\bar{z}(\nu, y)\big)-\frac{(n-3)(n-1)}{
\big(y+\bar{y}_i+\gamma_{n}(\nu+\bar{t}_i)\big)^2}\Bigg]\phi^i_{yy}
\nonumber\\[3mm]
&-\Big(W''\big(\bar{z}(\nu, y)\big)\Big)^2\phi^i+\Bigg[\frac{2(n-1)W''\big(\bar{z}(\nu, y)\big)}{y+\bar{y}_i+\gamma_{n}(\nu+\bar{t}_i)}
-\frac{(3-n)(n-1)}{\big(y+\bar{y}_i+\gamma_{n}(\nu+\bar{t}_i)\big)^3}\Bigg]\phi^i_{y}
\nonumber\\[3mm]
&+2W'''\big(\bar{z}(\nu, y)\big)\bar{z}_y\phi^i_y-W'''\big(\bar{z}(\nu, y)\big)W'\big(\bar{z}(\nu, y)\big)\phi+W^{(4)}\big(\bar{z}(\nu, y)\big)\abs{\bar{z}_y}^2\phi^i
\nonumber\\[3mm]
&+2W'''\big(\bar{z}(\nu, y)\big)\bar{z}_{yy}\phi^i
+\frac{2(n-1)}{y+\bar{y}_i+\gamma_{n}(\nu+\bar{t}_i)}W'''\big(\bar{z}(\nu, y)\big)\bar{z}_y\phi^i
+\phi^i_y\partial_\nu\gamma_{n}(\nu+\bar{t}_i)
\nonumber\\[3mm]
&+\frac{g^i\big(\nu+\bar{t}_i, y+\bar{y}_i+\gamma_{n}(\nu+\bar{t}_i)\big)}
{\Phi\big(\bar{t}_i, \bar{y}_i+\gamma_{n}(\bar{t}_i)\big)}
\qquad \text{in}\ B^{(s_i, t_i)}_i,
\label{eqj.2}
\end{align}
with the conditions
\begin{align}
\abs{\phi^i(0, 0)}\geq\frac{1}{5}, \qquad
\phi^i(s_i-\bar{t}_i, y)=0,\qquad \text{for}\ y\in \big(-\bar{y}_i-\gamma_{n}(\nu+\bar{t}_i),  +\infty\big),
 \end{align}
where
$$
\bar{z}(\nu, y):=z\big(\nu+\bar{t}_i, y+\bar{y}_i+\gamma_{n}(\nu+\bar{t}_i)\big)
$$
and the set $B^{(s_i, t_i)}_i$ is defined by $$B^{(s_i, t_i)}_i:=(s_i-\bar{t}_i, 0]\times(-\bar{y}_i-\gamma_{n}(\nu+\bar{t}_i),  +\infty).$$

Since $\lim\limits_{i\rightarrow +\infty} \abs{\bar{y}_i}=+\infty$ due to \eqref{eq3.9},  we  have that $\lim\limits_{i\rightarrow +\infty} \bar{y}_i=+\infty$ or $-\infty$ up to a subsequence.
Moreover,  notice that there exist two cases:

\noindent\textbf{Case $1$},  $\lim\limits_{i\rightarrow\infty}s_i-\bar{t}_i=-\infty$.

\noindent\textbf{Case $2$},  $\lim\limits_{i\rightarrow\infty}s_i-\bar{t}_i=-\varsigma$ for some positive constant $\varsigma>0$.
\\
Next we will prove that Lemma \ref{lem5} holds,  and the process will be
divided into the following two steps.

\textbf{Step one.}
We first consider \textbf{Case $1$},  and then have that $\phi^i\rightarrow\phi$ locally uniformly. Moreover,  $\abs{\phi(0, 0)}>\frac{1}{5}$ and $\phi$ satisfies
\begin{equation}\label{limeq}
\phi_{\nu}=-\phi_{yyyy}+2W''(1)\phi_{yy}-\big[W''(1)\big]^2\phi\qquad  \text{in}\ (-\infty, 0]\times\R.
\end{equation}

By the assumption on $\psi^i$ in \eqref{eq3.8},  we have that for all $(\nu, y)\in B^j_i$,
 \begin{equation}\label{esp0}
 \abs{\phi^i(\nu, y)}=\abs{\frac{\psi^i\big(\nu+\bar{t}_i, y+\bar{y}_i+\gamma_{n}(\nu+\bar{t}_i)\big)}
 {\Phi\big(\bar{t}_i, \bar{y}_i+\gamma_{n}(\bar{t}_i)\big)}}
\leq
\abs{\frac{\Phi\big(\nu+\bar{t}_i, y+\bar{y}_i+\gamma_{n}(\nu+\bar{t}_i)\big)}
{\Phi\big(\bar{t}_i, \bar{y}_i+\gamma_{n}(\bar{t}_i)\big)}},
 \end{equation}
 where the sets $B_i^j$ with  $j=0, 1$,  are defined by
\begin{equation}\begin{aligned}\label{dB}
B_i^0:&=(s_i-\bar{t}_i, 0]\times\big(-\bar{y}_i-\gamma_{n}(\nu+\bar{t}_i), \
\delta_0-\gamma_n(\nu+\bar{t}_i)-\bar{y}_i\big],
\\[2mm]
 B_i^1:&=(s_i-\bar{t}_i, 0]\times\big(\delta_0-\gamma_n(\nu+\bar{t}_i)-\bar{y}_i, \
+\infty\big).
\end{aligned}\end{equation}

Next we estimate the right hand side of the above inequality \eqref{esp0}.
There exist the following different situations:
\\
\noindent{\bf (1).}
When $j=1$ and $\bar{r}_i=\bar{y}_i+\gamma_{n}(\bar{t}_i)>\delta_0$,  by the definitions of $\gamma_n(t)$ and $\Phi(t, r)$ in \eqref{dgamma-n} and \eqref{dpsi},  we have that
\begin{align*}
\abs{\frac{\Phi\big(\nu+\bar{t}_i, y+\bar{y}_i+\gamma_{n}(\nu+\bar{t}_i)\big)}
{\Phi\big(\bar{t}_i, \bar{y}_i+\gamma_{n}(\bar{t}_i)\big)}}
&\leq \frac{\hbar(\nu+\bar{t}_i)}{\hbar(\bar{t}_i)}
\left(\frac{1+\abs{\bar{y}_i-\frac{1}{3\alpha}\log\abs{\bar{t}_i}}}
{1+\abs{y+\bar{y}_i-\frac{1}{3\alpha}\log\abs{\nu+\bar{t}_i}}}\right)^p
\\[2mm]
&\leq  C \frac{\hbar(\nu+\bar{t}_i)}{\hbar(\bar{t}_i)} \Big(1+\abs{y}+\log\Big[\frac{\abs{\nu+\bar{t}_i}}{\abs{\bar{t}_i}}\Big]\Big)^p
\\[2mm]
&\leq C\left(\abs{y}+1\right)^p,
\end{align*}
where the function $\hbar(t)$ is defined by
\begin{equation}\label{dehat}
  \hbar(t):=\abs{t}^{-\frac{1}{2}}\log\abs{t},
\end{equation}
and $C$ is a positive constant which depends on $n$ and $\alpha$.

\noindent{\bf (2).}
When $j=1$ and $\frac{\delta_0}{2}<\bar{r}_i\leq\delta_0$,  by the same argument as above,  we have that
 \begin{equation*}
\begin{aligned}
&\abs{\frac{\Phi\big(\nu+\bar{t}_i, y+\bar{y}_i+\gamma_n(\nu+\bar{t}_i)\big)}
{\Phi\big(\bar{t}_i, \bar{y}_i+\gamma_n(\bar{t}_i)\big)}}\leq C.
\end{aligned}
\end{equation*}
\\
\noindent{\bf (3).}
When $j=0$,  by the definitions of $\gamma_n(t)$ and $\Phi(t, r)$ in \eqref{dgamma-n} and \eqref{dpsi},  we have that
\begin{equation*}
\begin{aligned}
\abs{\frac{\Phi\big(\nu+\bar{t}_i, y+\bar{y}_i+\gamma_{n}(\nu+\bar{t}_i)\big)}
{\Phi\big(\bar{t}_i, \bar{y}_i+\gamma_{n}(\bar{t}_i)\big)}}
&\leq \frac{\hbar(\nu+\bar{t}_i)}{\hbar(\bar{t}_i)}
\left(1+\abs{\bar{y}_i-\frac{\alpha}{3}\log\abs{\bar{t}_i}}\right)^p
\\[2mm]
&\leq C(1+\abs{y})^p,
\end{aligned}
\end{equation*}
where we have used the fact that $\abs{\bar{y}_i-\frac{\alpha}{4}\log\abs{t_i}}\leq \abs{y}$ since $(\nu, y)\in B_i^0$ given by \eqref{dB}.

Combining above estimates,  we derive that
\begin{equation}\label{esp}
  \abs{\phi^i(\nu, y)}\leq C(\abs{y}+1)^p,
\end{equation}
for all $(\nu, y)\in B_i^j$ with $j=0, 1$,  where $p\in(n, n+1)$ and $C$ is a positive constant only depending on $\alpha$ and $n$.
Similarly,  we have that
\begin{equation}\label{esp1}
   \abs{\phi^i_y(\nu, y)}\leq C(\abs{y}+1)^p, \quad \abs{\phi^i_{yy}(\nu, y)}\leq C(\abs{y}+1)^p\quad \text{and}\quad \abs{\phi^i_{yyy}(\nu, y)}\leq C(\abs{y}+1)^p,
\end{equation}
for all $(\nu, y)\in B_i^j$ and $C$ is a positive constant which only depends on $\alpha$ and $n$.
Notice that
$$
\cup_{j=0}^1B_i^j=\big(s_i-\bar{t}_i, 0\big]\times\big(-\bar{y}_i-\gamma_n(\nu+\bar{t}_i),\,   +\infty\big)
$$
due to \eqref{dB}. Owing to \eqref{esp} and \eqref{esp1},  we have that
$$
\phi^i\rightarrow \phi\qquad \mbox{in}\ L_{\text{loc}}^2\big[(-\bar{\varsigma}, 0); H^2_{\text{loc}}(\R)\big],
$$
where $\bar{\varsigma}=\varsigma$ or $\infty$. Moreover,  we note that there holds that
$$
\lim\limits_{i\rightarrow+\infty}-\gamma_n(\nu+\bar{t}_i)-\bar{y}_i=M,
$$
up to a subsequence,  where $M=-\infty$ or $M\in(-\infty, 0]$.
For the second case,  that is $M\in(-\infty, 0]$,  we can get that $\phi$ satisfies that
 \begin{equation*}
\phi_{\nu}=-\phi_{yyyy}+2W''(1)\phi_{yy}-\big[W''(1)\big]^2\phi\qquad  \text{in}\ (-\infty, 0]\times[0, +\infty).
\end{equation*}
Then,  we can consider the function
\begin{equation*}
  \widetilde{\phi}(\nu, x):=\phi(\nu, x)\quad \text{when}\ x\geq0,  \quad \text{and} \quad \widetilde{\phi}(\nu, x):=\phi(\nu, -x)\quad \text{when}\ x<0,
\end{equation*}
which is the even extension of $\phi$ with respect to $x$. It is easy to check that $ \widetilde{\phi}$ satisfies equation \eqref{limeq}.

For \textbf{Case $2$}: $\lim\limits_{i\rightarrow\infty}s_i-\bar{t}_i=-\varsigma$ for some  constant $\varsigma>0$,
we have $\phi^i\rightarrow\phi$ locally uniformly,  and $\phi$ satisfies
\begin{equation*}
\left\{
\begin{array}{l}
\phi_{\nu}=-\phi_{yyyy}+2W''(1)\phi_{yy}-\big[W''(1)\big]^2\phi\qquad \text{in}\ (-\varsigma, 0]\times\R,
\\[3mm]
\phi(-\varsigma, y)=0, \qquad \text{for all}\ y\in\R.
\end{array}
\right.
\end{equation*}
By Proposition \ref{prop8},  the above equation has a unique solution $\phi\equiv0$.
However,  we have that
$$
\abs{\phi(0, 0)}=\lim\limits_{i\rightarrow\infty}\abs{\phi^i(0, 0)}\geq\frac{1}{5}
$$
in \eqref{eqj.2},  which leads to a contradiction.
Hence the \textbf{Case $2$} cannot happen.

\textbf{Step two.}  We claim that
\begin{equation}\label{des}
\phi(\nu, y)\equiv0,
\end{equation}
for all $(\nu, y)\in(-\infty, 0]\times\R$. This result contradicts with $\abs{\phi(0, 0)}>\frac{1}{5}$,  thus we can derive that Lemma \ref{lem5} holds.

In the rest part of this proof, the job is to show that the above conclusion in \eqref{des} holds.
Recall that $\alpha=\sqrt{W''(1)}>0$,  we then consider the following parabolic equation
\begin{equation*}
  u_\nu=-u_{yyyy}+2\alpha^2 u_{yy}, \qquad \forall\,  (\nu, y)\in (0, +\infty)\times\R.
\end{equation*}
Using the Fourier transformation for $y$,  the above equation has a formula solution
\begin{equation}\label{11}
\mathbf{Q}(\nu, y)=\frac{1}{2\pi}\int_\R \exp\Big\{-\nu\big(\abs{\xi}^4+2\alpha^2\abs{\xi}^2\big)\Big\}e^{i\xi y}{\mathrm d}\xi.
\end{equation}
Hence,  for any $T>0$ and $f\in L^\infty((-T, 0)\times\R)$,  the following initial value problem
\begin{equation*}
  u_\nu=-u_{yyyy}+2\alpha^2 u_{yy}-\alpha^4u+f(\nu, y), \qquad  (\nu, y)\in (-T, 0)\times\R,  \qquad u(-T, y)=0,  \qquad \forall\, y\in\R,
\end{equation*}
has a solution of the form
\begin{equation}\label{f1}
u(\nu, y)=\frac{1}{2\pi}\int_{0}^{\nu-T_1}\int_{\R}e^{-\alpha^4\tau}\mathbf{Q}(\tau, x)f(\nu-\tau, y-x){\mathrm d}x{\mathrm d}\tau,  \qquad \forall\, (\nu, y)\in(-T, 0)\times\R.
\end{equation}

According to formula \eqref{f1} and equation \eqref{eqj.2},  we have that
\begin{equation}\label{dfs}
\phi^i(\nu, y)=\int_{0}^{\nu-s_i+\bar{t}_i}\int_{\R}e^{-\alpha^4\tau}\mathbf{Q}(\tau, x)f_i\big(\phi^i(\nu-\tau, y-x),\, \nu-\tau,  y-x\big){\mathrm d}x{\mathrm d}\tau,
\end{equation}
 for all $(\nu, y)\in(s_i-\bar{t}_i, 0)\times\R$,  where the function $f_i(\phi^i, \nu, y)$ is defined as follows
\begin{align*}
f_i(\phi^i, \nu, y):=&-\frac{2(n-1)}{y+\bar{y}_i+\gamma_{n}(\nu+\bar{t}_i)}\phi_{yyy}^i-\frac{(n-3)(n-1)}{
\big(y+\bar{y}_i+\gamma_{n}(\nu+\bar{t}_i)\big)^2}\phi^i_{yy}
\\[2mm]
&+\Bigg[\frac{2(n-1)W''\big(\bar{z}(\nu, y)\big)}{y+\bar{y}_i+\gamma_{n}(\nu+\bar{t}_i)}
-\frac{(3-n)(n-1)}{\big(y+\bar{y}_i+\gamma_{n}(\nu+\bar{t}_i)\big)^3}\Bigg]\phi^i_{y}
\\[2mm]
&+2W'''\big(\bar{z}(\nu, y)\big)\bar{z}_y\phi^i_y-W'''\big(\bar{z}(\nu, y)\big)W'\big(\bar{z}(\nu, y)\big)\phi+W^{(4)}\big(\bar{z}(\nu, y)\big)\abs{\bar{z}_y}^2\phi^i
\\[2mm]
&+2W'''\big(\bar{z}(\nu, y)\big)\bar{z}_{yy}\phi^i+2\frac{n-1}{y+\bar{y}_i+\gamma_{n}(\nu+\bar{t}_i)}W'''\big(\bar{z}(\nu, y)\big)\bar{z}_y\phi^i
+\phi^i_y\partial_\nu\gamma_{n}(\nu+\bar{t}_i)
\\[2mm]
&+\frac{g^i\big(\nu+\bar{t}_i, y+\bar{y}_i+\gamma_{n}(\nu+\bar{t}_i)\big)}
{\Phi\big(\bar{t}_i, \bar{y}_i+\gamma_{n}(\bar{t}_i)\big)}
\\[2mm]
&+2\Big[W''\big(\bar{z}(\nu, y)\big)-W''(1)\Big]\phi^i_{yy}
-\Big(\big[W''\big(\bar{z}(\nu, y)\big)\big]^2-\big[W''(1)\big]^2\Big)\phi^i,
\end{align*}
where $\bar{z}(\nu, y)=z\big(\nu+\bar{t}_i,  y+\bar{y}_i+\gamma_{n}(\nu+\bar{t}_i)\big)$,  the function $z(\nu, r)$ is given by \eqref{dz}.

Using \eqref{eq3.3},  \eqref{eq8},  \eqref{esp},  \eqref{esp1},  the same arguments in proof of Lemma \ref{lem2},  we have that
\begin{align*}
\abs{f_i\big(\phi^i(\nu, y), \nu, y\big)}\leq& C(1+\abs{y})^p\Bigg[\exp\big\{-\alpha\abs{y+\bar{y}_i}\big\}
  +\|g^{i}\|_{C_{\Phi}((s_i, t_i)\times(0,+\infty))}
  \\[2mm]
&\qquad\qquad+\frac{1}{\left[-(\nu+\bar{t}_i)\right]^{3/4}}
+\sum_{l=1}^3\big(y+\bar{y}_i+\gamma_{n}(\nu+\bar{t}_i)\big)^{-l}\Bigg]
\overline{\chi}_{\left\{y>\frac{\delta_0}{2}-\gamma_n(\nu+\bar{t}_i)-\bar{y}_i\right\}},
\end{align*}
for all $(\nu, y)\in (s_i-\bar{t}_i, 0]\times(-\bar{y}_i-\gamma_n(\nu+\bar{t}_i), +\infty)$,  where $C$ depends on $\alpha$ and $n$,  and $\overline{\chi}_{A}$ denotes the characteristic function of the set $A$.

 Moreover a straightforward shift of contour gives that there exists a positive constant $C$ such that
$$\abs{\mathbf{Q}(\nu, y)}\leq C \nu^{-\frac{1}{4}}\exp\left\{-\frac{\abs{y}}{\nu^{1/4}}\right\}, $$
or see $(3.2)$ in \cite{BKT}.
Thus,  using \eqref{dfs} and the above inequality,  we have that
\begin{align}\label{13}
\abs{\phi^i(\nu, x)}&\leq C\int_0^{\nu-s_i+\bar{t}_i}\int^{+\infty}_{\frac{\delta_0}{2}-\bar{y}_i-\gamma_{n}(\nu+\bar{t}_i)}\tau^{-\frac{1}{4}}
\exp\Bigg\{-\frac{\abs{x-y}}{\tau^{1/4}}-\alpha^4\tau\Bigg\}
\Bigg[\exp\big\{-\alpha\abs{y+\bar{y}_i}\big\}
  \\[2mm]
&\qquad\quad+\frac{1}{\left[-(\nu+\bar{t}_i-\tau)\right]^{3/4}}
+\|g^{i}\|_{C_{\Phi}((s_i, t_i)\times\R)}
+\sum_{l=1}^3\big[y+\bar{y}_i+\gamma_{n}(\nu+\bar{t}_i)\big]^{-l}\Bigg]\big(1+\abs{y}\big)^p
{\mathrm d}y {\mathrm d}\tau.
\nonumber
\end{align}
Thus, by  Lebesgue dominated convergence theorem, the facts
$$
\abs{\bar{y}_i} \rightarrow \infty,
\qquad
\bar{t}_i\rightarrow-\infty
\qquad\mbox{and}\qquad
\|g^{i}\|_{C_{\Phi}((s_i, t_i)\times(0,+\infty))}\rightarrow 0,
$$
when $i$ goes to infinity in \eqref{eq8},  we can obtain \eqref{des}.
\end{proof}

\subsubsection{Proof of \textbf{Assertion}}\label{proofAssertion}

We will prove (\ref{eq3.9}) by contradiction arguments in the following five steps.
We assume that (\ref{eq3.9}) is not valid. Then there exists a sequence $\{i_m\in \mathbb{N}\}$ with $\lim\limits_{m\rightarrow \infty} i_m=\infty$ and $R>0$ such that there exists $l\in\{0,  1, 2, 3\}$ satisfying
\begin{equation*}
 \left\|\frac{\partial^l_r\psi^{i_m}}{\Phi}\right\|_{L^\infty\left(A^{\left(s_{i_m}, t_{i_m}\right)}_{R}\right)}>0.
\end{equation*}
Without loss of the generality,  we assume that
$$
\Big\|\frac{\psi^{i_m}}{\Phi}\Big\|_{L^\infty\left( A^{\left(s_{i_m}, t_{i_m}\right)}_{R}\right)}>0.
$$
For the other situations,
the following arguments are similar.

Let $(\hat{t}_{i_m},  r_{i_m})\in A^{(s_{i_m}, t_{i_m})}_{R}=\big\{(t, r)\in (s_{i_m}, t_{i_m})\times \R: \abs{r-\gamma_{n}(t)}<R+1\big\}$ such that there exist $\delta>0$ and
\begin{equation}\label{eq3.12}
\abs{\frac{\psi^{i_m}(\hat{t}_{i_m}, r_{i_m})}{\Phi(\hat{t}_{i_m}, r_{i_m})}}>\delta>0.
\end{equation}
Let us introduce the following change of variables
\begin{equation*}
\nu=t-\hat{t}_{i_m},\quad  y=r-y_{i_m}-\rho(t+\hat{t}_{i_m})\quad \text{and}\quad y_{i_m}=r_{i_m}-\gamma_{n}(\hat{t}_{i_m}),
\end{equation*}
and set
\begin{equation}\label{eq3.14}
  \widehat{\psi}^{i_m}(\nu, y):=\frac{\psi_{i_m}\big(\nu+\hat{t}_{i_m},\, y+y_{i_m}+\rho(\nu+\hat{t}_{i_m})\big)\,}
  {\Phi\big(\hat{t}_{i_m}, y_{i_m} +\rho(\hat{t}_{i_m})\big)},
\end{equation}
where we recall that $\rho(t)=\gamma_n(t)+h(t)$ defined in \eqref{drho1}. Hence,  there exists $y_0$ such that
$$
\lim\limits_{m\rightarrow\infty}y_{i_m}=y_0
\quad \mbox{and}\quad
\abs{y_0}<R+2.
$$
According to (\ref{eq3.5}),  we have that $\widehat{\psi}^{i_m}$ satisfies the following problem
\begin{align}
\widehat{\psi}^{i_m}_\nu
=&-\widehat{\psi}^{i_m}_{yyyy}
-\frac{2(n-1)}{y+y_{i_m}+\rho(\nu+\hat{t}_{i_m})}\widehat{\psi}^{i_m}_{yyy}
+\Bigg[ 2W''\big(\hat{z}(\nu, y)\big)-\frac{(n-3)(n-1)}{
\big(y+y_{i_m}+\rho(\nu+\hat{t}_{i_m})\big)^2}\Bigg] \widehat{\psi}^{i_m}_{yy}
\nonumber\\[2mm]
&-\Big(W''\big(\hat{z}(\nu, y)\big)\Big)^2\widehat{\psi}^{i_m}
+\Bigg[\frac{2(n-1)W''\big(\hat{z}(\nu, y)\big)}{y+y_{i_m}+\rho(\nu+\hat{t}_{i_m})}-\frac{(3-n)(n-1)}{\big(y+y_{i_m}+\rho(\nu+\hat{t}_{i_m})\big)^3}\Bigg]\widehat{\psi}^{i_m}_{y}
\nonumber\\[2mm]
&+2W'''\big(\hat{z}(\nu, y)\big)\hat{z}_{yy}\widehat{\psi}^{i_m}
+2\frac{n-1}{y+y_{i_m}+\rho(\nu+\hat{t}_{i_m})}W'''\big(\hat{z}(\nu, y)\big)\hat{z}_y\widehat{\psi}^{i_m}
\nonumber\\[2mm]
&+2W'''\big(\hat{z}(\nu, y)\big)\hat{z}_y\widehat{\psi}^{i_m}_y
-W'''\big(\hat{z}(\nu, y)\big)W'\big(\hat{z}(\nu, y)\big)\widehat{\psi}^{i_m}
+W^{(4)}\big(\hat{z}(\nu, y)\big)\abs{\hat{z}_y}^2
\widehat{\psi}^{i_m}
\nonumber\\[2mm]
&+\frac{g^{i_m}\big(\nu+\hat{t}_{i_m}, \,  y+y_{i_m}+\rho(\nu+\hat{t}_{i_m})\big)}
{\Phi\big(\hat{t}_{i_m}, y_{i_m}+\rho(\hat{t}_{i_m})\big)}
-\partial_\nu \rho(\nu+\hat{t}_{i_m})\widehat{\psi}^{i_m}_{y}
\qquad \text{in}\ \Gamma^{s_{i_m}, t_{i_m}},
\label{eq3.10}
\end{align}
with the conditions
\begin{align}
\abs{\widehat{\psi}^{i_m}(0, 0)}\geq\delta, \quad
\widehat{\psi}^{i_m}(s_{i_m}-\hat{t}_{i_m}, y)=0, \quad  y\in\big(-y_{i_m}-\rho(\nu+\hat{t}_{i_m}),  +\infty\big),
\label{eq3.10-111}
\end{align}
where
$$
\hat{z}(\nu, y)=z\big(\nu+\hat{t}_{i_m}, y+y_{i_m}+\rho(\nu+\hat{t}_{i_m})\big),
$$
and
$$
\Gamma^{s_{i_m}, t_{i_m}}:=\Big\{(\nu, y)\in \big(s_{i_m}-\hat{t}_{i_m},\, 0\big]\times \big(-y_{i_m}-\rho(\nu+\hat{t}_{i_m}),\, +\infty\big)\Big\}.
$$
As the same as previous part,
$\lim\limits_{m\rightarrow\infty}(s_{i_m}-\hat{t}_{i_m})=\hat{\varsigma}$,  where $\hat{\varsigma}$ equal to $-\infty$ or a negative number.

\textbf{Step} $\mathbf{1}$.
We first consider the case: $\lim\limits_{m\rightarrow\infty}(s_{i_m}-\hat{t}_{i_m})=-\infty$.
We have that $\widehat{\psi}^{i_m}\rightarrow\widehat{\psi}$ locally uniformly,  $\abs{\widehat{\psi}(0, 0)}>\delta>0$,  which $\widehat{\psi}$ satisfies the following equation
\begin{equation}\label{eq3.13}
\begin{aligned}
\widehat{\psi}_\nu=&-\widehat{\psi}_{yyyy}+2W''\big(\omega(y+y_0)\big)\widehat{\psi}_{yy}
+2W'''\big(\omega(y+y_0)\big)\omega'(y+y_0)\widehat{\psi}_y
-\Big[W''\big(\omega(y+y_0)\big)\Big]^2\widehat{\psi}
\\[3mm]
&+\Big[W''''\big(\omega(y+y_0)\big)\big(\omega'(y+y_0)\big)^2
+W'''\big(\omega(y+y_0)\big)\omega''(y+y_0)\Big]\widehat{\psi}, \qquad \text{in} \ (-\infty, 0]\times\R.
\end{aligned}
\end{equation}
According to (\ref{dpsi}),  (\ref{eq3.8}) and (\ref{eq3.14}),  by the definition of $\rho$ in \eqref{drho1},  similar as \eqref{esp},  we have
\begin{align}\label{eq3.15}
\nonumber\abs{\widehat{\psi}^{i_m}(\nu, y)}=&\abs{\frac{\psi_{i_m}\big(\nu+\hat{t}_{i_m}, y+y_{i_m}+\rho(\nu+\hat{t}_{i_m})\big)}
{\Phi\big(\hat{t}_{i_m}, y_{i_m}+\rho(\hat{t}_{i_m})\big)}}
\\[2mm]\nonumber
\leq&\abs{\frac{\Phi\big(\nu+\hat{t}_{i_m}, y+y_{i_m}+\rho(\nu+\hat{t}_{i_m})\big)}
{\Phi\big(\hat{t}_{i_m}, y_{i_m}+\rho(\hat{t}_{i_m})\big)}}
\\[2mm]
\leq &C\big(\alpha, n,  R, \|h\|_{L^\infty}\big) \big(1+\abs{y}\big)^p, \qquad
\forall\, (\nu, y)\in B_{\hat{t}_{i_m}, j},
\end{align}
for all $j=0, 1$,  where $\rho(t)=\gamma_n(t)+h(t)$ in \eqref{drho1} and the sets $B_{\hat{t}_{i_m}, j}$ are defined as the following
\begin{align*}
   B_{\hat{t}_{i_m}, 0}&:=\Big\{(\nu, y)\in (s_{i_m}-\hat{t}_{i_m}, 0]\times\R: 0 < y+y_{i_m}+\rho(\nu+\hat{t}_{i_m})\leq \delta_0\Big\},
\\[2mm]
  B_{\hat{t}_{i_m}, 1}&:=\Big\{(\nu, y)\in (s_{i_m}-\hat{t}_{i_m}, 0]\times\R: \delta_0 \leq y+y_{i_m}+\rho(\nu+\hat{t}_{i_m})<+\infty\Big\}.
\end{align*}
 We notice that
\begin{equation}\label{eq3.17}
\Gamma^{s_{i_m}, t_{i_m}}=\cup_{l=0}^{1}B_{\hat{t}_{i_m}, j}
=\big(s_{i_m}-\hat{t}_{i_m}, 0\big]\times\big(-y_{i_m}-\rho(\nu+\hat{t}_{i_m}),\, +\infty\big).
\end{equation}
Similarly,  we have that
\begin{equation}\label{f4}
 \sum_{l=1}^3\abs{\partial_y^l\widehat{\psi}^{i_m}(\nu, y)}\leq C (1+\abs{y})^p,
\end{equation}
 for all $(\nu, y)\in \Gamma^{s_{i_m}, t_{i_m}}$,  where $p\in(n, n+1)$ and $C$ depends on $ \alpha,  R$,  $n$ and $\|h\|_{L^\infty}$.

Since $\gamma_{n}(\nu+\hat{t}_{i_m})\rightarrow+\infty$ as $i$ goes to infinity,  by (\ref{f4}),  (\ref{eq3.10})-\eqref{eq3.10-111},  and $y_{i_m}\rightarrow y_0$,  we can get the limiting equation in \eqref{eq3.13}.
If $\hat{\varsigma}\in (-\infty,  0)$,  we have $\hat{\psi}(\hat{\varsigma}, y)=0$.
According to Proposition \ref{prop8},  we have that equation \eqref{eq3.13} has a unique solution $\hat{\psi}(\nu, y)\equiv0$,  which contradicts with $\abs{\hat{\psi}(0, 0)}>0$. Hence,  it only happens that $\hat{\varsigma}=-\infty$.
Thus,  we get the desired result.

\textbf{Step} $\mathbf{2}$.
We will prove the following orthogonality condition for $\widehat{\psi}$
\begin{equation}\label{eq3.16}
  \int_{\R}\widehat{\psi}(\nu, y)\omega'(y+y_0){\mathrm d}y=0, \qquad \text{for all}\ \nu\in(-\infty, 0].
\end{equation}
In fact,  according to (\ref{eq3.6}) and (\ref{eq3.14}),  we get that
\begin{align*}
  0&= \frac{\left[y_{i_m}+\rho(\nu+\hat{t}_{i_m})\right]^{1-n}}
  {\Phi\big(\hat{t}_{i_m}, y_{i_m}+\rho(\hat{t}_{i_m})\big)}
  \int_0^\infty \psi^{i_m}(\nu+\hat{t}_{i_m}, r)\omega'\big(r-\rho(\nu+\hat{t}_{i_m})\big)r^{n-1}{\mathrm d}r
 \\[2mm]
&=\int_{-y_{i_m}-\rho(\nu+\hat{t}_{i_m})}^\infty\widehat{\psi}^{i_m}(\nu, y)\omega'(y+y_{i_m})
  \left[\frac{y+y_{i_m}+\rho(\nu+\hat{t}_{i_m})}
  {y_{i_m}+\rho(\nu+\hat{t}_{i_m})}\right]^{n-1}{\mathrm d}y,
\end{align*}
where $\rho(t)=\gamma_n(t)+h(t)$ with $\|h(t)\|_{L^\infty}<1$ and the function $\omega$ is given by \eqref{eqq}.
And using  (\ref{eq3.15}),  (\ref{eq3.17}) and Lemma \ref{lem1},  we also get that
\begin{align*}
&\abs{\widehat{\psi}_{i_m}(\nu, y)\omega'(y+y_{i_m})}\leq C(R,  n,  \alpha,  \|h\|_{L^\infty})(1+\abs{y})^p
\exp\big\{-\alpha\abs{y+y_{i_m}}\big\}.
\end{align*}
Since $y_{i_m}\rightarrow y_0$ and $\hat{t}_{i_m}\rightarrow-\infty$,  as $m$ goes to infinity,
using Lebesgue Dominated Convergence Theorem,  we can  obtain that (\ref{eq3.16}) holds.

\textbf{Step} $\mathbf{3}$.
In this step,  we will prove the following  decay of $\widehat{\psi}(\nu, y)$:
there exists a constant $C$ of the form
$$
C=C\big(\alpha, n,  R, \|h\|_{L^\infty}\big)>0
$$
such that
\begin{equation}\label{eq3.18}
  \abs{\widehat{\psi}(\nu, y)}\leq \frac{C}{\left(1+\abs{y}\right)^2},  \qquad \forall\, (\nu, y)\in (-\infty, 0]\times\R.
\end{equation}

In fact,  for any $(\nu, y)\in B_{t_{i_m}, j}$,  by the definition of $\rho(t)$ in \eqref{drho1},   in view of the  proof of (\ref{eq3.15}),  we have
\begin{equation}\label{f3}
  \abs{\frac{g^{i_m}\big(\nu+\hat{t}_{i_m}, y+y_{i_m}+\rho(\nu+\hat{t}_{i_m})\big)}
{\Phi\big(\hat{t}_{i_m}, y_{i_m}+\rho(\hat{t}_{i_m})\big)}} \leq C \|g^{i_m}\|_{C_{\Phi}((s_{i_m},\, t_{i_m})\times(0,+\infty))}(1+\abs{y})^p, \
\end{equation}
for all $(\nu, y)\in \Gamma^{s_{i_m}, t_{i_m}}$ given by \eqref{eq3.17},  where $C$ depends on $\alpha, R,  n$ and $\|h\|_{L^\infty}$.

By formula \eqref{f1},  the solution of equation \eqref{eq3.10} has the form
\begin{equation}\label{f2}
  \widehat{ \psi}_{i_m}(\nu, y)=\int_{0}^{\nu-s_{i_m}+\hat{t}_{i_m}}\int_{\R}e^{-\alpha^4\tau}
  \mathbf{Q}(\tau, x)\hat{f}_{i_m}\big(\hat{\psi}^{i_m}(\nu-\tau, y-x), \nu-\tau, y-x\big){\mathrm d}x{\mathrm d}\tau,
\end{equation}
for all $(\nu, y)\in(s_{i_m}-\hat{t}_{i_m}, 0)\times\R$,  where $\alpha=\sqrt{W''(1)}$ and the function $\hat{f}_{i_m}$ is given by
\begin{align*}
\hat{f}_{i_m}=&2\Big[W''\big(\hat{z}(\nu, y)\big)-W''(1)\Big]\widehat{\psi}^{i_m}_{yy}
+2W'''\big(\hat{z}(\nu, y)\big)\partial_y\hat{z}(\nu, y)\widehat{\psi}^{i_m}_y
-\Big(\big[W''\big(\hat{z}(\nu, y)\big)\big]^2-\big[W''(1)\big]^2\Big)\widehat{\psi}^{i_m}
\\[3mm]
&+\Big\{\partial_{yy}W''\big(\hat{z}(\nu, y)\big)+W'''\big(\hat{z}(\nu, y)\big)\big[\partial_{yy}\hat{z}(\nu, y)-W'\big(\hat{z}(\nu, y)\big)\big]\Big\}\widehat{\psi}^{i_m}
\\[3mm]
&-\frac{2(n-1)}{y+y_{i_m}+\rho(\nu+\hat{t}_{i_m})}\widehat{\psi}^{i_m}_{yyy}
-\frac{(n-3)(n-1)}{\big(y+y_{i_m}+\rho(\nu+\hat{t}_{i_m})\big)^2}\widehat{\psi}^{i_m}_{yy}
\\[3mm]
&+\Bigg[\frac{2(n-1)W''\big(\hat{z}(\nu, y)\big)}{y+y_{i_m}+\rho(\nu+\hat{t}_{i_m})}
-\frac{(3-n)(n-1)}{\big(y+y_{i_m}+\rho(\nu+\hat{t}_{i_m})\big)^3}\Bigg]\widehat{\psi}^{i_m}_{y}
\\[3mm]
&+2\frac{n-1}{y+y_{i_m}+\rho(\nu+\hat{t}_{i_m})}W'''\big(\hat{z}(\nu, y)\big)\hat{z}_y\widehat{\psi}^{i_m}
+\frac{g^{i_m}\big(\nu+\hat{t}_{i_m}, y+y_{i_m}+\rho(\nu+\hat{t}_{i_m})\big)}
{\Phi\big(\hat{t}_{i_m}, y_{i_m}+\rho(\hat{t}_{i_m})\big)}
\\[3mm]
&-\partial_\nu \rho(\nu+\hat{t}_{i_m})\widehat{\psi}^{i_m}_{y}
\end{align*}
and $\hat{z}(\nu, y)=z\big(\nu+\hat{t}_{i_m}, y+y_{i_m}+\rho(\nu+\hat{t}_{i_m})\big)$.

According to the lemma \ref{lem1},  \eqref{eq3.3},  \eqref{f3},  \eqref{eq3.15} and \eqref{f4},  we have that
\begin{align*}
  \abs{\hat{f}_{i_m}}\leq &C(1+\abs{y})^p\Bigg[\exp\big\{-\alpha\abs{y+y_{i_m}}\big\}
  +\|g^{i_m}\|_{C_{\Phi}((s_{i_m}, t_{i_m})\times(0,+\infty))}
  \\[2mm]
  &\qquad\qquad+\frac{1}{\left[-(\nu+\hat{t}_{i_m})\right]^{3/4}}+\sum_{l=1}^3\left(y+y_{i_m}+\rho(\nu+\hat{t}_{i_m})\right)^{-l}\Bigg]
  \overline{\chi}_{\left\{y>\frac{\delta_0}{2}-\rho(\nu+\hat{t}_{i_m})-y_{i_m}\right\}},
\end{align*}
 for all $(\nu, y)\in \big(s_{i_m}-\hat{t}_{i_m},  0\big)\times(-\rho(\nu+\hat{t}_{i_m})-y_{i_m},+\infty)$,  where $C$ depends on $\alpha,  R,  n$ and $\|h\|_{L^\infty}$,  and $\overline{\chi}_{A}$ is the characteristic function of the set $A$.

Thus by \eqref{f2},  similar arguments as \eqref{13} and $\abs{y_{i_m}}\leq R+2$,  let $m$ goes to infinity,  and then we can get the desired result since $\|g^{i_m}\|_{C_{\Phi}((s_{i_m}, \bar{t}_{i_m})\times(0,+\infty))}\rightarrow0$ as index $m$ goes to infinity in \eqref{eq8} and the following fact that
\begin{align*}
&\int_{0}^{\nu+\check{t}_{i_m}}\int_{\R}e^{-\alpha^4\tau}\abs{\mathbf{Q}(\tau, y-x)}\big(1+\abs{x}\big)^p
\exp\big\{-\alpha\abs{x+y_{i_m}}\big\}{\mathrm d}x{\mathrm d}\tau
\\[2mm]
  &\leq C\int_{0}^{\nu+\check{t}_{i_m}}\int_{\R}\exp\left\{-\alpha^4\tau-\frac{1}{4}\abs{x}\right\}
  \exp\left\{-\frac{\alpha}{2}\abs{y-\tau^{1/4}x}\right\}{\mathrm d}x{\mathrm d}\tau
\\[2mm]
&\leq C\int_{0}^{\nu+\check{t}_{i_m}}\int_{\R}\exp\left\{-\alpha^4\tau-\frac{1}{4}\abs{x}\right\}
 \frac{1}{\left(1+\abs{y-\tau^{1/4}x}\right)^2}{\mathrm d}x{\mathrm d}\tau
 \\[2mm]
&\leq C\int_{0}^{\nu+\check{t}_{i_m}}\int_{\R}\exp\left\{-\alpha^4\tau-\frac{1}{4}\abs{x}\right\}
 \left(\frac{1+\abs{\tau^{1/4}x}}{1+\abs{y}}\right)^2{\mathrm d}x{\mathrm d}\tau
\\[2mm]
& \leq \frac{C}{\left(1+\abs{y}\right)^2},
\end{align*}
where $\check{t}_{i_m}:=\hat{t}_{i_m}-s_{i_m}$ and $C>0$ only depends on $\alpha$.

\textbf{Step} $\mathbf{4}$.
To proceed further,  we need  the following result. However, we can not find  the proof of this result in references. Here we give a proof in details.
\begin{lem}\label{lem2}
Considering the Hilbert space
$$H=\left\{\phi(y)\in H^2(\R):\int_{\R}\phi(y)\omega'(y){\mathrm d}y=0\right\}, $$
then the following inequality is valid
\begin{equation}\label{eqi}
\int_\R\abs{\phi''(y)-W''(\omega)\phi(y)}^2{\mathrm d}y\geq c\int_\R\abs{\phi(y)}^2{\mathrm d}y,   \qquad \forall\, \phi\in H,
\end{equation}
where $c>0$.
\end{lem}

\begin{proof}
For any $\phi\in H$,  since $\omega'>0$,  we can set $\phi(y)=\zeta_1(y)\omega'(y)$,  then
\begin{align*}
\int_\R\abs{\phi''(y)-W''(\omega)\phi(y)}^2{\mathrm d}y=&\int_\R\abs{\omega'''(y)\zeta_1(y)+2\zeta'_1(y)\omega''(y)+\zeta''_1(y)\omega'(y)
-W''(\omega)\zeta_1(y)\omega'(y)}^2{\mathrm d}y
\\[2mm]
=&\int_\R\abs{2\omega''(y)\zeta_1'(y)+\omega'(y)\zeta_1''(y)}^2{\mathrm d}y\geq 0,
\end{align*}
where  we have used $\omega'''=W''(\omega)\omega'$ and integration by parts in the last equality.

Hence we have
\begin{equation*}
  \int_\R\abs{\phi''(y)-W''(\omega)\phi(y)}^2{\mathrm d}y=0\quad \text{if and only if}\quad \ 2\omega''(y)\zeta_1'(y)+\omega'(y)\zeta_1''(y)=0.
\end{equation*}
All solutions of the last equation have the following  form
$$
\zeta_1(x)= c_1\int_0^x\frac{1}{(\omega'(y))^2}{\mathrm d}y+c_2, \qquad c_1, c_2\in\R.
$$
Since
$$
\phi(x)=\zeta_1(x)\omega'(x)\in H^2(\R)\qquad \mbox{and}\qquad
\int_{\R}\phi(y)\omega'(y){\mathrm d}y=0,
$$
we have that $c_1=0$ and $c_2=0$.

Now we assume that there exists a sequence $\{\phi_m\}\in H$ such that
\begin{equation}\label{esi}
\int_{\R}\abs{\phi_m}^2{\mathrm d}y=1\ \ \text{and}\ \ \int_{\R}\abs{\phi''_m-W''(\omega)\phi_m}^2{\mathrm d}y\leq\frac{1}{m}.
\end{equation}
Thus,  $\phi_m\rightharpoonup\phi$ in $H(\R)$ and $\phi_m\rightarrow\phi$ in $L^2(K)$ for any compact subset $K\subset\R$,  which implies that
$$0=\int_{\R}\phi_m(y)\omega'(y){\mathrm d}y\rightarrow\int_{\R}\phi\omega'{\mathrm d}y=0, $$
for some $\phi\in H(\R)$. By Fatou's Lemma,  we have
\begin{equation*}
  \int_{\R}\abs{\phi''-W''(\omega)\phi}^2{\mathrm d}y=0.
\end{equation*}
Hence $\phi=0$.

On the other hand,  we first have
\begin{align*}
\big|W''(1)\big|^2\int_{\R}\abs{\phi_m}^2{\mathrm d}y&\leq\int_{\R} \big|\phi''_m-W''(1)\phi_m\big|^2{\mathrm d}y
\\[2mm]
&\leq 2\Bigg\{\int_{\R} \big|\phi''_m-W''(\omega)\phi_m\big|^2{\mathrm d}y
+\int_{\R} \Big|\big[W''(1)-W''(\omega)\big]\phi_m\Big|^2{\mathrm d}y\Bigg\}.
\end{align*}
Thus,  by \eqref{esi},  we get that
$$
\big|W''(1)\big|^2\leq \int_{\R}\Big|\big[W''(1)-W''(\omega)\big]\phi\Big|^2{\mathrm d}y.
$$
According to the assumptions on $W$ in \eqref{eq1.4},  $W''(1)\neq 0$. The above last inequality implies that $\phi\neq0$,  which is a contradiction.
\end{proof}

\textbf{Step} $\mathbf{5}$.
We will prove that \eqref{eq3.9} holds. Multiplying (\ref{eq3.13}) by $\hat{\psi}(y)$ and integrating with respect to $y$,  and using \eqref{eqi},  we have
\begin{align*}
0=&\frac{1}{2}\int_\R \big(\widehat{\psi}^2\big)_\nu{\mathrm d}y+\int_\R\abs{\widehat{\psi}_{yy}-W''\big(\omega(y+y_0)\big)\widehat{\psi}(\nu, y)}^2{\mathrm d}y
\\[2mm]
\geq&
\frac{1}{2}\int_\R \big(\widehat{\psi}^2\big)_\nu{\mathrm d}y+c\int_\R\abs{\widehat{\psi}(\nu, y)}^2{\mathrm d}y,
\end{align*}
where the constant $c>0$. Then, according to the Gronwall inequality,  we get that
\begin{equation*}
  a(\nu)\geq a(0)e^{-2c\nu},  \ \ \ \ \forall\, \nu<0,\
\end{equation*}
where the function
$$
a(\nu):=\int_\R\abs{\widehat{\psi}(\nu, y)}^2{\mathrm d}y.
$$
This is a contradiction since (\ref{eq3.18}).
The proof of \eqref{eq3.9} is completed.
\bigskip

\subsection{The projected linear problem involving the operator $L$}
We consider the following projection problem:
\begin{equation}\label{eq3.19}
  \left\{
\begin{array}{l}
\psi_t=L[\psi]+g(t, r)-c(t)\partial_r\widehat{\omega}(t, r)\qquad \text{in}\ (s, -T]\times(0, \infty),
\\[2mm]
\psi(s, r)=0, \qquad  \forall\, r\in  (0, +\infty),
\\[2mm]
\partial_r^j\psi(t, 0)=0, \qquad  \forall\, t\in (s, -T], \ \ j=1, 3,
 \end{array}
\right.
\end{equation}
where the linear operator $L[\psi]=F'\big(z(t, r)\big)[\psi] $ is defined by \eqref{eq2.13},  the function $g\in C_{\Phi}((s,  -T)\times (0, +\infty))$
and $c(t)$ satisfies the following relation
\begin{align}
  &c(t)\int_{0}^\infty\partial_r\widehat{\omega}(t, r)\omega'\big(r-\rho(t)\big)r^{n-1}{\mathrm d}r
\nonumber\\[3mm]
&= \int_{0}^\infty \Big[\omega'''\big(r-\rho(t)\big)+\frac{n-1}{r}\omega''\big(r-\rho(t)\big)
  -W''\big(z(t, r)\big)\omega'\big(r-\rho(t)\big)\Big]
\nonumber\\[3mm]
&\qquad\qquad\times\left[-\psi_{rr}-\frac{n-1}{r}\psi_r+W''\big(z(t, r)\big)\psi\right] r^{n-1}{\mathrm d}r
\nonumber\\[3mm]
&\quad+\int_{0}^\infty\left[\partial_{rr}z(t, r)
  +\frac{n-1}{r}\partial_{r}z(t, r)-W'\big(z(t, r)\big)\right]\psi \omega'\big(r-\rho(t)\big)r^{n-1}{\mathrm d}r
\nonumber\\[3mm]
&\quad+\int_{0}^\infty \psi(t, r)\partial_t\big[\omega'\big(r-\rho(t)\big)\big]r^{n-1}{\mathrm d}r
 \,+\,
\int_{0}^\infty g(t, r)\omega'\big(r-\rho(t)\big) r^{n-1}{\mathrm d}r,
\label{eq3.20}
\end{align}
for all $t\in (s, -T]$,  where the functions $\widehat{\omega}(t, r)$ and $z(t, r)$ are defined by \eqref{eq2.3} and \eqref{dz} respectively.

If $\psi$ is a solution of (\ref{eq3.19}) and $c(t)$ satisfies (\ref{eq3.20}),  integration by parts will imply that
$\psi$ satisfies the following the orthogonality condition
\begin{equation}\label{eq3.21}
    \int_0^\infty \psi(t, r)\omega'\big(r-\rho_i(t)\big) r^{n-1}{\mathrm d}r=0,
\qquad \forall\, t\in (s, -T).
\end{equation}
For (\ref{eq3.20}),  we have the following result:
\begin{lem}\label{lem6}
Let $T>0$ big enough and $\psi,  \psi_r, \psi_{rr}$ and $g\in C_{\Phi}((s,  -T)\times (0, +\infty))$.
Then there exists $c(t)$ such that $(\ref{eq3.20})$ holds.
Furthermore the following estimates are valid
\begin{align}
  \abs{c(t)}\leq \frac{C_0}{\big[\log\abs{t}\big]^{p-1}\abs{t}^{1/2}}\left\{\frac{1}{\abs{t}^{1/4}}\left[\sum_{l=0}^2
\big\|\partial_r^l\psi\big\|_{C_{\Phi}((s, -T)\times(0, \infty))}\right]
\ +\
\|g\|_{C_{\Phi}((s, -T)\times(0, \infty))}\right\},
\label{estimateofc1}
\end{align}
and
\begin{align}
\abs{\frac{c(t)\partial_r\widehat{\omega}(t, r)}{\Phi(t, r)}}
\,\leq\,
&\frac{C_0}{\abs{t}^{1/4}}
\sum_{l=0}^2\big\|\partial_r^l\psi\big\|_{C_{\Phi}((s, -T)\times(0, \infty))}
\ +\
C_0\|g\|_{C_{\Phi}((s, -T)\times(0, \infty))},
\label{estimateofc2}
\end{align}
  for any $t\in[s,  -T]$,  where the function $\widehat{\omega}(t, r)$ is given by \eqref{eq2.3},  $p\in(n, n+1]$ and $C_0$ is a positive constant which does not depend on $s,  t,  T$,  $\psi$ and $g$.
\end{lem}

\begin{proof}
We first consider the left hand side of system $(\ref{eq3.20})$.
 By Lemma \ref{lem1} and the definition of $\widehat{\omega}(t, r)$ in \eqref{eq2.3},  we have that
\begin{equation*}
\begin{aligned}
\int_{0}^\infty\partial_r\widehat{\omega}(t, r)\omega'\big(r-\rho(t)\big)r^{n-1}{\mathrm d}r&=
\int_{\delta_0-\rho(t)}^\infty\omega'(r)\omega'(r)\big(r+\rho(t)\big)^{n-1}{\mathrm d}r
\ +\
O\left([\gamma_n(t)]^{n-2}\right)
\\[3mm]
&=\left(\rho(t)\right)^{n-1}\int_\R\big[\omega'(y)\big]^2dy
\ +\
O\left([\gamma_n(t)]^{n-2}\right),
\end{aligned}
\end{equation*}
where the functions $\rho(t)$ and $\gamma_n(t)$ are given by \eqref{drho1} and \eqref{dgamma-n} respectively.

For the first and second terms in the right hand side of \eqref{eq3.20},  we can obtain that
\begin{align*}
  I(t):=&\int_{0}^\infty\Big[-\omega'''\big(r-\rho_i(t)\big)-\frac{n-1}{r}\omega''\big(r-\rho_i(t)\big)
  +W''\big(z(t, r)\big)\omega'\big(r-\rho(t)\big)\Big]
  \\[3mm]
  &\qquad\qquad \times \left(\psi_{rr}+\frac{n-1}{r}\psi_r-W''\big(z(t, r)\big)\psi\right) r^{n-1} {\mathrm d}r
\\[3mm]
  &+\int_{0}^\infty\left[\partial_{rr}z(t, r)
  +\frac{n-1}{r}\partial_{r}z(t, r)-W'\big(z(t, r)\big)\right]\psi \omega'\big(r-\rho(t)\big)r^{n-1}{\mathrm d}r
  \\[3mm]
\leq &C\sum_{l=0}^2\big\|\partial_r^l\psi\big\|_{C_{\Phi}((s, -T)\times(0, \infty))}
   \int_{0}^\infty\Bigg[\abs{\omega'''\big(r-\rho(t)\big)
  -W''\big(z(t, r)\big)\omega'\big(r-\rho(t)\big)}
  \\[3mm]
&\qquad\qquad\qquad\qquad+\abs{\omega'\big(r-\rho(t)\big)}\abs{\partial_{rr}z(t, r)-W'\big(z(t, r)\big)}\Bigg]\Phi(t, r)r^{n-1}{\mathrm d}r
  \\[3mm]
&
 \,+\,
C\sum_{l=0}^2\big\|\partial_r^l\psi\big\|_{C_{\Phi}((s, -T)\times(0, \infty))}\int_{0}^\infty\Phi(t, r)
  \sum_{l=1}^3\abs{\omega^{(l)}\big(r-\rho(t)\big)}r^{n-2}{\mathrm d}r,
  \end{align*}
due to the facts that the functions $\psi,  \psi_r, \psi_{rr}$ belong to $C_{\Phi}((s,  -T)\times (0, +\infty))$.
By similar arguments in the proof of Lemma \ref{lem10} and the definition of $\Phi(t, r)$ in \eqref{dpsi},
 we have that
\begin{align*}
I(t)&\leq C\sum_{l=0}^2\big\|\partial_r^l\psi\big\|_{C_{\Phi}((s, -T)\times(0, \infty))} \Bigg\{ \int_{0}^\infty\Phi(t, r)
  \sum_{l=1}^3\abs{\omega^{(l)}\big(r-\rho(t)\big)}r^{n-2}{\mathrm d}r
\\[2mm]
&\qquad\qquad +\int_{0}^\infty\Phi(t, r)\abs{\omega'\big(r-\rho(t)\big)}r^{n-3}{\mathrm d}r\Bigg\}
\\[2mm]
&\leq C\frac{\big[\rho(t)\big]^{n-1}}{\abs{t}^{3/4}\big[\log\abs{t}\big]^{p-1}}
\left\{\sum_{l=0}^2\big\|\partial_r^l\psi\big\|_{C_{\Phi}((s, -T)\times(0, \infty))}\right\}.
  \end{align*}
Similarly,  by the definitions in \eqref{drho1} and \eqref{dpsi} again,  we have that
\begin{align*}
\int_{0}^\infty \psi(t, r)\partial_t\big[\omega'\big(r-\rho(t)\big)\big]r^{n-1}{\mathrm d}r
&= -\rho'(t)\int_{0}^\infty \psi(t, r)\omega''\big(r-\rho(t)\big)r^{n-1}{\mathrm d}r
\\[2mm]
  &=O\left\{\frac{\big[\rho(t)\big]^{n-1}}{\abs{t}^{5/4}}\right\}\|\psi\|_{C_{\Phi}((s, -T)\times(0, \infty))}.
\end{align*}
And,  there holds that
\begin{align*}
\abs{\int_{0}^\infty g(t, r)\omega'\big(r-\rho(t)\big)r^{n-1}{\mathrm d}r}
\leq
& C\int_{0}^\infty \Phi(t, r)\abs{\omega'\big(r-\rho(t)\big)}r^{n-1}{\mathrm d}r
\\[2mm]
\leq
& C\frac{\big[\rho(t)\big]^{n-1}}{\abs{t}^{1/2}\big[\log\abs{t}\big]^{p-1}}.
\end{align*}
According  the above estimates and \eqref{eq3.20},  we can get the validity of \eqref{estimateofc1}.

Using the following fact
\begin{equation*}
  \abs{\frac{\partial_r\widehat{\omega}(t, r)}{\Phi(t, r)}}\leq C\abs{t}^{1/2}\big[\log\abs{t}\big]^{p-1}, \qquad \text{for all} \ r>0,
\end{equation*}
where $\widehat{\omega}(t, r)$ is defined by \eqref{eq2.3} and $C$ is a constant which does not depend on $t$,  we can obtain \eqref{estimateofc2}.
\end{proof}

According to Lemma \ref{lem6},  we have
\begin{lem}\label{lem7}
Let $g\in C_{\Phi}((s,  -T)\times (0,\infty))$,  then there exists a uniform constant $T_0>0$ and a unique solution
$\psi^s$ of problem $(\ref{eq3.19})$. Moreover,   $\psi^s$ satisfies the orthogonality conditions in $(\ref{eq3.21})$ with $t\in(s, -T_0)$,  and
\begin{equation}\label{eq3.27}
  \sum_{l=0}^3\big\|\partial_r^l\psi^s\big\|_{C_{\Phi}((s, t)\times(0, \infty))}\leq C\|g\|_{C_{\Phi}((s, t)\times(0, \infty))},
\end{equation}
where $t\in(s, -T_0)$ and $C$ is a uniform constant which does not depend on $t,  s$.
\end{lem}

\begin{proof}
We will use a fixed-point argument to prove this lemma.
Let
$$
X^s:=\Big\{\psi \,:\, \sum_{l=0}^3\big\|\partial_r^l\psi\big\|_{C_{\Phi}((s, s+1)\times(0, \infty))}<+\infty\Big\},
$$
and $\mathcal{T}^s(g)$ be the solution of (\ref{eq3.5}) with $-T=s+1$.
We consider the operator $\mathcal{A}^s: X^s\rightarrow X^s$ defined by
$$
\mathcal{A}^s(\psi):=\mathcal{T}^s(g-C(\psi)),
\qquad
C(\psi)=c(t)\partial_r\widehat{\omega}(t, r),
$$
where $\widehat{\omega}(t, r)$ is given by \eqref{eq2.3} and
the function $c(t)$ satisfies (\ref{eq3.20}).
Hence $c(t)$ depends on $\psi$,  and $\psi$ satisfies the orthogonality conditions in \eqref{eq3.21}.
By Lemma \ref{lem5},  we have that
\begin{equation}\label{eq3.26}
  \sum_{l=0}^3\big\|\partial_r^l\mathcal{A}^s(\psi)\big\|_{C_{\Phi}((s, s+1)\times(0, \infty))}\leq C_1\|g-C(\psi)\|_{C_{\Phi}((s, s+1)\times(0, \infty))},
\end{equation}
for a uniform constant $C_1>0$.

Set
$$c:=(C_0+C_1)\|g\|_{C_{\Phi}((s, s+1)\times(0, \infty))}$$
and
$$
X^s_c:=\big\{\psi:\|\psi\|_{C_{\Phi}((s, s+1)\times(0, \infty))}<2c\big\},
$$
where $C_0$ and $C_1$ are given by  Lemma \ref{lem6} and (\ref{eq3.26}) respectively.
It is easy to derive the following two facts:

\smallskip
\noindent{\textbf{(1).}}
By (\ref{eq3.26}) and Lemma \ref{lem6},  we have
\begin{equation*}
  \begin{aligned}
\sum_{l=0}^3\big\|\partial^l_r\mathcal{A}^s(\psi)\big\|_{C_{\Phi}((s, s+1)\times(0, \infty))}
\leq & C_0\Big(\|C(\psi)\|_{C_{\Phi}((s, s+1)\times(0, \infty))}+\|h\|_{C_{\Phi}((s, s+1)\times(0, \infty))}\Big)
\\
\leq&\frac{CC_0}{\sqrt[4]{\abs{s+1}}}\Big(\sum_{l=0}^2\big\|\partial_r^l\psi\big\|_{C_{\Phi}((s, s+1)\times(0, \infty))}
\Big)+c,
\end{aligned}
\end{equation*}
where $C$ is given by Lemma \ref{lem6}.

\smallskip
\noindent{\textbf{(2).}}
For any $\psi_1,  \psi_2\in X^s_c$,  by (\ref{eq3.26}) and Lemma \ref{lem6} again,  we have
\begin{align*}
\sum_{l=0}^3\big\|\partial^l_r\left[\mathcal{A}^s(\psi_1)-\mathcal{A}^s(\psi_2)\right]\big\|_{C_{\Phi}((s, s+1)\times(0, \infty))}
\leq
& C_0\,\|C(\psi_1)-C(\psi_2)\|_{C_{\Phi}((s, s+1)\times(0, \infty))}
\\[2mm]
\leq & C_0\,\|C(\psi_1-\psi_2)\|_{C_{\Phi}((s, s+1)\times(0, \infty))}
\\[2mm]
\leq&\frac{CC_0}{\sqrt[4]{\abs{s+1}}}\,\|\psi_1-\psi_2\|_{C_{\Phi}((s, s+1)\times(0, \infty))},
\end{align*}
where $C_0$ is given by Lemma \ref{lem6}.

Hence,  taking $s$ large enough,  we have that the operator $\mathcal{A}^s$ is a contraction map from $X^s_c$ to itself.
According to the Banach Fixed Theorem, we know that there exists a unique $\psi^s\in X^s_c$ such that $\mathcal{A}^s(\psi^s)=\psi^s$, which is a solution to problem (\ref{eq3.19}) with $-T=s+1$.

Next we will extend the solution $\psi^s(t, r)$ in $(s, s+1)\times(0,+\infty)$ to $(s,  -T_0)\times(0,+\infty)$,
which will satisfy the orthogonality condition in (\ref{eq3.20})
and the a priori estimate in Lemma \ref{lem5}. We choose $T_0$ large enough such  that $\frac{CC_0}{\sqrt[4]{T_0}}<1$,  where $C$ is given by Lemma \ref{lem6} and $C_0$ is given by (\ref{eq3.26}). Thus the above fixed-point argument can be repeated when $s+1\leq -T_0$. Hence passing  finite steps of fixed-point  arguments, the solution $\psi^s(t, r)$ can be extended to $(s, -T_0]$. Moreover the solution $\psi^s$ satisfies (\ref{eq3.27}) and the orthogonality condition.
\end{proof}

\subsection{The solvability of a linear projected  problem}
In this section,  we devote to building the solvability of the following linear parabolic projected problem:
\begin{equation}\label{eq3.1}
\left\{
\begin{array}{l}
\psi_t=L[\psi]+g(t, r)-c(t)\partial_r\widehat{\omega}(t, r)
 \qquad\text{in}\
(-\infty, -T]\times(0, \infty),
\\[2mm]
\int_{\R}r^{n-1}\psi(t, r)\omega'\big(r-\rho(t)\big){\mathrm d}r=0, \hspace{0.75cm} \text{for all}\ t\in (-\infty, -T],
 \end{array}
\right.
\end{equation}
for a bounded function $g$,  and $T>0$ fixed sufficiently large.
In the above, the linear operator $L$ is given by \eqref{eq2.13},  the functions $\rho(t)$  and $\widehat{\omega}(t, r)$ are defined by (\ref{drho1}) and \eqref{eq2.3} respectively.
The function $c(t)$ solves the following relation:
\begin{align}
&c(t)\int_{0}^\infty\partial_r\widehat{\omega}(t, r)\omega'\big(r-\rho(t)\big)r^{n-1}{\mathrm d}r
\nonumber\\[3mm]
&=
 \int_{0}^\infty \Big[\omega'''\big(r-\rho(t)\big)+\frac{n-1}{r}\omega''\big(r-\rho(t)\big)
  -W''\big(z(t, r)\big)\omega'\big(r-\rho(t)\big)\Big]
\nonumber\\[3mm]
&\qquad\qquad \times\left[-\psi_{rr}-\frac{n-1}{r}\psi_r+W''\big(z(t, r)\big)\psi\right]r^{n-1}{\mathrm d}r
\nonumber\\[3mm]
&\quad+\int_{0}^\infty\Bigg[\partial_{rr}z(t, r)+\frac{n-1}{r}\partial_{r}z(t, r)-W'\big(z(t, r)\big)\Bigg]\psi \omega'\big(r-\rho(t)\big)r^{n-1}{\mathrm d}r
\nonumber\\[3mm]
&\quad+\int_{0}^\infty \psi(t, r)\partial_t[\omega'\big(r-\rho(t)\big)]r^{n-1}{\mathrm d}r
 \,+\,
\int_{0}^\infty g(t, r)\omega'\big(r-\rho(t)\big)r^{n-1}{\mathrm d}r,
\label{eq3.2}
\end{align}
for all $t<-T$,  where $z(t, r)$ is given by \eqref{dz}.
Indeed,  this can be solved uniquely since if $T$ is taken sufficiently large,  the coefficient $\int_{0}^\infty\partial_r\widehat{\omega}(t, r)\omega'\big(r-\rho(t)\big)r^{n-1}{\mathrm d}r$ is strictly positive.

The main result in this section is as follows:
\begin{prop}\label{prop2} There exist positive constants $T_1$ and  $C$ such that for any $t\leq -T_1$ and $g\in C_{\Phi}((-\infty, t)\times(0, \infty))$,
 problem \eqref{eq3.1}-\eqref{eq3.2} has a solution $\psi=\mathcal{A}(g)$,   which defines a linear operator of $g$ and
satisfies the following estimate
\begin{equation}\label{eq3.4}
 \sum_{l=0}^3\big\|\partial^l_r\psi\big\|_{C_{\Phi}((-\infty, t)\times(0, \infty))}
 \leq C\|g\|_{C_{\Phi}((-\infty, t)\times(0, \infty))},
\qquad \text{for all}\  t\leq-T_1,
\end{equation}
where we denote the $l$-th order derivative of $\psi(t, r)$ with respect of $r$ by $\partial^l_r\psi$.
\end{prop}

$\mathbf{Proof \ of\  Proposition \ \ref{prop2} }$: We choose a sequence $s_j\rightarrow -\infty$. Let $\psi^{s_j}$ be the solution to problem (\ref{eq3.19}) with $s=s_j$,  according to Lemma \ref{lem7}.
By (\ref{eq3.27}),  we can find that the sequence $\{\psi^{s_j}\}$
converges to $\psi$ (up to subsequence) locally uniformly in $(-\infty, -T_1)\times(0, \infty)$.
Using (\ref{eq3.7}) and Proposition \ref{prop7},  we have that $\psi$ is a solution of (\ref{eq3.1}) and satisfies (\ref{eq3.4}). The proof is completed.

\section{Solving the nonlinear projected problem} \label{sec:nl}
In this section,  we mainly solve the nonlinear problem (\ref{eq2.10})-(\ref{eq2.11}) by using the fixed-point arguments.
%
%
According to the result in Proposition \ref{prop2},  $\phi$ is a solution of (\ref{eq2.10})-(\ref{eq2.11}) if only if $\phi\in C_{\Phi}((-\infty, -T_1)\times(0, \infty)) $ is a fixed point of the  following operator
\begin{equation}\label{eq4.2}
 \mathbf{ T}(\phi):=\mathcal{A}\big(E(t, r)+N(\phi)-c(t)\partial_r\widehat{\omega}(t, r)\big),
\end{equation}
where $T_1$  and $\mathcal{A}$ are given by Proposition \ref{prop2}.

Let $T\geq T_1$,  We define a set
\begin{equation}\label{LambdaT}
  \Lambda_T:=\Big\{h\in C^1(-\infty, -T]:\ \|h\|_{\Lambda_T}<1\Big\},
\end{equation}
with the norm
\begin{equation}\label{LambdaTNorm}
  \|h\|_{\Lambda_T}:=\sup_{t\leq-T}|h(t)|+\sup_{t\leq -T}\Big(\frac{|t|}{\log|t|}|h'(t)|\Big),
\end{equation}
and also a close domain
\begin{equation}\label{eq4.9}
X_{T}:=\left\{\phi:\phi\in C_{\Phi}((-\infty, -T)\times(0, \infty))\ \ \text{and}\ \ \sum_{l=0}^2\big\|\partial^l_r\phi\big\|_{C_{\Phi}((-\infty, -T)\times(0, \infty))}
\leq \frac{2 \widehat{C}}{\log T}\right\},
\end{equation}
where the space $ C_{\Phi}((-\infty, -T)\times(0, \infty))$ is defined by \eqref{eq3.3} and $\widehat{C}$ is a fixed constant.

The main result is given by the following proposition
\begin{prop}\label{prop3}
 There exists $T_2\geq T_1$
such that for any given function $h(t)$ with each $h\in \Lambda_{T_2}$,  there is a solution $\phi(t, r, h)$ to $\phi=\mathbf{ T}(\phi)$ with respect to $\rho(t)=\gamma_n(t)+h(t)$. The solution $\phi(t, r, h)$ satisfies problem (\ref{eq2.10})-(\ref{eq2.11}).
Furthermore,  the following estimate holds
\begin{equation}\label{eq4.3}
  \sum_{l=0}^2\big\|\partial^l_r\phi(t, r, h_1)-\partial^l_r\phi(t, r, h_2)\big\|_{C_{\Phi}((-\infty, -T_1)\times(0, \infty))}
  \leq C\frac{1}{\log T_2}\|h_1-h_2\|_{\Lambda_{T_2}},
\end{equation}
where $T_1$  and $\mathbf{T}$ are given by Proposition \ref{prop2} and \eqref{eq4.2} respectively,   $C$ is a positive constant which does not depend on $h_1,  h_2$ and $T_2$.
\end{prop}

To prove the above proposition,  we first prepare two lemmas.
We note that the error term $E(t, r)$ and the nonlinear term $N(\phi)$ in \eqref{Error}-(\ref{nonlinearterm}) are all dependent of $h$,  due to the setting $\rho(t)=\gamma_n(t)+h(t)$ in \eqref{drho1}. So we
denote $E(t, r)$ and $N(\phi)$ by $E(t, r, h)$ and $N(\phi, h)$.

\begin{lem}\label{lem8}
Let $h_1, h_2\in\Lambda_{T}$ and $\phi_1, \phi_2\in X_{T}$,  then there exists a constant $C$ depending on $\widehat{C}$ such that
\begin{equation}\label{eq4.4}
\begin{aligned}
 &\big\|N(\phi_1, h_1)-N(\phi_2, h_2)\big\|_{C_{\Phi}((-\infty, -T)\times(0, \infty))}
 \\[2mm]
 &\leq \frac{C}{\log T} \Bigg\{\|h_1-h_2\|_{\Lambda_{T}}
 +\sum_{l=0}^2\big\|\partial_r^l\big(\phi_1-\phi_2\big)\big\|_{C_{\Phi}((-\infty, -T)\times(0, \infty))}\Bigg\}
\end{aligned}
\end{equation}
and
\begin{equation}\label{eq4.5}
\begin{aligned}
 \|E(t, r, h_1)-E(t, r, h_2)\|_{C_{\Phi}((-\infty, -T)\times(0, \infty))}\leq &\frac{C}{\log T}\|h_1-h_2\|_{\Lambda_{T}}.
\end{aligned}
\end{equation}
\end{lem}

\begin{proof}
By the definition of $N(\phi,  h)$ in (\ref{nonlinearterm}) and the smoothness of $W$,  we have that
\begin{align}\label{eq4.6}
&\nonumber\abs{N(\phi_1, h)-N(\phi_2, h)}
\\[2mm]\nonumber
&=\abs{\int_0^1\Big\{F'\big(z(t, r)+y\phi_1+(1-y)\phi_2\big)[\phi_1-\phi_2]-F'\big(z(t, r)\big)[\phi_1-\phi_2]\Big\}{\mathrm d}y}
\\[2mm]\nonumber
&=\abs{\int_0^1\Big\{F''\big(z(t, r)+\theta\big[y\phi_1+(1-y)\phi_2\big]\big)
\big[y\phi_1+(1-y)\phi_2,\, \phi_1-\phi_2\big]\Big\}{\mathrm d}y}
\\[2mm]
&\leq C\left(\sum_{l=0}^2\big\|\partial_r^l\big(\phi_1-\phi_2\big)\big\|_{C_{\Phi}((-\infty, T_1)\times(0, \infty))}\right)
\left(\sum_{l=0}^2\sum_{i=1}^2\big\|\partial_r^l\big(\phi_i\big)\big\|_{C_{\Phi}((-\infty, T_1)\times(0, \infty))}
\right)\Phi(t, r),
\end{align}
and
\begin{align}
\abs{N(\phi, h_1)-N(\phi, h_2)}
&
=\abs{\int_0^{1}\int_{0}^{1}F''\big(yz_1(t, r)+(1-y)z_2(t, r)+s\big)
\big[z_1(t, r)-z_2(t, r),\, \phi\big]{\mathrm d}yds}
\nonumber\\[2mm]
&\leq C\|h_1-h_2\|_{\Lambda_{T_1}}\|\phi\|_{C_{\Phi}((-\infty, T_1)\times(0, \infty))}\Phi(t, r),
\label{eq4.7}
\end{align}
where $z_i(t, r)$ is given by \eqref{dz} with $\rho(t)=\gamma_n(t)+h_i(t)$,  for $i=1, 2$.
In the above, $C$ is a positive constant independent of $h, h_1, h_2$ and $\phi, \phi_1, \phi_2$,  and the linear operator $F''(u)[v_1, v_2]$ is defined by \eqref{dFs}.
 Hence combining (\ref{eq4.6}) and (\ref{eq4.7}),  we can obtain that (\ref{eq4.4}) holds.

Next we prove that \eqref{eq4.5} also holds.
By \eqref{de1} and \eqref{de2},  similar arguments in Lemma \ref{lem10} will imply that
\begin{equation*}\begin{aligned}
\abs{E_1(t, r, h_1)-E_1(t, r, h_2)}
&\leq C\Bigg\{\abs{\rho'_1(t)\omega'\big(r-\rho_1(t)\big) \,-\, \rho'_2(t)\omega'\big(r-\rho_2(t)\big)}
\\[2mm]
&\quad\qquad\qquad+\frac{(n-1)^2(n-3)}{r^3}\abs{\omega'\big(r-\rho_1(t)\big)
\,-\,
\omega'\big(r-\rho_2(t)\big)}
\Bigg\}\overline{\chi}_{\left\{r\geq\delta_0\right\}}
\\[2mm]
&\quad
+\frac{C}{[\gamma_n(t)]^2}\Big\{\abs{h_1(t)-h_2(t)}+\abs{h'_1(t)-h'_2(t)}\Big\}\overline{\chi}_{\left\{ \frac{ \delta_0}{2}<r<\delta_0\right\}}
\\[2mm]
&\leq  C\|h_1-h_2\|_{\Lambda}\frac{\Phi(t, r)}{\log T},
\end{aligned}\end{equation*}
and
\begin{align*}
&\abs{E_2(t, r, h_1)-E_2(t, r, h_2)}
\\[2mm]
&
\leq  C\Bigg\{\abs{\rho'_1(t)\widetilde{\omega}'\big(r-\rho_1(t)\big)-\rho'_2(t)\widetilde{\omega}'\big(r-\rho_2(t)\big)}
\,+\,
\frac{1}{r^3}\sum_{l=0}^3\abs{\partial_r^l\omega\big(r-\rho_1(t)\big)
-\partial_r^l\omega\big(r-\rho_2(t)\big)}
\\[2mm]
&\qquad\qquad
\,+\, \frac{1}{r^3}\sum_{l=0}^3\abs{\partial_r^l\widetilde{\omega}\big(r-\rho_1(t)\big)
-\partial_r^l\widetilde{\omega}\big(r-\rho_2(t)\big)}
\Bigg\}\overline{\chi}_{\left\{r\geq\delta_0\right\}}
\\[2mm]
&\quad
+\frac{C}{[\gamma_n(t)]^2}\Big\{\abs{h_1(t)-h_2(t)}+\abs{h'_1(t)-h'_2(t)}\Big\}\overline{\chi}_{\left\{ \frac{\delta_0}{2}<r<\delta_0\right\}}
\\[2mm]
 &\leq  C\|h_1-h_2\|_{\Lambda}\frac{\Phi(t, r)}{\log T},
\end{align*}
where $\Phi(t, r)$ is defined by \eqref{dpsi} and $C>0$ is a uniform constant independent of $h_1, h_2$ and $t$. Using the fact that $E(t, r)=E_1(t, r)+E_2(t, r)$ and combining the above estimates,  we can get (\ref{eq4.5}).
\end{proof}

\begin{lem}
Let $h_1, h_2\in \Lambda_{T}$,  $\phi_1, \phi_2\in X_{T}$,  $c(\phi, h, t)$ satisfy equation $(\ref{eq3.20})$
with respect to $\phi$ and $\rho=\gamma_n+h$. Then
\begin{equation}\label{eq4.8}\begin{aligned}
\abs{c(\phi_1, h_1, t)-c(\phi_2, h_2, t)}
  &\leq \frac{C}{T^{3/4}\big[\log T\big]^{p-1}}\|\phi_1-\phi_2\|_{C_{\Phi}((-\infty, -T)\times(0, \infty))}
    \\[2mm]
    &\quad+\frac{C}{T^{1/2}\big[\log T\big]^{p-1}}\|h_1-h_2\|_{\Lambda_{T}},
\end{aligned}\end{equation}
where $C$ depends on $\widehat{C}$ in $(\ref{eq4.9})$,  and $p\in(n, n+1]$.
\end{lem}
\begin{proof}
Proving this lemma just need to do some similar calculations in Lemmas \ref{lem6} and \ref{lem8},  we omit it here.
\end{proof}

\textbf{Proof of  Proposition \ref{prop3}}
We consider the operator $\mathbf{T}$ defined by (\ref{eq4.2}) from the domain
$X_{T}$ in (\ref{eq4.9}) to itself. We will prove $\mathbf{T}$ is a contraction mapping. Thus by fixed-point theorem,
 the operator $\mathbf{T}$ has a unique fixed point $\phi$,  i.e. $\mathbf{T}(\phi)=\phi$.

For any $\phi_1, \phi_2\in X_{T}$,  according to Lemma \ref{lem10},  Lemma \ref{lem8} and Proposition \ref{prop2},  we find that
\begin{equation*}
\sum_{l=0}^2\big\|\partial^l_r\mathbf{T}(0)\big\|_{C_{\Phi}((-\infty, T_4)\times(0, \infty))}\leq \frac{\widehat{C}}{\log T}
\end{equation*}
and
\begin{equation*}
\sum_{l=0}^2\big\| \partial^l_r \left[\mathbf{T}(\phi_1)-\mathbf{T}(\phi_2)\right]\big\|_{C_{\Phi}((-\infty, T)\times(0, \infty))}
\leq  \frac{\widehat{C}}{\log T}
\| \phi_1-\phi_2\|_{C_{\Phi}((-\infty, T)\times(0, \infty))}.
\end{equation*}
Hence,  $\mathbf{T}$ is a contraction mapping in $X_{T}$ for any $T$ large enough.
Hence,  according to Banach fixed-point theorem,  there exists a unique $\phi\in X_{T}$ such that
$\mathbf{T}(\phi)=\phi$.

Next we will prove the estimate (\ref{eq4.3}). Choosing $h_1, h_2\in \Lambda_{T}$,  according to the above proof,  we know that there exist
$\phi_i=\phi(t, r, h_i)$,  $i=1, 2$,  are two solutions to problem \eqref{eq2.10}-\eqref{eq2.11} with $\rho=\gamma_n+h_i$ respectively.

 We note that $\phi_1-\phi_2$ does not
 satisfy the orthogonality condition (\ref{eq2.11}).
Let us consider a function $\bar{\phi}=\phi_1-\bar{\phi}_2$,  where
\begin{equation*}
  \bar{\phi}_2= \phi_2-\tilde{c}(t)\partial_r\widehat{\omega}_1(t, r),
\end{equation*}
where $\widehat{\omega}_l(t, r):=\omega\big(r-\rho_1(t)\big)\chi(\gamma_n(t)r)$ with $\rho_1(t)=\gamma_n(t)+h_1(t)$,  the cut-off function $\chi$ is defined by \eqref{dcut-off},
and $\tilde{c}(t)$ is defined by the following relation
\begin{equation*}
  \tilde{c}(t)\int_{0}^\infty \partial_r\widehat{\omega}_1(t, r)\omega'\big(r-\rho_1(t)\big)r^{n-1}{\mathrm d}r
=\int_{0}^\infty \phi_2(t, r)\omega'\big(r-\rho_1(t)\big)r^{n-1}{\mathrm d}r,
\end{equation*}
for all $t\leq-T$.
According to the proof of Lemma \ref{lem6},  the coefficient of $ \tilde{c}(t)$ in the left hand side of above equality is strictly positive,
hence the function $\tilde{c}(t)$ is well-defined.

Thus $\bar{\phi}$ satisfies the following problem
\begin{equation*}
\left\{
\begin{array}{l}
\bar{\phi}_t=F'\big(z_1(t, r)\big)[\bar{\phi}]
\,+\,
\big[E(t, r, h_1)-E(t, r, h_2)\big]
\,+\,
\big[N(\phi_1, h_1)-N(\phi_2, h_2)\big]
\\[3mm]
\quad \quad+c_2(t)\big[\partial_r\widehat{\omega}_2(t, r)-\partial_r\widehat{\omega}_1(t, r)\big]
  \,+\,   R(h_1, h_2) \, -\,   \big[c_1(t)-c_2(t)\big]\partial_r\widehat{\omega}_1(t, r)
\\[3mm]
\quad \quad \ \text{in}\ (-\infty, -T)\times(0, \infty),
\\[3mm]
 \int_{0}^\infty \bar{\phi}(t, r)\omega'\big(r-\rho_1(t)\big)r^{n-1}{\mathrm d}r=0,  \qquad  \text{for all}\ t<-T,
 \end{array}
\right.
\end{equation*}
where $\widehat{\omega}_i(t, r):=\omega(r-\rho_i(t))\chi(\gamma_n(t)r)$ with $i=1, 2$,
and the term $R(h_1, h_2)$ is defined by
\begin{align*}
R(h_1, h_2)=&\,\tilde{c}'(t)\partial_r\widehat{\omega}_1(t, r)
+\tilde{c}(t)\partial_t\big[\partial_r\widehat{\omega}_1(t, r)\big]
\\[2mm]
&+\tilde{c}(t)F'(z_1(t, r))\big[\widehat{\omega}_1(t, r)\big]
+F'\big(z_2(t, r)\big)[\phi_2]
-F'\big(z_1(t, r)\big)[\phi_2],
\end{align*}
where $z_i(t, r)$ is defined by \eqref{dz} with $\rho(t)=\gamma_n(t)+h_i(t)$ for $i=1, 2$.

Recall that the linear operator $L[\phi]=F'\big(z(t, r)\big)[\phi]$ is given by
\begin{align*}
L[\phi]:=&-\phi_{rrrr}-\frac{2(n-1)}{r}\phi_{rrr}+\Bigg[2W''\big(z(t, r)\big)-\frac{(n-3)(n-1)}{r^2}\Bigg]\phi_{rr} -\big(W''\big(z(t, r)\big)\big)^2\phi
\\[2mm]
&+\Bigg[\frac{2(n-1)W''\big(z(t, r)\big)}{r}-\frac{(3-n)(n-1)}{r^3}\Bigg]\phi_{r} +2W'''\big(z(t, r)\big)\phi_rz_r+W^{(4)}\big(z(t, r)\big)\abs{z_r}^2\phi
\\[2mm]
&-W'''\big(z(t, r)\big)W'\big(z(t, r)\big)\phi +2W'''\big(z(t, r)\big)z_{rr}\phi+2\frac{n-1}{r}W'''\big(z(t, r)\big)z_r\phi,
\end{align*}
where the approximate solution $z(t,r)$ is defined by \eqref{dz}.

Hence by Proposition \ref{prop2},  the proof of Lemma \ref{lem2},  Lemmas {\ref{lem7}} and {\ref{lem8}},   we have
\begin{align}
 \sum_{l=0}^2 \|  \partial^l_r\bar{\phi}\|_{C_{\Phi}((-\infty, -T)\times(0, \infty))}\leq &C\frac{1}{\log T}\Big\{\|h_1-h_2\|_{\Lambda_T}+\|\phi_1-\phi_2\|_{C_{\Phi}((-\infty, -T)\times(0, \infty))}\Big\}
\nonumber\\[2mm]
&+C\sup_{t\leq -T}\frac{\abs{t}^{1/2}}{\left(\log\abs{t}\right)^{1-p}}\Big(\abs{\tilde{c}(t)}+\abs{\tilde{c}'(t)}\Big).
\label{eq4.10}
\end{align}
By the orthogonality condition (\ref{eq2.11}),  we have
\begin{equation}\label{eq4.11}\begin{aligned}
\abs{\int_{0}^\infty \phi_2(t, r)\partial_r\widehat{\omega}_1(t, r)r^{n-1}{\mathrm d}r}
&=\abs{\int_{0}^\infty \phi_2(t, r)\Big[\partial_r\widehat{\omega}_2(t, r)-\partial_r\widehat{\omega}_1(t, r)\Big]r^{n-1}{\mathrm d}r}
\\[2mm]
&\leq \frac{C}{\big[\log\abs{t}\big]^{p-1}\abs{t}^{\frac{1}{2}}\log T}\|h_1-h_2\|_{\Lambda_{T}}\abs{\gamma_n(t)}^{n-1},
\end{aligned}\end{equation}
where we have used the following fact
\begin{equation*}
  \abs{\frac{\partial_r\widehat{\omega}(t, r)}{\Phi(t, r)}}\leq C\abs{t}^{1/2}\big[\log\abs{t}\big]^{p-1}, \qquad \text{for all} \ r>0,
\end{equation*}
where $\widehat{\omega}(t, r)$ is defined by \eqref{eq2.3}.

We consider
\begin{equation}\label{eq4.12}\begin{aligned}
&\abs{\frac{{\mathrm d}}{{\mathrm d}t}\int_{0}^\infty \phi_2(t, r)\partial_r\widehat{\omega}_1(t, r)r^{n-1}{\mathrm d}r}
=\abs{\frac{{\mathrm d}}{{\mathrm d}t}\int_{0}^\infty \phi_2(t, r)
\big[\partial_r\widehat{\omega}_2(t, r)-\partial_r\widehat{\omega}_1(t, r)\big]r^{n-1}{\mathrm d}r
}.
\end{aligned}\end{equation}
By the definition of $F'\big(z(t, r)\big)[\phi]$ and integration by parts,  we have
\begin{align}
&\int_{0}^\infty (\phi_2)_t\big[\partial_r\widehat{\omega}_2(t, r)-\partial_r\widehat{\omega}_1(t, r)\big]r^{n-1}{\mathrm d}r
\nonumber  \\[2mm]
&=\int_{0}^\infty \Big[(\phi_2)_t-F'\big(z_2(t, r)\big)[\phi]\Big]\Big(\partial_r\widehat{\omega}_2(t, r)-\partial_r\widehat{\omega}_1(t, r)\Big)r^{n-1}{\mathrm d}r
\nonumber\\[2mm]
&\quad+\int_0^\infty\left(-\phi_{rr}-\frac{n-1}{r}\phi_r+W''\big(z_2(t, r)\big)\right)\partial_{rr}\big[\partial_r\widehat{\omega}_2(t, r)
-\partial_r\widehat{\omega}_1(t, r)\big]r^{n-1}{\mathrm d}r
\nonumber  \\[2mm]
&\quad+(n-1)\int_0^\infty\left(-\phi_{rr}-\frac{n-1}{r}\phi_r+W''\big(z_2(t, r)\big)\right)\partial_r\big[\partial_r\widehat{\omega}_2(t, r)-\partial_r\widehat{\omega}_1(t, r)\big]r^{n-2}{\mathrm d}r
\nonumber  \\[2mm]
&\quad+\int_0^\infty W''\big(z_2(t, r)\big)\Big[\phi_{rr}+\frac{n-1}{r}\phi_r-W''\big(z_2(t, r)\big)\phi\Big]\Big(\partial_r\widehat{\omega}_2(t, r)-\partial_r\widehat{\omega}_1(t, r)\Big)r^{n-1}{\mathrm d}r
\nonumber  \\[2mm]
&\quad+\int_{0}^\infty\left(\partial_{rr}z_2+\frac{n-1}{r}\partial_rz_2-W'\big(z_2(t, r)\big)\right) W'''\big(z_2(t, r)\big)\phi\Big(\partial_r\widehat{\omega}_2(t, r)-\partial_r\widehat{\omega}_1(t, r)\Big)r^{n-1}{\mathrm d}r.
\label{eq4.13}
\end{align}
By the previous fixed-point arguments and the above equality \eqref{eq4.13},   we have
\begin{equation}\label{eq4.14}\begin{aligned}
\abs{\int_{0}^\infty r^{n-1}(\phi_2)_t\Big[\partial_r\widehat{\omega}_2(t, r) -\partial_r\widehat{\omega}_1(t, r)\Big]{\mathrm d}r}
\leq
\frac{C}{\log T}\frac{\|h_1-h_2\|_{\Lambda_{T}}}{\abs{t}^{\frac{1}{2}}\big[\log\abs{t}\big]^{p-1}}
\abs{\gamma_n(t)}^{n-1}.
\end{aligned}\end{equation}
By the above estimates (\ref{eq4.11}),  (\ref{eq4.12}),  (\ref{eq4.14}) and the definition of $\tilde{c}(t)$,  we have
\begin{equation*}
  \abs{\tilde{c}(t)}+\abs{\tilde{c}'(t)}\leq \frac{C}{\log T}\frac{\|h_1-h_2\|_{\Lambda_{T}}}{\abs{t}^{\frac{1}{2}}\big[\log\abs{t}\big]^{p-1}},
\qquad \forall\,  t<-T.
\end{equation*}
Hence
\begin{equation*}
 \sum_{l=0}^2 \big\|\partial^l_r \bar{\phi}\big\|_{C_{\Phi}((-\infty, -T)\times(0, \infty))}
 \leq  \frac{C}{\log T}\Big[
\big\|\phi_1-\phi_2\big\|_{C_{\Phi}((-\infty, -T)\times(0, \infty))}+\|h_1-h_2\|_{\Lambda_{T}}\Big],
\end{equation*}
where $C$ is uniform positive constant independent of $T$.

Eventually,  we have that
\begin{align*}
\sum_{l=0}^2\big\| \partial_r^l[\phi_1-\phi_2]\big\|_{C_{\Phi}((-\infty, -T)\times(0, \infty))}
&\leq
\sum_{l=0}^2 \big\| \partial_r^l\bar{\phi}\big\|_{C_{\Phi}((-\infty, -T)\times(0, \infty))}
+C\sup_{t\leq -T}\frac{\abs{t}^{1/2}}{\big(\log\abs{t}\big)^{p-1}}\abs{\tilde{c}(t)}
\\[2mm]
&\leq C \frac{1}{\log T}\Big[
\| \phi_1-\phi_2\|_{C_{\Phi}((-\infty, -T)\times(0, \infty))}+\|h_1-h_2\|_{\Lambda_{T}}\Big],
\end{align*}
 where we choose $T$ large enough. Thus we can obtain the estimate (\ref{eq4.3}).

\section{The reduction procedure: choosing the parameter $h$} \label{sec:tc}
As we stated in Section \ref{section2.2},  the left job is to choose suitable $h(t)$, i.e. $\rho(t)=\gamma_n(t)+h(t)$, such that the function $c(t)$ vanishes in equation (\ref{eq2.10}).

\subsection{Deriving the reduced equation involving $h(t)$}
According to (\ref{eq2.14}),  the relation $c(t)=0$ is equivalent to
\begin{align}\label{eq7.12}
0=&\int_{0}^\infty \Big[\omega'''\big(r-\rho(t)\big)+\frac{n-1}{r}\omega''\big(r-\rho(t)\big)
  -W''\big(z(t, r)\big)\omega'\big(r-\rho(t)\big)\Big]
\nonumber\\[2mm]
&\qquad\qquad \times\left(-\phi_{rr}-\frac{n-1}{r}\phi_r+W''\big(z(t, r)\big)\phi\right) r^{n-1}{\mathrm d}r
\nonumber\\[2mm]
& +\int_{0}^\infty\left[\partial_{rr}z(t, r)
  +\frac{n-1}{r}\partial_{r}z(t, r)-W'\big(z(t, r)\big)\right]\phi \omega'\big(r-\rho(t)\big)r^{n-1}{\mathrm d}r
\nonumber\\[2mm]
& +\int_{0}^\infty \phi(t, r)\partial_t\big[\omega'\big(r-\rho(t)\big)\big]r^{n-1}{\mathrm d}r
 \,+\,
\int_{0}^\infty \big(E(t, r)+N(\phi)\big)\omega'\big(r-\rho(t)\big)r^{n-1}{\mathrm d}r,
\end{align}
where the error term  $E(t, r)$ and the nonlinear term $N(\phi)$ are defined by \eqref{Error}-(\ref{nonlinearterm}),
$z(t, r)$ is given by \eqref{dz} and $\phi,  \phi_r,  \phi_{rr}\in C_{\Phi}((-\infty, -T)\times(0,+\infty))$ defined by \eqref{eq3.3}.
The tedious computations of all terms in \eqref{eq7.12} will be given in the following parts.

\subsubsection{}
We first estimate the projection of the error term $E(t, r)$. According to \eqref{de1} and \eqref{de2},  we have that
\begin{align}
\int_{0}^{+\infty}E(t, r)\omega'\big(r-\rho(t)\big)r^{n-1}{\mathrm d}r
=& \int_{0}^{\delta_0}E(t, r)\omega'\big(r-\rho(t)\big)r^{n-1}{\mathrm d}r
\nonumber\\[2mm]
&+\sum_{l=1}^5\int_{\delta_0}^{+\infty}\widetilde{E}_l(t, r)\omega'\big(r-\rho(t)\big)r^{n-1}{\mathrm d}r,
\label{integerate E}
\end{align}
where $\gamma_n(t)$ is given by \eqref{dgamma-n} with fixed small positive number $\delta_0$,   and  the functions $ \widetilde{E}_1(t, r),  \cdots,  \widetilde{E}_5(t, r)$ are defined as follows
\begin{equation}\label{dee}\begin{aligned}
\widetilde{E}_1(t, r)&:=\omega'\big(r-\rho(t)\big)\left(\rho'(t)+\frac{(n-3)(n-1)^2}{2r^3}\right),
\\[2mm]
\widetilde{E}_2(t, r)&:=-\partial_{rr}\Big(\partial_{rr}\widehat{\omega}(t, r)-W'\big(\widehat{\omega}(t, r)\big)\Big)
+W''\big(\widehat{\omega}(t, r)\big)\Big(\partial_{rr}\widehat{\omega}(t, r)
-W'\big(\widehat{\omega}(t, r)\big)\Big),
\\[2mm]
\widetilde{E}_3(t, r)&:=\frac{2(n-1)}{r}\Big[W''\big(\widehat{\omega}(t, r)\big)-W''\big(\omega\big(r-\rho(t)\big)\big)\Big]\omega'\big(r-\rho(t)\big),
\\[2mm]
\widetilde{E}_4(t, r)&:=F'\big(\widehat{\omega}(t, r)\big)\big[\widetilde{z}(t, r)\big]-\frac{(n-1)(n-3)}{r^2}\left[\frac{(n-3)}{2r}\omega'\big(r-\rho(t)\big)+\omega''\big(r-\rho(t)\big)\right],
\\[2mm]
\widetilde{E}_5(t, r)&:=F''\big(\widehat{\omega}(t, r)+\theta\widetilde{z}(t, r)\big)
\big[\widetilde{z}(t, r), \widetilde{z}(t, r)\big]-\frac{\partial\widetilde{z}(t, r)}{\partial t},
\end{aligned}\end{equation}
where the functions $\widehat{\omega}(t, r)$ and $\widetilde{z}(t, r)$ are defined by \eqref{eq2.3} and \eqref{dzt} respectively.

For the first term in the right hand side of \eqref{integerate E},  by Lemma \ref{lem10},  we have that
\begin{equation}\label{intE_1}\begin{aligned}
\abs{\int_{0}^{\delta_0}E(t, r)\partial_r\big[\omega\big(r-\rho(t)\big)\big]r^{n-1}{\mathrm d}r}
&\leq C \int_{0}^{\delta_0}\Phi(t, r)\omega'\big(r-\rho(t)\big)r^{n-1}{\mathrm d}r
\\[2mm]
&\leq C\frac{\,  \big[\rho(t)\big]^{n-1}\log\abs{t}}{\abs{t}^{1/2}\, }\exp\Bigg\{-\frac{\gamma_n(t)}{2}\Bigg\},
\end{aligned}\end{equation}
where $\gamma_n(t)$ is defined by \eqref{dgamma-n} and $C>0$ only depends on $n$
and $\alpha=\sqrt{W''(1)}$.

Next we will estimate other terms in the right hand side of \eqref{integerate E}. Using the definition of $\widehat{\omega}(t, r)$ in \eqref{eq2.3}
 and equation in \eqref{eqq},  we have that
 \begin{equation*}
   \widetilde{E}_2(t, r)=\widetilde{E}_3(t, r)=0, \qquad \text{for all}\ r>\delta_0.
 \end{equation*}

For $\widetilde{E}_1(t, r)$ in \eqref{dee},  by Lemma \ref{lem1},  we have that
\begin{align}
&\int_{\delta_0}^\infty\widetilde{E}_1(t, r)\omega'\big(r-\rho(t)\big)r^{n-1}{\mathrm d}r
\nonumber\\[2mm]
&= \int_{\delta_0}^\infty\omega'\big(r-\rho(t)\big)\left(\rho'(t)+\frac{(n-3)(n-1)^2}{2r^3}\right)\omega'\big(r-\rho(t)\big)r^{n-1}{\mathrm d}r
\nonumber\\[2mm]
&=\left(\rho'(t)+\frac{(n-3)(n-1)^2}{2\rho^3(t)}
\right)\big[\rho(t)\big]^{n-1}\int_{\R}\big[\omega'(x)\big]^2{\mathrm d}x
 \,+\,
O\left(\frac{\big[\rho(t)\big]^{n-1}\log\abs{t}}{\abs{t}^{5/4}}\right),
\label{ese1}
\end{align}
where we have used the fact that
$$
\int_\R[\omega'(y)]^2y{\mathrm d}y=0,
$$
in the last equality.

Next we estimate the terms $\widetilde{E}_4(t, r)$ given in \eqref{dee},  by the definition of $\widetilde{z}(t, r)$ in \eqref{dzt} and the same arguments in proof of Lemma \ref{lem10},  \eqref{eq2.13}, for all $r>\delta_0$,  we have that
\begin{align*}
\widetilde{E}_4(t, r)&=\frac{(n-1)(n-3)}{r^3}
\Bigg\{4\partial_{rrr}\widetilde{\omega}\big(r-\rho(t)\big)-4W''\big(\omega\big(r-\rho(t)\big)\big)\partial_r\widetilde{\omega}\big(r-\rho(t)\big)
 -\frac{(n-3)}{2}\omega'\big(r-\rho(t)\big)
 \\[3mm]
&\qquad\qquad\qquad-2(n-3)\Big[\partial_{rrr}\widetilde{\omega}\big(r-\rho(t)\big)-W''\big(\omega\big(r-\rho(t)\big)\big)\partial_{r}\widetilde{\omega}\big(r-\rho(t)\big)
\omega'\big(r-\rho(t)\big)
\\[3mm]
&\qquad\qquad\qquad\qquad
-W'''\big(\omega\big(r-\rho(t)\big)\big)\widetilde{ \omega}\big(r-\rho(t)\big)\Big]\Bigg\}
\\[3mm]
&\quad+\frac{1}{r^4}\Bigg\{\Big[12(n-4)-(n-1)(n-3)\Big]\partial_{rr}\widetilde{\omega}\big(r-\rho(t)\big)
-4(n-4)W''\big(\omega\big(r-\rho(t)\big)\big)\widetilde{\omega}\big(r-\rho(t)\big)\Bigg\}
\\[3mm]
&\quad+\rho'(t)\frac{(n-1)(n-3)}{r^2}\widetilde{\omega}'\big(r-\rho(t)\big)
 \,+\,
O\left(\frac{1}{r^5}\sum_{j=1}^ke^{-\frac{\alpha}{2}}\abs{r-\rho(t)}\right),
\end{align*}
where we have used that $\widetilde{\omega}(x)$ in \eqref{w1} and its derivatives have exponential decay.

Thus,  by same arguments in above estimate of the projection of $\widetilde{E}_2(t, r)$ and the equalities in \eqref{eq2.1},  \eqref{w1} and \eqref{eqq},
integrating by parts,  we have that
\begin{align*}
&\int_{\delta_0}^\infty\widetilde{E}_4(t, r)\omega'\big(r-\rho(t)\big)r^{n-1}{\mathrm d}r
\\[2mm]
&=4(n-3)(n-1)\int^\infty_{\frac{\gamma_n}{2}-\rho(t)}\omega'(y)
\Big[\widetilde{\omega}'''(y)-W''\big(\omega(y)\big)\widetilde{\omega}'(y)\Big]\big[y+\rho(t)\big]^{n-4}{\mathrm d}y
\\[2mm]
&\quad+\frac{(n-3)^2(n-1)}{2}\int^\infty_{\delta_0-\rho(t)}\omega'(y)
\Big[2y\omega''(y)+\omega'(y)\Big]\big[y+\rho(t)\big]^{n-4}{\mathrm d}y
 \,+\,
\int^\infty_{\delta_0-\rho}\frac{\omega'(y)\big[y+\rho(t)\big]^{n-6}}{\exp\Big\{\alpha\frac{\abs{y}}{2}\Big\}}{\mathrm d}y
 \\[2mm]
&\quad \,+\,  \int^\infty_{\delta_0-\rho(t)}\Bigg[\Big(12(n-4)-4(n-1)(n-3)\Big)\widetilde{\omega}''(y)
-4(n-4)W''\big(\omega(y)\big)\widetilde{\omega}(y)\Bigg]
\omega'(y)\big[y+\rho(t)\big]^{n-4}{\mathrm d}y
\\[2mm]
&\quad+\frac{\rho'(t)}{\big[\rho(t)\big]^2}\int^\infty_{\delta_0-\rho(t)}\widetilde{\omega}'(y)\omega'(y)dy
  \\[2mm]
&=O\left(\frac{\big[\rho(t)\big]^{n-1}}{\abs{t}^{5/4}}\right),
\end{align*}
where we have used \eqref{eq2.4},  \eqref{eq2.1},  \eqref{dwtd} and the fact that
\begin{equation*}
  \int_\R \omega'(y)\Big[2y\omega''(y)+\omega'(y)\Big]{\mathrm d}y=0.
\end{equation*}

For the term $\widetilde{E}_5(t, r)$ in \eqref{dee},  by the definitions in \eqref{dzt} and \eqref{dFs},
the properties of $W$ in \eqref{eq1.4} and  the oddness of $\widetilde{\omega}$ in \eqref{w1} respectively,  and the equalities in \eqref{eq2.1},  we get that
\begin{align*}
  &\int_{\delta_0}^\infty\widetilde{E}_5(t, r)\omega'\big(r-\rho(t)\big)r^{n-1}{\mathrm d}r
  \\[2mm]
&=\int_{\delta_0}^\infty\Bigg\{ \partial_{rr}\left[W'''\big(\omega\big(r-\rho(t)\big)\big)\big(\omega\big(r-\rho(t)\big)\big)^2\right]-W'''\big(\omega\big(r-\rho(t)\big)\big)W''\big(\omega\big(r-\rho(t)\big)\big)\big(\omega\big(r-\rho(t)\big)\big)^2
 \\[2mm]
&  \quad +2\Big[\partial_{rr}\widetilde{\omega}\big(r-\rho(t)\big)-W''\big(\omega\big(r-\rho(t)\big)\big)\widetilde{\omega}\big(r-\rho(t)\big)\Big]
W'''\big(\omega(r-\rho(t)\big)\widetilde{\omega}\big(r-\rho(t)\big)\Bigg\}\omega'\big(r-\rho(t)\big)r^{n-5}{\mathrm d}r
  \\[2mm]
& \quad  +O\left(\int^\infty_{\delta_0-\rho(t)}\frac{\omega'(y)\big[y+\rho(t)\big]^{n-6}}{\exp\Big\{\alpha\frac{\abs{y}}{2}\Big\}}{\mathrm d}y \right)
  \\[2mm]
  &=O\left( \frac{\big[\rho(t)\big]^{n-1}}{\abs{t}^{5/4}}\right).
\end{align*}

\subsubsection{ }
We next estimate other terms in the right hand side of  \eqref{eq7.12}.
By $\phi\in C_{\Phi}((-\infty, -T)\times(0, +\infty))$  and the definitions of $\Phi(t, r)$
and $\rho(t)$ in \eqref{dpsi} and \eqref{drho1} respectively,  we have that
\begin{align*}
  &\int_{0}^{\infty}\abs{N(\phi)}\omega'\big(r-\rho(t)\big)r^{n-1}{\mathrm d}r
\\[3mm]
&\leq C\int_{0}^\infty\big(\Phi(t, r)\big)^2\omega'\big(r-\rho(t)\big)r^{n-1}{\mathrm d}r
\\[3mm]
&\leq C\frac{\big[\log\abs{t}\big]^2} {\abs{t}}\int^{+\infty}_{\frac{\gamma_n(t)}{4}-\rho(t)}
\left (1+\abs{x+\rho(t)-\frac{\alpha}{3}\log\abs{t}}\right )^{-2p}\big[x+\rho(t)\big]^{n-1}\omega'(x){\mathrm d}x
\\[3mm]
&\quad +C\frac{\big[\log\abs{t}\big]^2}{\abs{t}}\big[\rho(t)\big]^{n-1}\exp\Bigg\{-\frac{\gamma_n(t)}{2}\Bigg\}.
\end{align*}
Thus,  by Lemma \ref{lem1} and direct computations,  we have that
\begin{equation}\label{en}
  \int_{0}^{\infty}\abs{N(\phi)}\omega'\big(r-\rho(t)\big)r^{n-1}{\mathrm d}r\leq C \frac{\big[\rho(t)\big]^{n-1}}{\abs{t}\big[\log\abs{t}\big]^{2(p-1)}},
\end{equation}
where $C$ only depends on $\alpha,  n$ and $\|h\|_{L^\infty}$.

We estimate the first and second term in \eqref{eq7.12}.
According to Taylor expansion and \eqref{eqq},  the same argument as the proof of Lemma \ref{lem2},  we have that
\begin{align*}
\Bigg{|}&\int_{0}^\infty\Big[-\omega'''\big(r-\rho(t)\big)
-\frac{n-1}{r}\omega''\big(r-\rho(t)\big)
+W''\big(z(t, r)\big)\omega'\big(r-\rho(t)\big)\Big]
\\[3mm]
&\qquad\qquad \times\left[\phi_{rr}+\frac{n-1}{r}\phi_{r}-W''\big(z(t, r)\big)\phi\right]
r^{n-1}{\mathrm d}r
\\[3mm]
&\  +\int_{0}^\infty\left[\partial_{rr}z(t, r)+\frac{n-1}{r}\partial_{r}z(t, r)-W'\big(z(t, r)\big)\right]W'''\big(z(t, r)\big)\phi\omega'\big(r-\rho(t)\big)r^{n-1}{\mathrm d}r\Bigg{|}
\\[3mm]
&\leq C\int_{-\rho(t)}^\infty\abs{\left[\omega^{(4)}(r)-W''\big(\omega(r)\big)\omega''(r)-W'''\big(\omega(r)\big)
[\omega'(r)]^2\right]\phi\big(t, r+\rho(t)\big)}\big[r+\rho(t)\big]^{n-1}{\mathrm d}r
\\[3mm]
&\leq C\frac{\big[\rho(t)\big]^{n-1}}{\abs{t}\big[\log\abs{t}\big]^{p-2}}.
\end{align*}
By Lemma \ref{lem1} and the definition of $\rho$ in (\ref{drho1}),  similar arguments as in \eqref{en},  we have
\begin{align*}
\abs{\rho'(t)\int_{0}^\infty \phi(t, r)\omega''\big(r-\rho(t)\big)r^{n-1}{\mathrm d}r}
&\leq \frac{C}{\abs{t}^{3/4}}
 \abs{ \int_{0}^\infty \Phi(t, r)\omega''\big(r-\rho(t)\big)r^{n-1}{\mathrm d}r}
\\[2mm]
&
  \leq C \frac{\big[\rho(t)\big]^{n-1}}{\abs{t}^{5/4}}.
\end{align*}

Combining the above estimates,   we can obtain that
(\ref{eq7.12}) is equivalent to the following ODE:
\begin{equation}\label{eq7.13}\begin{aligned}
  \rho'(t)+\frac{(n-3)(n-1)^2}{2\rho^3(t)}=Q\big(\rho(t), \rho'(t)\big),
\end{aligned}\end{equation}
for all $t<-T$, where
we recall that $\rho(t)=\gamma_n(t)+h(t)$,  and $\gamma_n(t)$ is given by \eqref{dgamma-n}.
The function $h(t)$ belongs to the set $\Lambda_T$ with  $T>T_2$,  where $T_2$ is given by Proposition \ref{prop3}
and we also recall the following two definitions in \eqref{LambdaT}-\eqref{LambdaTNorm}
$$
\Lambda_T=\left\{h(t): h\in C^1(-\infty,  -T]
\quad \text{and}\quad
\|h\|_{\Lambda_T}< 1  \right\},
$$
where
\begin{equation}\label{eq7.16}
  \|h\|_{\Lambda_T}=\sup_{t\leq-T}\abs{h(t)}+\sup_{t\leq-T}\left[\frac{\abs{t}}{\log\abs{t}}\abs{h'(t)}\right].
\end{equation}

According to the above arguments,   Proposition \ref{prop3} and Lemma \ref{lem10},  we have that
\begin{prop}\label{prop1}
Let $p\in(n, n+1]$ and $P(h(t), h'(t)):=Q\big(\rho(t), \rho'(t)\big)$ in \eqref{eq7.13}. Then for all $h,  h_1,  h_2\in\Lambda_T$,  there hold that
\begin{equation*}
  \abs{P(h, h')}\leq \frac{C(n, \alpha)}{\abs{t}\big[\log\abs{t}\big]^{2(p-1)}}
\end{equation*}
and
\begin{equation*}
  \abs{P\big({h}_1, (h_1)'\big)-P\big(h_2, (h_2)'\big)}\leq \frac{C(n, \alpha)}{\abs{t}\big[\log\abs{t}\big]^{2(p-1)}}\|h_1-h_2\|_{\Lambda_T},
\end{equation*}
where we recall that $\alpha=\sqrt{W''(1)}>0$ and $C(n, \alpha)$ is an uniform constant depending on $n$ and $\alpha$.
\end{prop}

\subsection{Solving the reduced equation involving $h$}
In the rest of this section we will study the ODE in \eqref{eq7.13}. We look for solutions of \eqref{eq7.13} of the form
$\rho(t)=\gamma_n(t)+h(t)$,  then $h(t)$ satisfies that
\begin{equation*}
  h'(t)+\gamma'_n(t)+\frac{(n-3)(n-1)^2}{2\big[\gamma_n(t)+h(t)\big]^3}=P\big(h(t), h'(t)\big)\qquad \text{in}\ (-\infty, -\widehat{T}_0],
\end{equation*}
where $\widehat{T}_0>T_2$,  where $T_2$ is given by Proposition \ref{prop3}. By the definition of $\gamma_n(t)$ in \eqref{dgamma-n},  the above equation is equivalent to
\begin{equation}\label{eq7.14}
  h'(t)+\frac{3h(t)}{4t}=\widetilde{P}\big(h(t), h'(t)\big)\qquad \text{in}\ (-\infty, -\widehat{T}_0],
\end{equation}
where
\begin{equation}\label{dpt}
  \widetilde{P}\big(h(t), h'(t)\big):=P\big(h(t), h'(t)\big)+\frac{(n-3)(n-1)^2}{2}\left[\frac{1}{\big(\gamma_n(t)\big)^3}-\frac{1}{\big(\gamma_n(t)+h(t)\big)^3}
  -\frac{3h(t)}{\big(\gamma_n(t)\big)^4}\right].
\end{equation}
We will solve equation \eqref{eq7.14} by applying the fixed-point theorem in a suitable space with $h(-\widehat{T}_0)=0$.
It is easily to check that if $h(t)$ is a solution of \eqref{eq7.14} with initial data $0$,  then it has the form
\begin{equation}\label{eq7.17}
  h(t)=-\frac{1}{{(-t)}^{\frac{3}{4}}}\int^{-T_0}_t{(-s)}^{\frac{3}{4}} \widetilde{P}\big(h(s), h'(s)\big)\mathrm{d}s,
\end{equation}
with $t\leq-\widehat{T}_0$.

Let us define two operators as following
\begin{equation*}
\mathcal{ P}(h(t)):=-\frac{1}{{(-t)}^{\frac{3}{4}}}\int^{-\widehat{T}_0}_t{(-s)}^{\frac{3}{4}} \widetilde{P}\big(h(s), h'(s)\big)\mathrm{d}s
\qquad \text{and}\qquad
\mathcal{ P}'(h(t)):=\partial_t \mathcal{ P}(h(t)).
\end{equation*}
Then using Proposition \ref{prop1} and \eqref{dpt},  we have that
\begin{equation}\label{eq7.15}
  \abs{\mathcal{ P}(0)}\leq\frac{ \widetilde{C}(n, \alpha)}{\log \widehat{T}_0}\quad \text{and}\quad
  \frac{\abs{t}}{\log\abs{t}}\abs{\mathcal{ P}'(0)}\leq\frac{ \widetilde{C}(n, \alpha)}{\log \widehat{T}_0},
\end{equation}
with $\widehat{T}_0>e^{2}$,  where $\widetilde{C}(n, \alpha)$ is a positive constant depending on $n$ and $\alpha$.
We consider the domain
\begin{equation*}
  Y:=\left\{h(t)\in C^1(-\infty, -\widehat{T}_0]\ :\ \|h(t)\|_{\Lambda_{\widehat{T}_0}}\leq \frac{2\widetilde{C}(n, \alpha)}{\log \widehat{T}_0}\right\},
\end{equation*}
where the norm $\|\cdot\|_{\Lambda_T}$ is given by \eqref{eq7.16} and $\widetilde{C}(n, \alpha)$ is a positive constant in \eqref{eq7.15}.

According to Proposition \ref{prop1} and \eqref{dpt},  we get that for all $h_1, h_2\in Y$,
\begin{equation*}
  \abs{\mathcal{ P}(h_1)-\mathcal{ P}(h_2)}\leq\frac{C(n, \alpha)}{\log \widehat{T}_0} \|h_1(t)-h_2(t)\|_{\Lambda_{\widehat{T}_0}}
\end{equation*}
and
\begin{equation*}
  \abs{\mathcal{ P}'(h_1)-\mathcal{ P}'(h_2)}\leq\frac{C(n, \alpha)}{\log \widehat{T}_0} \|h_1(t)-h_2(t)\|_{\Lambda_{\widehat{T}_0}},
\end{equation*}
where $C(n, \alpha)>0$ only depends on $n$ and $\alpha$. Thus by Banach fix point theorem,   there exists $h(t)\in Y$ such that
$\mathcal{P}(h(t))=h(t)$,  if we choose $\widehat{T}_0$ big enough. Thus we proved that the ODE in \eqref{eq7.13} is solvable.
Furthermore,  by the formula in \eqref{eq7.17},  we get that
\begin{equation*}
  \abs{h(t)}\leq\frac{C}{\log\abs{t}},  \qquad \text{as} \ t\rightarrow-\infty.
\end{equation*}



\end{document}